\documentclass[11pt]{article}
\usepackage[singlespacing]{setspace}
\usepackage{amsmath, amssymb, amscd, amsthm, amsfonts}
\usepackage{mathtools}
\usepackage[mathscr]{euscript}
\usepackage{mathrsfs}
\usepackage{enumitem}
\usepackage{graphicx}
\usepackage{float}
\usepackage{tikz}
\usepackage{tocloft}
\usepackage{psfrag}
\usepackage{caption}
\usepackage{subcaption}
\RequirePackage{doi}
\usepackage{hyperref}
\hypersetup{
colorlinks = true,
urlcolor   = blue,
citecolor  = blue,
linkcolor  = blue,
}
\usepackage{xcolor}
\usepackage{scalerel,stackengine}
\usepackage[colorinlistoftodos]{todonotes}
\usepackage{authblk}
\usepackage{fancyhdr}

\newlength\FHoffset
\fancyheadoffset{\FHoffset}

\newlength\FHleft
\newlength\FHright

\renewcommand{\headrulewidth}{1.0pt} 
\newbox\FHline
\setbox\FHline=\hbox{\hsize=\paperwidth%
  \hspace*{\FHleft}%
  \rule{\dimexpr\headwidth-\FHleft-\FHright\relax}{\headrulewidth}\hspace*{\FHright}%
}

\title{\textbf{Flexibility of Two-Dimensional Euler Flows with Integrable Vorticity}}

\author[]{\large \textbf{Elia Bru\`e}\footnote{Department of Decision Sciences, Universit\`a Bocconi, Milano, Italy. \textit{Email:}  \href{mailto:elia.brue@unibocconi.it}{elia.brue@unibocconi.it}. } \;
\textbf{\&} \;
\textbf{Maria Colombo}\footnote{EPFL, Station 8,
CH-1015 Lausanne,
Switzerland. \textit{Email:}  \href{mailto:maria.colombo@epfl.ch}{maria.colombo@epfl.ch}. } \; 
\textbf{\&} \; 
\textbf{Anuj Kumar}\footnote{Department of Mathematics, University of California Berkeley, CA 94720, USA. \textit{Email:}  \href{mailto:anujkumar@berkeley.edu}{anujkumar@berkeley.edu}. }}
\date{}

\newtheoremstyle{mystyle}
  {}
  {}
  {\itshape}
  {}
  {\bfseries}
  {.}
  { }
  {\thmname{#1}\thmnumber{ #2}\thmnote{ (#3)}}

\theoremstyle{mystyle}
\newtheorem{theorem}{Theorem}[section]
\newtheorem{proposition}[theorem]{Proposition}
\newtheorem{lemma}{Lemma}[section]

\newtheorem{conjecture}[theorem]{Conjecture}

\newtheorem{definition}{Definition}[section]

\theoremstyle{definition}
\newtheorem{remark}{Remark}[section]

\newcommand\norm[1]{\left\lVert#1\right\rVert}

\newcommand{\wt}[1]{\widetilde{#1}}
\newcommand{\ol}[1]{\overline{#1}}
\newcommand{\R}{\mathbb{R}}
\newcommand{\N}{\mathbb{N}}
\newcommand{\Z}{\mathbb{Z}}

\DeclareRobustCommand{\rchi}{{\mathpalette\irchi\relax}}
\newcommand{\irchi}[2]{\raisebox{\depth}{$#1\chi$}} 

\stackMath
\newcommand\reallywidecheck[1]{%
\savestack{\tmpbox}{\stretchto{%
  \scaleto{%
    \scalerel*[\widthof{\ensuremath{#1}}]{\kern-.6pt\bigwedge\kern-.6pt}%
    {\rule[-\textheight/2]{1ex}{\textheight}}
  }{\textheight}%
}{0.5ex}}%
\stackon[1pt]{#1}{\scalebox{-1}{\tmpbox}}%
}

\makeatletter
\newcommand\RedeclareMathOperator{%
  \@ifstar{\def\rmo@s{m}\rmo@redeclare}{\def\rmo@s{o}\rmo@redeclare}%
}
\newcommand\rmo@redeclare[2]{%
  \begingroup \escapechar\m@ne\xdef\@gtempa{{\string#1}}\endgroup
  \expandafter\@ifundefined\@gtempa
     {\@latex@error{\noexpand#1undefined}\@ehc}%
     \relax
  \expandafter\rmo@declmathop\rmo@s{#1}{#2}}
\newcommand\rmo@declmathop[3]{%
  \DeclareRobustCommand{#2}{\qopname\newmcodes@#1{#3}}%
}
\@onlypreamble\RedeclareMathOperator
\makeatother

 \RedeclareMathOperator{\div}{div}
 
\DeclareMathOperator\supp{supp}

\def\Xint#1{\mathchoice
{\XXint\displaystyle\textstyle{#1}}%
{\XXint\textstyle\scriptstyle{#1}}%
{\XXint\scriptstyle\scriptscriptstyle{#1}}%
{\XXint\scriptscriptstyle\scriptscriptstyle{#1}}%
\!\int}
\def\XXint#1#2#3{{\setbox0=\hbox{$#1{#2#3}{\int}$ }
\vcenter{\hbox{$#2#3$ }}\kern-.6\wd0}}

\def\dashint{\Xint-}

\numberwithin{equation}{section}

\newcommand{\eps}{{\varepsilon}}

\begingroup

\endgroup

\newtheorem{prob}[theorem]{Problem}

\makeatletter

\definecolor{darkbrown}{rgb}{0.35, 0.22, 0.11}

\begin{document}

\maketitle

\begin{abstract}
We propose a new convex integration scheme in fluid mechanics, and we provide an application to the two-dimensional Euler equations. We prove the flexibility and nonuniqueness of $L^\infty L^2$ weak solutions with vorticity in $L^\infty L^p$ for some $p>1$, surpassing for the first time the critical scaling of the standard convex integration technique.

To achieve this, we introduce several new ideas, including:
\begin{itemize}
    \item[(i)] A new family of building blocks built from the Lamb-Chaplygin dipole.
    
    \item[(ii)] A new method to cancel the error based on time averages and non-periodic, spatially-anisotropic perturbations.
\end{itemize}
\end{abstract}

\section{Introduction}



We investigate the homogeneous incompressible Euler equations:
\begin{equation}\label{EU}
    \begin{cases}
        \partial_t u + \div (u \otimes u) + \nabla p = 0,
        \\
        \div u = 0 \, ,
    \end{cases}\tag{EU}
\end{equation}
where $u$ represents the unknown velocity field and $p$ denotes the scalar pressure field. These equations are posed on the two-dimensional domain $\mathbb{T}^2 = [-\pi, \pi]^2$ with periodic boundary conditions.

Our focus lies on weak solutions $u \in C_t L^2_x$ with vorticity $\omega := \text{curl} \, u$ that is uniformly integrable, $\omega \in C_t L^p_x$. The main result of this paper asserts that the system \eqref{EU} is {\it flexible} within this class, for small values of $p>1$. A first example of flexibility is:

\begin{theorem}[Flexibility]
\label{thm:main}
    There exists $p>1$ such that the following holds. For any divergence-free velocity fields $u_{\rm star}, u_{\rm end}\in L^2(\mathbb{T}^2)$ with zero mean, and any $\eps>0$, there exists a weak solution $u\in C([0,1];L^2(\mathbb{T}^2))$ to \eqref{EU} with $\omega\in C([0,1];L^p(\mathbb{T}^2))$ such that $\|u(\cdot,0) - u_{\rm star}\|_{L^2} \le \eps$ and $\|u(\cdot, 1) - u_{\rm end}\|_{L^2} \le \eps$.
\end{theorem}

An immediate consequence of Theorem \ref{thm:main} is that there exists a dense set of initial conditions $u_{\text{start}}\in L^2\cap W^{1,p}$ with zero mean, such that the Cauchy problem associated with \eqref{EU} admits \textit{non-conservative} weak solutions $u\in C_tL^2_x$ with vorticity $\omega \in C_tL^p_x$. To see this, it is enough to pick $u_{\text{start}}$ with much higher kinetic energy than $u_{\text{end}}$.

Our construction is also able to produce non-uniqueness for the same set of wild initial conditions $u_{\text{start}}\in L^2\cap W^{1,p}$.

\begin{theorem}[Nonuniqueness]\label{thm:nonuniqueness}
    There exists $p>1$ and a dense subset of  divergence-free velocity fields 
    $u_{\rm {start}}\in L^2(\mathbb{T}^2)\cap W^{1,p}(\mathbb{T}^2)$ with zero mean, such that \eqref{EU} admits infinitely many non-conservative weak solutions $u\in C([0,1];L^2(\mathbb{T}^2))$ with $\omega\in C([0,1];L^p(\mathbb{T}^2))$ satisfying $u(\cdot, 0) = u_{\rm start}$.
\end{theorem}

The weak solutions to \eqref{EU} constructed in this paper are built using a {\it new convex integration scheme}. It is not surprising that a small modification of our approach is able to establish the following variant of flexibility for \eqref{EU} in the class of $C_tL^2_x$ with solutions having $C_tL_x^p$ vorticity.

\begin{theorem}[Time-Wise Compact Support]\label{thm:main2}
    There exists a non-trivial weak solution $u\in C(\R;L^2(\mathbb{T}^2))$ to \eqref{EU} with compact support in time and $\omega \in L^\infty(\R; L^p(\mathbb{T}^2))$ for some $p>1$.
\end{theorem}

As for many implementations of convex integration in fluid dynamics, a technical modification of the proof of Theorems~\ref{thm:main} and \ref{thm:main2} allows to obtain flexibility of solutions while prescribing their kinetic energy.
We expect the following statement to follow by our arguments for some $p>1$:  for every positive smooth function $e:[0,1]\to \R$ there exists a weak solution $u\in C([0,1];L^2(\mathbb{T}^2))$ with $\omega \in C([0,1];L^p(\mathbb{T}^2)$ such that 
    \begin{equation}
        e(t) = \frac{1}{2}\int_{\mathbb{T}^2} |u(x,t)|^2\, dx \, , \quad t\in [0,1]\, .
    \end{equation}
The proof of such result should follow by modifying our iterative Proposition \ref{prop:iteration} below by including closeness to the energy profile, which in turn can be guaranteed with the idea introduced in \cite{BDSV} of space-time cutoffs.
 However, given the required technical complications we prefer to not to purse this goal here.




\subsection{Context and Motivations}

A classical well-posedness result by Wolibner \cite{Wolibner33} and Hölder \cite{Holder33} ensures that \eqref{EU} is well-posed in \(C^{1,\alpha}\) for any \(\alpha > 0\). The proof of this result is based on the fact that the vorticity \(\omega = \text{curl} \, u\) of a solution to \eqref{EU} is transported by the velocity field \(u\) when the latter is regular enough. More precisely, we have the following \textit{vorticity formulation} of the Euler system:
\begin{equation}\label{EU vor}
    \begin{cases}
        \partial_t \omega + (u\cdot \nabla) \omega = 0 \, , \\
        u = \nabla^\perp \Delta^{-1} \omega \, ,
    \end{cases}
    \tag{EUvor}
\end{equation}
where the second equation, namely the {\it Biot--Savart law}, expresses the inverse of the curl operator on \(\mathbb{T}^2\).
These well-posedness results are in stark contrast to the three-dimensional case, where Elgindi \cite{ElgindiAnnals} proved the formation of finite-time singularities due to vortex stretching. In two dimensions, the borderline case $u \in C^1$ is more delicate. In this class, Bourgain and Li \cite{BL15} and later Elgindi and Masmoudi \cite{EM20} proved strong ill-posedness for the Euler equation (see also \cite{Cordoba2024}).

\bigskip

The transport structure of \eqref{EU vor} suggests that $L^p$ norms of the vorticity are formally conserved for any $p \in [1,\infty]$. For $p>1$, this property was utilized in \cite{DPM87} to establish the existence of distributional solutions starting from an initial data with vorticity in $L^p$. A similar existence result is significantly more difficult for $p=1$, and it was demonstrated by Delort \cite{Delort91} (see also \cite{EvansMuller94,Majda93,VecchiWu}), extending the existence theory to initial vorticities in $H^{-1}$ (where this latter condition ensures the finiteness of the energy) whose positive (or negative) part is absolutely continuous.

\subsubsection{Uniqueness and Yudovich Class}

The class of weak solutions with uniformly bounded vorticity holds a special significance in the well-posedness theory of the 2D-Euler equations. According to the classical result by Yudovich \cite{Yud62,Yud63} (see also the proof in \cite{Loeper06} and generalizations \cite{Vishik99,CS21}), it is stated that for any initial data $\omega_0 \in L^\infty$, there exists a unique solution $\omega\in L^\infty_tL^\infty_x$ to \eqref{EU vor} originating from $\omega_0$.
The insight behind Yudovich's uniqueness result is that a bounded vorticity yields an almost Lipschitz velocity field via the Biot-Savart law. Considering the transport structure of \eqref{EU vor}, this almost Lipschitz regularity suffices to guarantee well-posedness.

\medskip

A central question is whether Yudovich's result extends to the class of weak solutions with vorticity in $L^\infty_t L^p_x$ for $p<\infty$. In this framework, the fluid velocity $u\in L^\infty_t W^{1,p}_x$ is only Sobolev regular, and Yudovich's Gronwall type argument breaks down. However, in view of the developments on the DiPerna--Lions theory \cite{DiPernaLions,Ambrosio04} dealing with passive scalar with Sobolev velocity fields, one might expect to prove some form of well-posedness for \eqref{EU vor} in this class.

Recent developments point in the direction of non-uniqueness, although none of them has fully solved the problem yet. Vishik \cite{Vis18,Vis18a}, (see also \cite{OurLectureNotes}), has been able to demonstrate non-uniqueness within the class of $L^\infty_t L^p_x$ vorticity, with an additional degree of freedom: an external body force belonging to the integrability space $L^1_t L^p_x$. The non-uniqueness takes the form of symmetry breaking, and its construction is based on the existence of an unstable vortex.

A second attempt has been pursued by Bressan and Shen \cite{BrSh21}, based on numerical experiments which exhibit the symmetry-breaking type of non-uniqueness observed by Vishik. Their work represents a preliminary step towards a computer-assisted proof.

In the framework of point vortex systems, Grotto and Pappalettera \cite{GP22} have recently demonstrated that any configuration of $N$ initial point vortices is the singular limit of an evolution of $N+2$ point vortices. As a corollary, they managed to prove non-uniqueness for the ``symmetrized" weak vorticity formulation \eqref{EU vor weak} in the class of measure-valued vorticities.
Literature on point vortex systems is extensive and challenging to summarize; we refer the reader to the monographs \cite{MajdaBertozzi02,MarchioroPulvirenti94}.

\subsubsection{Convex Integration and Flexibility}

The first flexibility results for the $2D$ Euler equations were obtained by Scheffer \cite{Scheffer93}, who constructed nontrivial weak solutions $u\in L^2_tL^2_x$ with compact support in space and time. The existence of infinitely many dissipative weak solutions to the Euler equations was first proven by Shnirelman in \cite{Shnirelman00}, in the regularity class $L^\infty_t L^2_x$.

\bigskip

The convex integration method in fluid dynamics was pioneered by De Lellis and Székelyhidi in the context of the Onsager conjecture \cite{DeLellisSzekelyhidi09,DeLellisSzekelyhidi13}. These constructions, inspired by Muller and Sverak's work on Lipschitz differential inclusions \cite{MullerSverak03}, as well as Nash's paradoxical constructions for the isometric embedding problem \cite{Nash}, have led to a remarkable sequence of works including \cite{BDLISZ15,DSZ17,Buckmaster15}, culminating in the resolution of the flexibility part of the Onsager Conjecture by Isett \cite{Isett2018Annals}. Further developments in the study of the Onsager conjecture can be found in \cite{Choffrut13,CDS12,BDLSZ16,Isett22,NovackVicol23,GiriRadu23}.
We refer to the surveys \cite{BVsurvey19,BVsurvey21,DSsurvey17,DSsurvey19,DSsurvey22} for a more complete history of the Onsager program.

\bigskip

A significant breakthrough in convex integration was achieved by Buckmaster and Vicol \cite{BV19}, who introduced {\it intermittency} into the scheme. This innovation allows to treat the three-dimensional Navier-Stokes equations, which yields the first flexibility result for weak solutions. Since then, convex integration with intermittency has proven to be powerful and versatile, applicable to various problems 
\cite{MS18,BrueColomboDeLellis21,cheskidov2022sharp,CheskidovLuo2,CheskidovLuo3,NovackVicol23,BMNV23,Kwon2023}.
In \cite{BrueColombo23}, the first and second authors designed a convex integration scheme with intermittency to address the problem of uniqueness of the two-dimensional Euler equations \eqref{EU} with vorticity $\omega \in L^\infty_t L^p_x$, in relation to Yudovich's result. However, they could not reach the class of integrable vorticities, proving nonuniqueness in the class of weak solutions with $\omega \in L^\infty_t L^{1,\infty}_x$, where $L^{1,\infty}$ is a Lorentz space. Subsequently, Buck and Modena \cite{BuckModena24,BuckModena24II} proved nonuniqueness and flexibility in the class of weak solutions with $\omega \in L^\infty_t H^p_x$, where $H^p$ is the Hardy space with parameter $0 < p < 1$. Remarkably, their solutions are also admissible. The first nonuniqueness result with $L^p$ initial vorticity in the class of admissible solutions was established by Mengual in \cite{Mengual23}.
All these developments have highlighted the limitations of classical convex integration constructions with intermittency in two dimensions. Due to an inherent obstruction arising from the mechanism used to cancel the error and the Sobolev embedding theorem, convex integration solutions cannot achieve \( L^1 \) integrability for the vorticity. For an explanation of this obstruction, we refer the reader to \cite[Section 1.1]{BrueColombo23}.

Our main results, Theorem \ref{thm:main} and Theorem \ref{thm:main2}, represent the first convex integration constructions that overcome the inherent obstruction and achieve integrability of the vorticity beyond the \( L^1 \) space. This is due to a completely new design, which will be explained in the next sections.

\subsubsection{Energy Conservation, Vanishing Viscosity, Turbulence}

Energy-dissipating solutions to the Euler equations are crucial in the theory of turbulence, particularly in relation to the concepts of {\it anomalous dissipation} and the {\it zeroth law of turbulence}. In three dimensions, the celebrated conjecture by Onsager, now established as a theorem, states that weak solutions to the Euler equations belonging to the class \(u \in  C^{\alpha}_{x,t}\) conserve energy when \(\alpha > \frac{1}{3}\), but may dissipate energy when \(\alpha < \frac{1}{3}\). The critical threshold \(\alpha = \frac{1}{3}\) is dimensionless. Recently, \cite{GiriRadu23} constructed solutions \(u \in C^{\alpha}_{x,t}\) to the two-dimensional Euler equations \(\eqref{EU}\) that do not conserve energy for a given $\alpha < 1/3$.

\bigskip

The question of energy conservation is particularly meaningful in the context of weak solutions \(u \in L^\infty_t L^2_x\) to \eqref{EU} with uniformly integrable vorticity \(\omega \in L^\infty_t L^p_x\). A natural conjecture, generalizing Onsager's conjecture, is the following.

\begin{conjecture}[Energy Conservation]
\begin{itemize}
    \item[(i)] If $p\ge 3/2$, any weak solution \( u \in L^\infty_t L^2_x \) to \eqref{EU} with \( \omega \in L^\infty_t L^p_x \) conserves the kinetic energy.

    \item[(ii)] If \( p < \frac{3}{2} \), there exist weak solutions \( u \in L^\infty_t L^2_x \) to \eqref{EU} with \( \omega \in L^\infty_t L^p_x \) that do not conserve the energy.
\end{itemize} 
\end{conjecture}

Energy conservation for $p\ge 3/2$ has already been established; see, for instance, \cite{CCFS08}, \cite{CLL16}. To the best of our knowledge, Theorem \ref{thm:main} is the first advancement in the direction of flexibility.

In two space dimensions, vanishing viscosity solutions are known to exhibit more rigid properties compared to generic weak solutions to \eqref{EU}. Specifically, if the initial vorticity \(\omega_0 \in L^p\) of a vanishing viscosity solution is integrable in \(L^p\) for some \(p > 1\), the solution automatically conserves energy. See \cite{CLL16}, \cite{LMP21}, \cite{derosaPark24}, and \cite{CrippaSpirito15}.
Notably, the solutions constructed in Theorem \ref{thm:main} cannot be vanishing viscosity solutions. To the best of the authors' knowledge, these represent the first examples of weak solutions to \eqref{EU} with uniformly integrable vorticity \(\omega \in L^\infty_t L^p_x\) that are not vanishing viscosity solutions. 
In contrast, Yudovich solutions are always vanishing viscosity \cite{Masmudi07,ConstantinDrivasElgindi22,CiampaCrippaSpirito21}.

\begin{remark}[Energy Conservation vs Nonuniqueness]
    In the context of weak solutions \( u \in L^\infty_t L^2_x \) to \eqref{EU} with uniformly integrable vorticity \( \omega \in L^\infty_t L^p_x \), nonuniqueness is expected for every \( p < \infty \), whereas energy conservation is expected for \( p \geq \frac{3}{2} \). This highlights the distinct nature of non-uniqueness and energy non-conservation.
\end{remark}

The vorticity formulation \eqref{EU vor} for weak solutions \( u \in L^\infty_t L^2_x \) to \eqref{EU} with \( \omega \in L^\infty_t L^p_x \) makes distributional sense as soon as \(p \ge \frac{4}{3}\), since \(u \cdot \omega \in L^\infty_t L^1_x\) within this range. Solutions constructed in Theorem \ref{thm:main} and Theorem \ref{thm:main2} are not distributional solutions to \eqref{EU vor}; however, they satisfy the so-called \textit{symmetrized weak vorticity formulation}, dating back to the works of Delort and Schochet:
\begin{equation}\label{EU vor weak}
    \int_{\mathbb{T}^2} \omega(x,t)\phi(x)\, dx - \int_{\mathbb{T}^2} \omega(x,0)\phi(x)\, dx
    = \int_0^t \int_{\mathbb{T}^2 \times \mathbb{T}^2} H_\phi(x,y) \omega(x,s) \omega(y,s)\, dx\, dy
\end{equation}
for every test function \(\phi \in C^\infty(\mathbb{T}^2)\), where $H_\phi(x,y) := (\nabla \phi(x) - \nabla \phi(y)) \cdot K(x,y)$ and \(K(x,y)\) is the Biot-Savart kernel.

\bigskip

In the study of 2D turbulence, the concept of \textit{enstrophy defect} plays an important role in connection with the \textit{enstrophy cascade} \cite{Eyink01}. This concept suggests that, in certain turbulent regimes, weak solutions to \eqref{EU vor} might not satisfy the \textit{local enstrophy balance}; that is, integral quantities such as
\begin{equation}
    \int \beta(\omega(x,t)) \, dx \, ,\quad \beta \in C_c^\infty(\mathbb{R}) \, ,
\end{equation}
might not be conserved.
The local enstrophy balance is closely connected with the so-called \textit{renormalization} property: we say that \(\omega \in L^\infty_t L^p_x\) is a renormalized solution to \eqref{EU vor} if
\begin{equation}\label{EU vor ren}
    \partial_t \beta(\omega) + \div(u \beta(\omega)) = 0 \, ,
    \quad \text{for every \(\beta \in C_c^\infty(\mathbb{R})\)} \, .
\end{equation}
Notice that the notion of renormalized solution to \eqref{EU vor} is meaningful for every \(\omega \in L^1_t L^p_x\) with \(p \ge 1\). It was observed in \cite{Eyink01} and further elaborated in \cite{MMN06} that any \(\omega \in L^\infty_t L^p_x\) with \(p \ge 2\) is a renormalized solution to \eqref{EU vor}, as a consequence of the DiPerna-Lions theory \cite{DiPernaLions}. Moreover, Crippa and Spirito \cite{CrippaSpirito15} have shown that vanishing viscosity solutions are renormalized for every \(p \ge 1\). In stark contrast, our Theorem \ref{thm:main} provides the first example of non-renormalized solutions to \eqref{EU} with vorticity \(\omega \in L^\infty_t L^p_x\) for some \(p > 1\).

In view of the recent result \cite{BrueColomboKumar}, one might guess that the renormalization property holds for \(p > \frac{3}{2}\) and might fail for \(p < \frac{3}{2}\). However, this is only a speculation, and we pose it as an open question.

\begin{prob}
    Find \( p^* \ge 1 \) such that any weak solution \( u \in L^\infty_t L^2_x \) to \eqref{EU} with \( \omega \in L^\infty_t L^p_x \) is renormalized (i.e., satisfies \eqref{EU vor ren}) for \( p > p^* \), while there are non-renormalized solutions for \( p < p^* \).
\end{prob}
To the best of the author's knowledge, the most accurate estimate to date is \( 1 + \frac{1}{6500} < p^* \le  2 \).

\subsection{A New Convex Integration Scheme}


The main obstacle to achieve a two dimensional vorticity that is $L^1$ integrable in space is the homogeneous nature of the perturbations in convex integration schemes. Typically, these perturbations are periodic with a large wavelength \(\lambda \gg 1\), which makes them appear homogeneous at scales \( r \gg 1/\lambda \). From the embedding theorem of Nash \cite{Nash} and the foundational works of De Lellis and Székelyhidi \cite{DeLellisSzekelyhidi09,DeLellisSzekelyhidi13}, every convex integration scheme to date possess this characteristics.

Some form of heterogeneity of perturbation has been introduced in convex integration schemes with intermittency, starting from the work of Buckmaster and Vicol \cite{BV19}. In these schemes, the \(\lambda\)-periodic structure is maintained, but the perturbations exhibit heterogeneity at a much smaller scale, \( 1/\mu \ll 1/\lambda \), depending on the extent of the intermittency. However, at larger scales \( r \gg 1/\lambda \), the perturbation remains homogeneous, which mantains the obstruction to achieve uniform \( L^1 \) integrability for the vorticity.



\bigskip

In this paper, we overcome this hurdle by drawing inspiration from our recent work \cite{BrueColomboKumar} on the linear transport equation
\begin{equation}\label{TR}
    \partial_t \rho + u \cdot \nabla \rho = 0 \, ,
    \quad
     \, \div u = 0 \, .
\end{equation}
with an incompressible velocity field \( u \in L^\infty_t W^{1,p}_x \) and density \(\rho \in L^\infty_t L^r_x\), which lies within the framework of DiPerna–Lions theory \cite{DiPernaLions}.

In this context, the convex integration approach has been applied to produce nonuniqueness and flexibility for the Cauchy problem associated with \eqref{TR}. However, this method encounters similar obstructions as in the Euler setting and cannot achieve the sharp range of well-posedness recently obtained in our work \cite{BrueColomboKumar}. In this work, we introduce a novel linear construction that generates nonunique solutions to \eqref{TR} beyond the range attainable by convex integration. The key feature enabling this is the spatial heterogeneity of both the density and the velocity field.

Inspired by this analogy, the first key idea we introduce in convex integration is to completely change the design of the perturbations by:
\begin{itemize}
    \item eliminating the \(\lambda\) periodicity,
    \item achieving truly heterogeneous perturbations at every scale.
\end{itemize}

		
		
		
			
			
			

Given the error cancellation mechanism in the convex integration method, which relies on the low-frequency interaction between highly oscillating perturbations, it seems impractical to use perturbations with the aforementioned characteristics. The key idea to overcome this challenge is to exploit the time variable to \textit{rebuild spatial oscillations through time averages}.

\bigskip

At a qualitative level, the principal part of the perturbation will be concentrated in a single small moving region at any given time. This approach retains the intermittent structure while losing periodicity. As is standard in convex integration, the primary component of the perturbation must almost solve the Euler equations. Therefore, we introduce another key idea: a new family of building blocks constructed from the \textit{Lamb-Chaplygin dipole}. These building blocks will incorporate several new features:

\begin{itemize}
\item \textit{Variable speed:} This is useful to rebuild the error at the previous step without employing slow functions. The latter are unnatural in our scheme since we no longer have fast and slow variables.

\item \textit{Variable support size:} As the building blocks are almost Euler solutions with a constant \( L^2 \) norm, the scaling of the Euler equations forces a variable size of the support.

\end{itemize}


\subsubsection{Overview}
In this section, we present a more detailed description of the new convex integration scheme based on the qualitative features described above. As with any convex integration scheme, given a solution \( u_q \) of the Euler--Reynolds system:
\begin{equation}
    \partial_t u_q + \div(u_q \otimes u_q) + \nabla p_q = \div(R_q) \, ,
    \quad
    \div u_q = 0 \, ,
\end{equation}
our goal is to produce a new velocity field \( u_{q+1} = u_q + v_{q+1} \), which solves the Euler--Reynolds system with a smaller error \( R_{q+1} \). As a first step we decompose $R_q$ into rank-one tensors:
\begin{align}
    -\div(R_q) = \div\left(\sum_{i=1}^4 a_i(x,t) \xi_i \otimes \xi_i\right) + \nabla P^{\, d},
\end{align}
where the coefficients \( a_i \), \( i \in \{1, 2, 3, 4\} \), are roughly the same size as \( R_q \), see Lemma \ref{lemma:error dec}.
Our perturbation \( v_{q+1} \) consists of only one building block moving in one of the \( \xi_i \) directions at any given time. More precisely, it will be \(\tau_{q+1}\)-periodic in time, and each interval of the form \([k\tau_{q+1}, (k+1)\tau_{q+1}]\), $k\in \Z$, will be divided into four sub-intervals of equal length, each associated with a different direction \(\xi_i\). See Figure \ref{fig: time interval} below.

\begin{figure}
\centering
 \includegraphics[scale = 0.8]{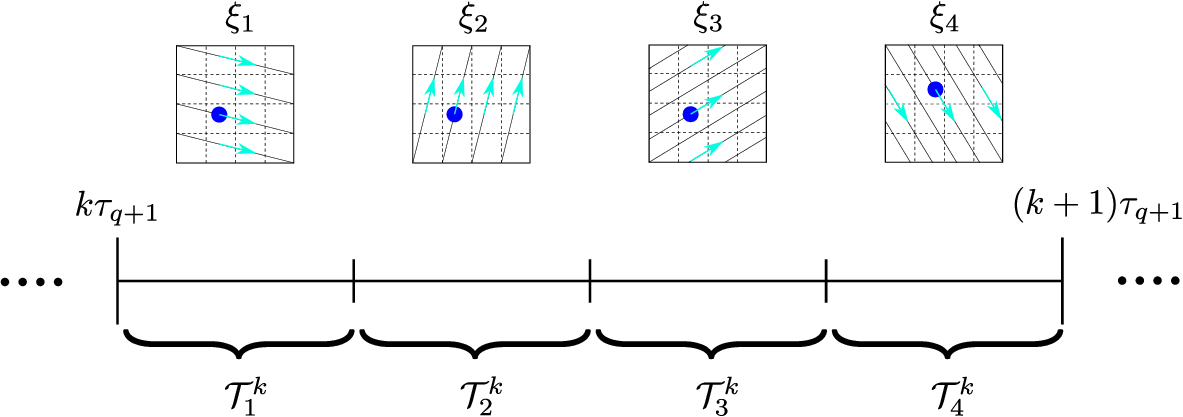}
 \caption{shows the time interval $[k \tau_q, (k+1) \tau_{q+1})$ which is further divided into four subintervals of equal length. The building block moves only in one of the direction $\xi_i$, $i \in \{1, \dots, 4\}$, on a given subinterval.}
 \label{fig: time interval}
\end{figure}

For the sake of simplicity, in rest of the outline we assume that \( R_q(x,t) = a(x,t) \xi \otimes \xi \), where \( \xi \) is a unit vector in a rationally dependent direction.

\subsubsection{The New Building Blocks}\label{sec:introbb}

Our building block is an almost exact solutions to the Euler system with source/sink term on $\mathbb{T}^2$:
\begin{equation}
\begin{cases}
\partial_t V + \div(V\otimes V) + \nabla P 
=  S \frac{d}{dt} \, (\eta(t) r(t))   +  \div(F) \, ,
\\
\div(V)  = 0\, ,
\end{cases}
\end{equation}
where $F\in C^\infty(\mathbb{T}^2; \R^{2\times 2})$ is a small symmetric tensor that will be enclosed in the new error $R_{q+1}$. The source/sink term $S \frac{d}{dt} \, (\eta(t) r(t))$ will be employed for the error cancellation through time-average, see the sketch in Section \ref{sec:ta-errcanc}. Up to leading order, the velocity field has the following structure 
\begin{align}\label{eq:V^p after W}
V(x,t) =  \eta(t) W_{r(t)}(x - x(t)) \, ,
\end{align}
where \( W_r(x) \) is the scaled core of the \textit{Lamb-Chaplygin dipole}. The center of the core, \( x(t) \), travels at variable speed in the direction \( \xi \), according to the ODE:
\begin{equation}
    x'(t) = \frac{\eta(t)}{r(x(t))}\xi \, , \qquad \text{where} \qquad r(x) = r_{q+1} \overline{a}(x).
\label{intro: size of bb}
\end{equation}
The function \( r(x)  \) is a space-dependent scale, proportional to a time average of \( a(x,t) \) over intervals of length \(\tau_{q+1} \ll 1\). The parameter \( r_{q+1} \) plays the role of the intermittency parameter in our construction. The function \( \eta(t) \) is a smooth cut-off with support of size \(\tau_{q+1}/4\); it serves to switch on and off the building block when swapping between the four directions. A fundamental feature of our construction is that \( \xi \) is chosen such that the time period of \( t \mapsto t\xi \) is large, \( \lambda_{q+1} \gg 1 \). This parameter will play the role of the frequency in our construction.

 \begin{figure}
\centering
\begin{tabular}{lc}
\begin{subfigure}{0.5\textwidth}
\centering
 \includegraphics[scale = 0.5]{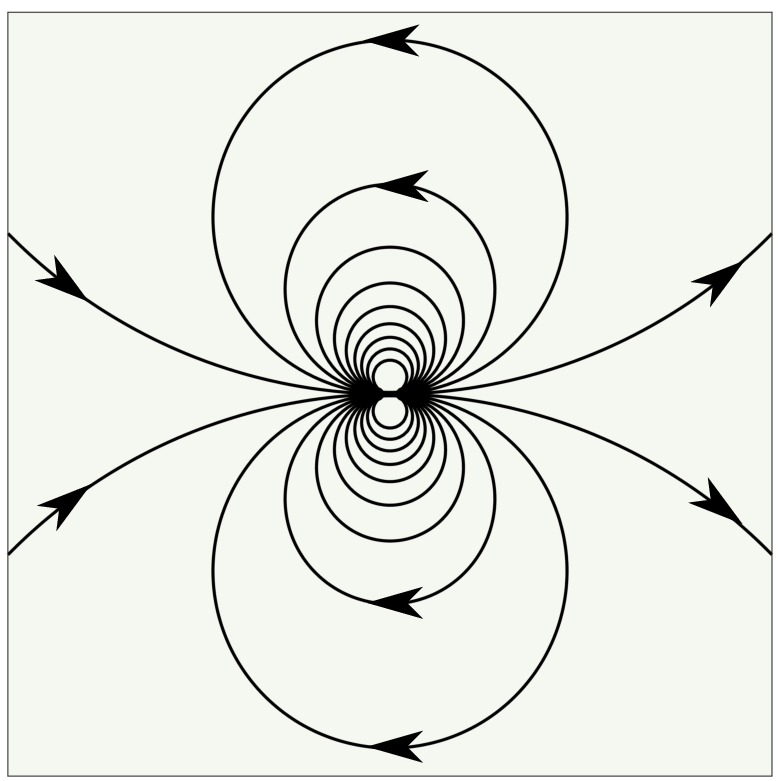}
\caption{}
\end{subfigure} &
\begin{subfigure}{0.5\textwidth}
\centering
 \includegraphics[scale = 0.5]{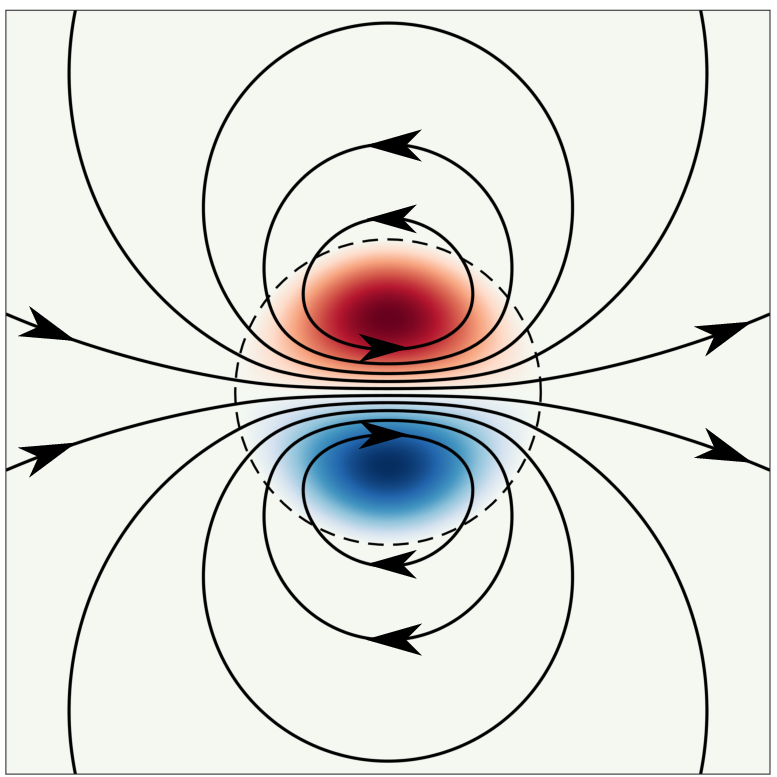}
\caption{}
\end{subfigure}
\end{tabular}
 \caption{shows the streamlines of (a) doublet flow, (b) Lamb--Chaplygin dipole. Panel (b) highlights the core region with a dashed circle, where the vorticity in the Lamb--Chaplygin dipole concentrates.}
 \label{fig1}
\end{figure}

\begin{remark}[Comparison with Intermittent Jets]
    Our new building blocks share some features with the \textit{intermittent jets} introduced in \cite{BuckColomboVicol}. The main differences are:
    \begin{itemize}
        \item Our new building blocks almost solve the Euler equations without the necessity of introducing time correctors.

        \item The support and velocity of our building blocks vary.

        \item The intermittent jets are shaped as ellipsoids in contrast with the more rounded design of ours.
    \end{itemize}  
    The third point is the most problematic in our construction. The ellipsoidal design serves to control the size of the divergence but worsens the size of the vorticity. This loss is irremediable. Our scheme is able to reach \(\omega \in L^\infty_t L^1_x\) using intermittent jets but cannot get past it.
\end{remark}

 \begin{figure}[h]
\centering
 \includegraphics[scale = 0.5]{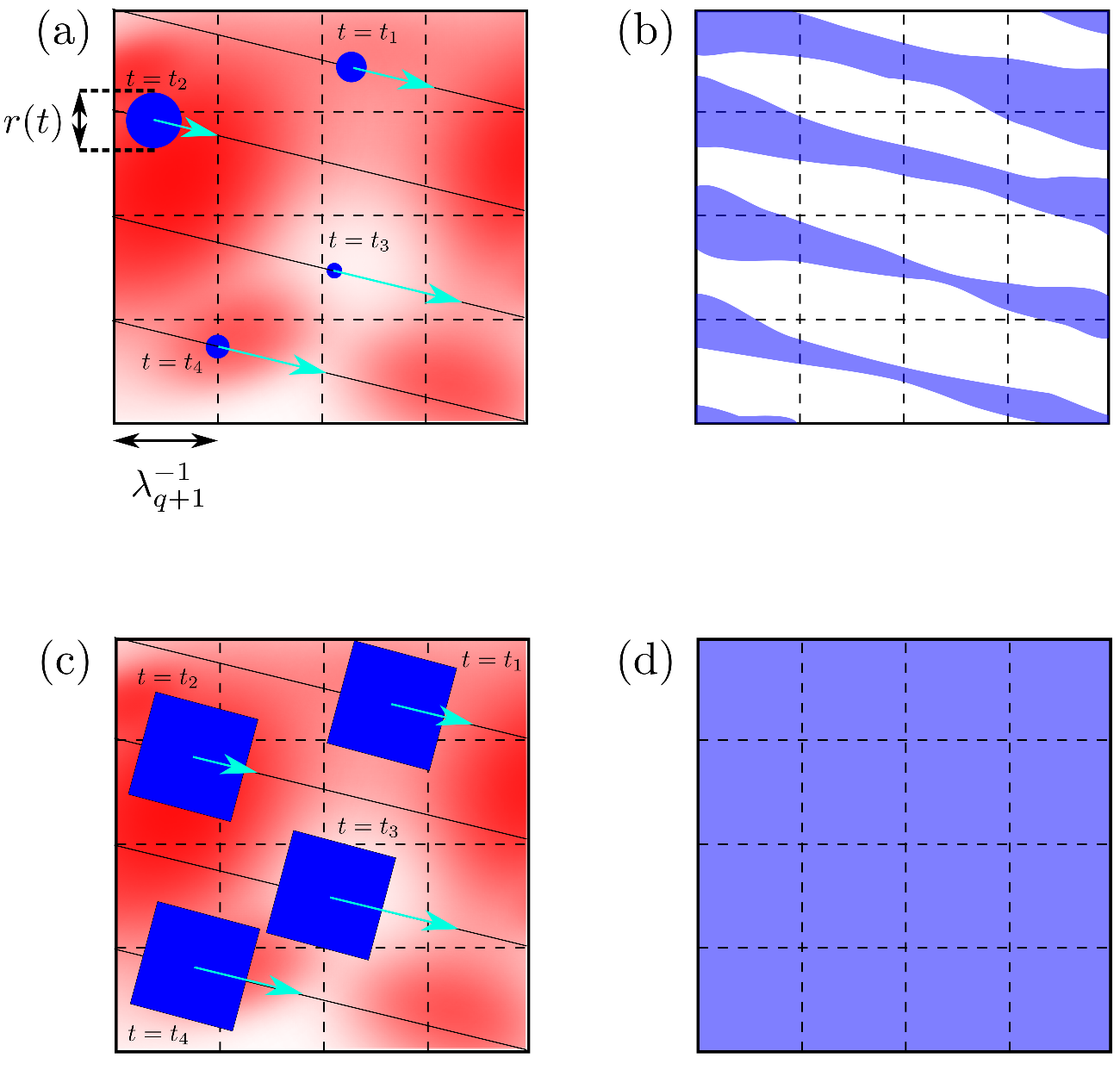}
 \caption{Panel (a) shows the building block at four different times as it traverses through the torus. The background red color is supposed to show the intensity of error. The size of the building block varies depending on the error according to the relation (\ref{intro: size of bb}). Panel (b) shows the trail of the building block, which is a $1/\lambda_{q+1}$ dense stripe and can be narrow in some regions. Panel (c) shows the auxiliary building block, whose size remains (order of $\lambda_{q+1}^{-1}$) fixed as it translates. Panel (d) shows the trail of the auxiliary building block, which is the entire $\mathbb{T}^2$.}
 \label{fig2}
\end{figure}

\subsubsection{Time Average and Error Cancellation}\label{sec:ta-errcanc}

In this section, we illustrate the key calculation: how to reconstruct the error using time averages of the source/sink term \( S \frac{d}{dt} (\eta(t) r(t)) \). For the sake of simplicity, let us assume that \( S(x,t) = U(x - x(t)) \) for some profile \( U \) independent of time. We should stress that this assumption is not satisfied in our framework; however, we will demonstrate how to compensate for this deviation by introducing an \textit{auxiliary building block} (see Section \ref{sec:auxiliary}).

On every time interval \(\mathcal{T} \subset [0,\infty]\) of length \(\tau_{q+1}/4\) where \(\eta(t)\) is supported, the time-average of the source term will be
\begin{equation}
    \begin{split}
        \dashint_{\mathcal{T}} (\eta(t) r(t))' U(x - x(t)) \, dt 
        &=
         - \dashint_{\mathcal{T}} \eta(t) r(t) \frac{d}{dt} U(x - x(t)) \, dt
        \\&
        = \dashint_{\mathcal{T}} \eta(t)^2 (\xi \cdot \nabla) U(x - x(t))\, dt 
        \\&
        = \div \left( \dashint_{\mathcal{T}} \eta(t)^2 U(x-x(t)) \otimes \xi  \, dt \right)
    \end{split}
\end{equation}
We will carefully design $r(x)$, $\eta(t)$, and $U(x)$ to have
\begin{equation}
    \dashint_{\mathcal{T}} \eta(t)^2 U(x-x(t)) \otimes \xi  \, dt
    = a(x,t) \xi \otimes \xi  + G
\end{equation}
where $G\in C^\infty(\mathbb{T}^2; \R^{2\times 2})$ is a new small error.

In other words, the time-mean of the source term is going to match exactly the error at the previous stage of the iteration. Thus, we only need to address the mean-free part of the source term. This is achieved by introducing an ad hoc time corrector \( Q_{q+1}(x,t) \); see Section \ref{sec:timecorr}.




\subsection*{Acknowledgements:}
EB would like to express gratitude for the financial support received from Bocconi University.
MC is supported by the Swiss State Secretariat for Education, Research and lnnovation (SERI) under contract number MB22.00034 through the project TENSE.

  \section{Principal Building Block}
\label{sec: principal bb}

In this section, we construct the principal building block velocity field for our convex integration scheme. The main ingredient of our construction is the {\it Lamb--Chaplygin dipole}, an exact solution of the 2D Euler equation on $\mathbb{R}^2$, which translates with constant velocity without undergoing deformation. The Lamb--Chaplygin dipole consists of two regions: a circular inner core of some fixed radius $r$ where the vorticity of opposite signs concentrate, and the complement of the core where the velocity field is irrotational. The Lamb--Chaplygin dipole can be understood as the desingularization of a potential flow known as the \textit{doublet flow}.

We also introduce a cutoff and gluing procedure that enables us to relocate this dipole from $\mathbb{R}^2$ to $\mathbb{T}^2$, while ensuring that it remains an approximate solution of the Euler equation. In addition to this relocation, we will consider a dipole on $\mathbb{T}^2$ whose core size $r$ expands or contracts as it translates.  The effect of this is a source/sink term in the approximate Euler equation, which will ultimately be used to cancel the error term in the convex integration scheme.

\bigskip



\subsection{Doublet Flow on $\mathbb{R}^2$}
\label{subsec: doublet}
A doublet flow is a divergence-free potential flow on $\mathbb{R}^2 \setminus \{0\}$. We denote $x=(x_1,x_2)\in \R^2$ points in the Euclidean plane.
In the complex notation, the potential $\wt{\Phi}$ and the stream function $\wt{\Psi}$ of the doublet flow are given by
\begin{align}
\wt{\Phi}(x_1, x_2) + i \wt{\Psi}(x_1, x_2) \coloneqq - \frac{1}{z}, \quad \text{where} \quad z = x_1 + i x_2\, .
\label{principal bb: phi psi tilde}
\end{align}
We define the velocity field $\wt{V} = (\wt{V}_1, \wt{V}_2)$ and the pressure as
\begin{align}
\wt{V} \coloneqq \nabla^\perp \wt{\Psi} = \nabla \wt{\Phi}\, , \qquad \wt{P} \coloneqq \partial_1 \wt{\Phi} - \frac{|\nabla \wt{\Phi}|^2}{2}.
\label{principal bb: V tilde}
\end{align}
where $\nabla^\perp = (-\partial_2,  \partial_1)$.
From our definition, it is clear that $\wt{V}_1(x_1, x_2) - i \wt{V}_2(x_1, x_2) = {z^{-2}}
$, that $\wt{V} $ is $-2$-homogeneous, while the pressure is not homogeneous, and that they solve
\begin{equation}\label{eq:doublet1}
    -\partial_1 \wt{V} + (\wt{V} \cdot \nabla) \wt{V} + \nabla \wt{P} = 0 \, ,
    \quad \text{on $\R^2 \setminus \{0\}$}\, .
\end{equation}

%
Since \begin{align}
\wt{V}_1(x_1,x_2)  x = -\nabla^\perp \left(\frac{x_1 x_2}{x_1^2 + x_2^2}\right) \qquad \text{and} \qquad \wt{V}_2(x_1,x_2)  x = - \nabla^\perp \left(\frac{x_2^2}{x_1^2 + x_2^2}\right)\, ,
\end{align}
we have
\begin{align}
\div (\wt{V} \otimes x) = 0.
\label{principal bb: div tensor V x} \qquad \text{for $x \in \mathbb{R}^2 \setminus \{0\}$.}
\end{align}

\subsection{Desingularization of the Doublet Flow: The Lamb--Chaplygin Dipole}

To describe the velocity field in this section, we use polar coordinates $\rho$ and $\theta$, where $x_1 = \rho \cos\theta$ and $x_2 = \rho \sin \theta$. The Lamb--Chaplygin dipole is a desingularization of the doublet flow \cite{meleshko1994chaplygin, lamb1924hydrodynamics}. It consists of two regions: 
\begin{itemize}
    \item[(i)] $\rho \leq 1$ the desingularized region, where the vorticity is nonzero,

    \item[(ii)] $\rho \geq 1$, where the flow is potential given in section \ref{subsec: doublet}.
\end{itemize}
We define the stream function $\ol{\Psi}$ for this flow as
\begin{align}
\ol{\Psi} \coloneqq 
\begin{cases}
\rho \sin \theta - \frac{2 J_1(b\rho)}{b J_0(b)} \sin \theta \qquad & \text{when} \quad \rho \leq 1, \\
\frac{\sin \theta}{\rho} \qquad & \text{when} \quad \rho \geq 1. 
\end{cases}
\label{principal bb: overbar Psi}
\end{align}
Here, $J_0$ and $J_1$ are the Bessel functions of first kind of zero and first order respectively. The number $b \approx 3.831705970\dots$ is the first non-trivial zero of $J_1$. For $\rho \geq 1$
we have
$\ol{\Psi} = \wt{\Psi}.$
Using this stream function, we define the velocity field $\ol{V}= (\ol{V}_1, \ol{V}_2) : \mathbb{R}^2 \to \mathbb{R}^2$ and the pressure as
\begin{align}
\ol{V} \coloneqq \nabla^\perp \ol{\Psi}, \qquad \ol{P} \coloneqq \ol{V}_1 - \frac{\ol{V}_1^2 + \ol{V}_2^2}{2} - 1_{|x| \leq 1} \frac{b^2}{2} (\ol{\Psi} - x_2)^2  \quad \text{for } x\in\mathbb R^2 \,
.
\label{principal bb: overbar velocity}
\end{align}
The velocity field defined in such a way belongs to $C^{1, 1}$, i.e., the derivative is Lipschitz, and it solves the equation \eqref{eq:doublet1}. Indeed, the Lamb--Chaplygin dipole has the property that $\omega = b^2 (\ol{\Psi} - x_2)$ \cite{meleshko1994chaplygin, lamb1924hydrodynamics} for $\rho \leq 1$, and $\omega = 0$ for $\rho > 1$. From the expression of $\ol{\Psi}$ in (\ref{principal bb: overbar Psi}), we then see $\omega \in C^{0, 1}$. Because we also have $-\Delta \ol{\Psi} = \omega$ and we have only one mode in $\theta$, hence this equation is a second order ODE in the radial variable. This means $\Psi \in C^{2, 1}$, which then gives the required regularity for the velocity field.\\


Finally, we define the rescaled version of $\ol{V}$ and $\ol{P}$ for a given $r > 0$ as
\begin{align}
\ol{V}_{r} \coloneqq \frac{1}{r}\ol{V}\left(\frac{x}{r}\right), \quad
\ol{P}_{r} \coloneqq \frac{1}{r^2}\ol{P}\left(\frac{x}{r}\right).
\label{principal bb: scaled lamb chaplygin}
\end{align} 
By scaling, they solve
\begin{align}- \frac{1}{r} \partial_1 \ol{V}_{r} + (\ol{V}_{r} \cdot \nabla) \ol{V}_{r} + \nabla \ol{P}_{r} = 0\, .
\label{principal bb: scaled lamb chaplygin eqn}
\end{align}
and by \eqref{principal bb: V tilde} and the $-1$-homogeneity of $\wt{\Phi}$ we have
\begin{align}
\ol{V}_r(x) =  \frac{1}{r}\nabla^\perp \ol{\Psi}\left(\frac{x}{r}\right) =\frac{1}{r}\nabla^\perp \wt{\Psi}\left(\frac{x}{r}\right) = \frac{1}{r} \nabla \wt{\Phi}\left(\frac{x}{r}\right) = r \nabla \wt{\Phi}(x)
\quad \text{when } |x| > r.
\label{principal bb: lc outside rad r}
\end{align}
\begin{align}
\ol{P}_r(x) = \frac{1}{r^2}\ol{P}\left(\frac{x}{r}\right) =  {\partial_1 \wt \Phi}  + r^2\frac{|\nabla \wt \Phi|^2}{2}, \quad \text{when } |x| > r.
\label{principal bb: pressure outside rad r}
\end{align}
Since $\ol{V}$ is $C^{1,1}$ and $\ol{V}_r$ coincides with $r \nabla \wt{\Phi}$ outside $B_r$,  for $n=0,1,2$ we have
\begin{align} 
|\nabla^n \wt{\Phi}| \leq C\frac{1}{|x|^{n+1}}, \quad 
|\nabla^n \ol{V}_r| \leq C \frac{r}{\max\{r,|x|\}^{n+2}} \quad \text{for } x\in \mathbb R.
\label{principal bb: estimate on Vbar}
\end{align}

\subsection{Decomposition of the Lamb--Chaplygin Vortex}
\label{subsec:Lamb-Chaplygin}

Fix $\alpha\in (0,1)$. We consider a smooth cutoff function $\rchi : [0, \infty) \to \mathbb{R}$ such that $\rchi(x)=0$ for $x\leq 1 $, and $\rchi(x) = 1$ for $x\geq 2$. We define a rescaled verision $\rchi_\alpha : [0, \infty) \to \mathbb{R}$ of this cutoff function as
\begin{align}\label{def:cutoff}
\rchi_\alpha(x):= \rchi\Big(\frac{|x|}{r^\alpha}\Big) \, .
\end{align}
Next, we define a pressure field $\Pi_r : \mathbb{R}^2 \to \mathbb{R}$: 
\begin{align}
\Pi_r(x) \coloneqq 
r \chi_\alpha(x) \wt{\Phi}(x)\, ,
\label{principal bb: req pressure}
\end{align}
and a velocity field $W_r : \mathbb{R}^2 \to \mathbb{R}^2$:
\begin{align}
W_r(x) \coloneqq 
\ol{V}_r(x) - \nabla \Pi_r(x) =
\ol{V}_r(x) - r\nabla (\rchi_\alpha(x) \, \wt{\Phi}(x)).
\label{principal bb: decomposition}
\end{align}
Notice that the velocity field $W_r$ is not divergence-free. Indeed,
\begin{align}\label{divWr}
\div W_r = - r \Delta (\rchi_\alpha \wt{\Phi}) = r\Delta \rchi_\alpha \wt{\Phi} + 2 r \nabla \rchi_\alpha \cdot \nabla \wt{\Phi} \, .
\end{align}

\begin{lemma}
\label{lemma: properties Pi Wr}
Let $\Pi_r$ and $W_r$ be defined as in (\ref{principal bb: req pressure}) and (\ref{principal bb: decomposition}), respectively. Then the following statements hold.
\begin{enumerate}[label = (\roman*)]
\item $\supp \Pi_r \subseteq \mathbb{R}^2 \setminus B_{r^\alpha}(0)$, and $\supp W_r \subseteq \ol{B}_{2 r^\alpha}(0)$,

\item $\norm{W_r}_{L^1} \leq C(\alpha) r |\log r|$, and $\norm{W_r}_{L^p} \leq C(p) r^{\frac{2}{p} - 1}$ for every $p \in (1, \infty]$,
\item 
$
\norm{D W_r}_{L^p} +r\norm{D^2 W_r}_{L^p} \leq C(p) r^{\frac{2}{p} - 2}$, and  $\norm{\div W_r}_{L^p} \leq C(p) \, r^{1 - \alpha} \, r^{\alpha \left(\frac{2}{p} - 2\right)}$ for every $p\in [1,\infty]$.
\item 
\begin{align}
\frac{1}{r}\int_{\mathbb{R}^2} W_r = 
\begin{pmatrix}
2 \pi \\
0
\end{pmatrix}.
\label{principal bb: main Wr integral}
\end{align}
\end{enumerate}
\end{lemma}
\begin{proof}
The item (i) follows from the definitions (\ref{principal bb: req pressure}) and (\ref{principal bb: decomposition}) and noticing
(\ref{principal bb: lc outside rad r}).
From (\ref{principal bb: estimate on Vbar}) and \eqref{def:cutoff}, we obtain the following estimates
\begin{align}
|\Pi_r| \leq C \frac{r}{|x|}\, ,
\qquad
|\nabla \Pi_r| \leq C \frac{r}{r^\alpha} \frac{1}{|x|} + C \frac{r}{|x|^2}\, , \qquad 
|\nabla^2 \Pi_r| \leq C \frac{r}{r^{2 \alpha}} \frac{1}{|x|} + C \frac{r}{|x|^3}\, ,
\label{principal bb: estimate Pi nn}
\end{align}
for every $x\in \R^2$.
As $W_r \equiv \ol{V}_r$ when $x \in B_{r^\alpha}(0)$, from (\ref{principal bb: estimate on Vbar}), we conclude that
\begin{align}
|W_r| \leq C \frac{1}{r} \quad \text{when} \quad |x| \leq r, \qquad 
 |W_r| \leq C \frac{r}{|x|^2} \quad \text{when} \quad r < |x| \leq 2r^\alpha \, .
 \label{principal bb: est 1 Wr}
\end{align}
Combining (i) with (\ref{principal bb: est 1 Wr}), (ii) easily follows.

\bigskip

By (\ref{principal bb: estimate on Vbar}) and (\ref{principal bb: decomposition}), we note that $|D W_r| + r |D^2 W_r| \leq C r^{-2}$ when $ |x| \leq r$, and $|D W_r|\leq r |x|^{-3}$, $|D^2 W_r|  \leq r |x|^{-4}$ when $r < |x| \le 2r^\alpha$. 
Using this information along with the fact that $\supp W_r \subseteq \ol{B}_{2 r^\alpha}(0)$ and $ 0 < \alpha < 1$, the required estimate on the $L^p$ norm of $D W_r$ follows. Next, the $\div W_r$ is nonzero only in the annulus $B_{2 r^\alpha}(0) \setminus B_{r^\alpha}(0)$ and that $|\div W_r| \leq C r^{1 - 3 \alpha}$, which completes the proof of (iii).



\bigskip

Finally, we compute the space integral of $\frac{1}{r}W_r$. From equation \eqref{divWr} and integrating by parts twice, we obtain
\begin{align}
\frac{1}{r}\int_{\mathbb{R}^2} W_r \, {\rm d} x & = - \frac{1}{r}\int_{\mathbb{R}^2} \div W_r
\begin{pmatrix}
    x_1 \\
    x_2
\end{pmatrix} \, {\rm d} x  \nonumber \\
& = \int_{\mathbb{R}^2} \left(\Delta \rchi_\alpha \wt{\Phi} + 2 \nabla \rchi_\alpha \cdot \nabla \wt{\Phi}\right)
\begin{pmatrix}
    x_1 \\
    x_2
\end{pmatrix} \, {\rm d} x
\\ \nonumber
& = -\int_{\mathbb{R}^2}   \wt{\Phi}
\nabla \rchi_\alpha 
{\rm d} x \; + \;
 \int_{\mathbb{R}^2} \nabla \rchi_\alpha \cdot \nabla \wt{\Phi} 
\begin{pmatrix}
    x_1 \\
    x_2
\end{pmatrix}  \, {\rm d} x  \nonumber \\
& = \int_{\mathbb{R}^2}   
\frac{\rchi_\alpha^\prime}{\rho} 
\begin{pmatrix}
\cos^2 \theta \\
\cos \theta \, \sin \theta
\end{pmatrix}
\, \rho \, {\rm d} \rho \, {\rm d} \theta \; + \;
\int_{\mathbb{R}^2}   
\frac{\rchi_\alpha^\prime}{\rho} 
\begin{pmatrix}
\cos^2 \theta \\
\cos \theta \, \sin \theta
\end{pmatrix}
 \, \rho \, {\rm d} \rho \, {\rm d} \theta \nonumber \\
& = 
\begin{pmatrix}
2 \pi \\
0
\end{pmatrix}.
\label{principal bb: new calc avg W r}
\end{align}
In the fourth line, we used the polar coordinate system $(\rho, \theta)$ along with the identities
\begin{align}
\nabla \rchi_\alpha = \begin{pmatrix}
\cos \theta \\
\sin \theta
\end{pmatrix}\, \rchi^\prime_\alpha, 
\qquad
\wt{\Phi} = - \frac{\cos \theta}{\rho}, \qquad \nabla \rchi_{\alpha} \cdot \nabla \wt{\Phi} = \partial_{\rho} \rchi_{\alpha} \partial_\rho \wt{\Phi} 
= \frac{\rchi_\alpha^\prime}{\rho^2} \cos \theta.
\end{align}
\end{proof}

\subsection{Compactly Supported Approximate Solution on $\mathbb R^2$}

In this section, we demonstrate that the compactly supported velocity field $W_r$ constructed in Section~\ref{subsec:Lamb-Chaplygin} (when translated with speed $r^{-1}$ in the $x_1$ direction) is an approximate solution to the momentum part of the Euler equation.

In the forthcoming sections, we will work with building blocks traveling with nonconstant speed, which will be achieved by varying the radius $r$ in time. Thus, the dependence of $W_r$ on the scale parameter $r>0$ will play a central role. A fundamental identity is given by \eqref{principal bb: eqn partial r W_r}, which involves the derivative with respect to $r>0$.

\begin{proposition}[Constant-Speed Building Block, Principal Part]
\label{prop:bb constant spead}
Fix $\alpha\in (0,1)$.
Let the velocity field $W_r$ be defined in (\ref{principal bb: decomposition}). Then there exist pressure fields $P_1, P_2 \in C^{1, 1}(\mathbb{R}^2; \mathbb{R}^2)$ and tensors $F_1, F_2 \in C^{1, 1}(\mathbb{R}^2; \mathbb{R}^{2\times 2})$, such that
\begin{align}
 - \frac{1}{r} \partial_1 W_r + \div (W_r \otimes W_r) + \nabla P_1 = \div(F_1), \label{principal bb: eqn W_r} 
 \\
 \partial_r W_r = \frac{1}{r} W_r + \nabla P_2 + \div(F_2). \qquad \quad
 \label{principal bb: eqn partial r W_r} 
\end{align}
Moreover, the following properties hold:
\begin{enumerate}[label = (\roman*)]
\item 
$\supp P_1, \, \supp P_2, \, \supp F_1, \, \supp F_2 \subseteq B_{2 r^\alpha}(0)$,

\item $\norm{F_1}_{L^{p}} \le C(p) r^{2 - 2 \alpha} \, r^{\alpha \left(\frac{2 }{p} - 2 \right)}$ and $\norm{F_2}_{L^{p}} \leq C(p) r^{\frac{2}{p} - 1}$ for every $p\in (1,\infty)$.

\item $\norm{\nabla P_2}_{L^\infty} + \norm{\div(F_2)}_{L^\infty}  +r\norm{\nabla^2 P_2}_{L^\infty}+r\norm{D\div(F_2)}_{L^\infty}\le C r^{-2}$.
\end{enumerate}
\end{proposition}

\begin{remark}
From (\ref{principal bb: eqn W_r}), it is clear that $W_r(x - t r^{-1} e_1)$ is an approximate solution to the momentum part of the Euler equation.
\end{remark}

\begin{remark}[Space-Time Smoothing of $W_r$]
\label{rmk:smoothing}
    The velocity field $W_r$, pressure fields $P_1, P_2$, and tensors $F_1$, $F_2$ are almost $C^2$ but not smooth. However, we can smooth them out according to
    \begin{equation}
        W_r \ast \rho_\ell \, ,
        \quad 
        P_1\ast \rho_\ell \, ,
        \quad
        P_2 \ast \rho_\ell \, ,
        \quad
        F_2 \ast \rho_\ell \, ,
    \end{equation}
    where $\rho_\ell$ is a smooth convolution kernel supported at scale $\ell$. We then replace $F_1$ with
    \begin{equation}
        F_1 \ast \rho_\ell + (W_r \otimes W_r)\ast \rho_\ell - (W_r\ast \rho_\ell) \otimes (W_r\ast \rho_\ell) \, ,
    \end{equation}
    which is smooth as well.
    If $\ell$ is chosen small enough, then the regularized velocity field, pressure fields, and error tensors will satisfy all the properties in Proposition \ref{prop:bb constant spead}.
\end{remark}

\begin{proof}[Proof of Proposition \ref{prop:bb constant spead}]
We begin by deriving equation (\ref{principal bb: eqn W_r}) and proving the relevant properties of $P_1$ and $F_1$. 
Using (\ref{principal bb: scaled lamb chaplygin eqn}), (\ref{principal bb: req pressure}) and (\ref{principal bb: decomposition}), we discover that $W_r$ satisfies
\begin{equation}
\begin{split}
- \frac{1}{r} \partial_1 W_r - \frac{1}{r} \partial_1 \nabla \Pi_r + \div(W_r \otimes W_r) + & \div(W_r \otimes \nabla \Pi_r + \nabla \Pi_r \otimes W_r) 
\\&+ \div(\nabla \Pi_r \otimes \nabla \Pi_r) + \nabla \ol{P}_r = 0 \, .
\end{split}
\label{principal bb: initial eqn Wr}
\end{equation}
Next, we have the following identity:
\begin{align}
\div (\nabla \Pi_r \otimes \nabla \Pi_r) = \Delta \Pi_r \, \nabla \Pi_r +  \nabla \frac{|\nabla \Pi_r|^2}{2} \, .
\label{principal bb: a simple identity}
\end{align}
Using the definition of the pressure $\Pi_r$ from (\ref{principal bb: req pressure}), we get
\begin{align}
\Delta \Pi_r \, \nabla \Pi_r = r^2\left(\Delta \rchi_\alpha \, \wt{\Phi} + 2 \nabla \rchi_\alpha \cdot \nabla \wt{\Phi} \right) \left(\nabla \rchi_\alpha \wt{\Phi} + \rchi_\alpha \nabla \wt{\Phi}\right).
\end{align}
From the first term in the parentheses on the right-hand side and the definition of the cutoff function $\rchi_\alpha$, we see that $\supp \Delta \Pi_r \, \nabla \Pi_r \subseteq B_{2 r^\alpha}(0) \setminus B_{r^\alpha}(0)$. 
From this, we note that the integral of $\Delta \Pi_r \, \nabla \Pi_r$ is zero (on $\mathbb{R}^2$ or equivalently on $B_{2r^\alpha}$) by integration by parts:
\begin{align}
\label{eq:decay1}
\int_
{\mathbb{R}^2} \Delta  \Pi_r \, \nabla \Pi_r \, d x
 &=
   \int_{B_{2r^\alpha}} \Big(\div (\nabla \Pi_r \otimes \nabla \Pi_r) 
 - 
  \nabla \frac{|\nabla \Pi_r|^2}{2}\Big) \, d x \nonumber 
  \\& = \int_{\partial B_{2r^\alpha}} \Big(\nabla \Pi_r  (\nabla \Pi_r \cdot n ) 
- 
 \frac{|\nabla \Pi_r|^2}{2} \, n \Big)\, d \mathcal H^1 = 0.
\end{align}
The last equality follows by direct computation in polar coordinates since $\Pi_r = r\wt{\Phi}$ is explicit when $|x| \geq 2 r^\alpha$, hence the integrand is $r^{-4}[ (\cos^2 \theta -1/2) \hat \rho- \sin \theta \cos \theta \hat \theta]$. 
 Using this fact along with equation (\ref{principal bb: initial eqn Wr}) and identity (\ref{principal bb: a simple identity}), we see that $W_r$ satisfies equation (\ref{principal bb: eqn W_r}) if we define
\begin{align}
P_1 &= \ol{P}_r - \frac{\partial_1 \Pi_r}{r}  + \frac{|\nabla \Pi_r|^2}{2},  \\
F_1 &= W_r \otimes \nabla \Pi_r + \nabla \Pi_r \otimes W_r + \mathcal{B}(\Delta \Pi_r \nabla \Pi_r) \, ,
\end{align}
where $\mathcal{B}$ is the Bogovskii operator defined in Appendix \ref{appendix:antidiv}.
When $|x| \geq 2 r^{\alpha}$, we see $\Pi_r =r\wt{\Phi}$ and from (\ref{principal bb: pressure outside rad r})
, we see that $P_1 = 0$, when $|x| \geq 2 r^{\alpha}$. Hence, $\supp P_1 \subseteq B_{2 r^\alpha}(0)$. From Proposition \ref{prop:Bogovskii2}, we also see that $\supp \mathcal{B}(\Delta \Pi_r \nabla \Pi_r) \subseteq B_{9 r^\alpha}(0)$, which after combining with (i) in Lemma \ref{lemma: properties Pi Wr} implies $\supp F_1 \subseteq B_{9 r^\alpha}(0)$.

\bigskip

We now estimate the $L^p$ norm of $F_1$. From (i) in Lemma \ref{lemma: properties Pi Wr}, (\ref{principal bb: estimate Pi nn}) and (\ref{principal bb: est 1 Wr})
, we see that $|W_r| \, |\nabla \Pi_r|$ is supported in $\{r^\alpha \leq |\cdot| \leq 2 r^\alpha\}$ and bounded by ${r^{2-4 \alpha}}$. Analogously, $ \Delta \Pi_r \nabla \Pi_r$ is supported   in $\{r^\alpha \leq |\cdot| \leq 2 r^\alpha\}$ and bounded by $
C {r^{2-5 \alpha}}$.
Hence, for every $p\in (1,\infty)$ we obtain that 
\begin{align}
\norm{W_r \otimes \nabla \Pi_r + \nabla \Pi_r \otimes W_r}_{L^p} \leq C(p) r^{2 - 2 \alpha} \, r^{\alpha \left(\frac{2 }{p} - 2 \right)}\, ,
\label{principal bb: F1 Lp 1}
\end{align}
\begin{align}\label{eqn:DeltaPnablaP}
\norm{\Delta \Pi_r \nabla \Pi_r}_{L^p} \leq C(p) r^{2 - 3\alpha} r^{\alpha(\frac{2}{p}-2)}  .
\end{align}
From Appendix \ref{appendix:antidiv}, an application of Proposition \ref{prop:Bogovskii2} and \eqref{eqn:DeltaPnablaP} 
provides us with the following estimate
\begin{align}
\norm{\mathcal{B} (\Delta \Pi_r \nabla \Pi_r)}_{L^p} \leq C(p) r^\alpha \norm{\Delta \Pi_r \nabla \Pi_r}_{L^p} \leq C(p) r^{2 - 2\alpha} r^{\alpha(\frac{2}{p}-2)} \, , \quad \text{for every $p\in (1,\infty)$}\, .
\label{principal bb: F1 Lp 2}
\end{align}
Combining (\ref{principal bb: F1 Lp 1}) and (\ref{principal bb: F1 Lp 2}) gives us the required $L^p$ estimate on $F_1$.


\bigskip

Now we focus on equation (\ref{principal bb: eqn partial r W_r}). We first observe that
\begin{align}
\partial_r \ol{V}_r 
= 
- \frac{1}{r} \ol{V}_r - \left(\frac{x}{r} \cdot \nabla \right) \ol{V}_r 
= 
\frac{1}{r} \ol{V}_r - \div\left(\ol{V}_r \otimes \frac{x}{r}\right)\, .
\label{principal bb: W_r prilim eqn}
\end{align}
From which we see that
\begin{align}
\partial_r W_r = \frac{1}{r} W_r - r\nabla (\wt{\Phi} \partial_r \rchi_\alpha) - \div\left(\ol{V}_r \otimes \frac{x}{r}\right). 
\label{principal bb: partial r Wr inter 1}
\end{align}
As regards the last term in the right-hand side, we observe that  $|x| \geq r$, we have $\ol{V}_r =r \wt{V}$ by \eqref{principal bb: lc outside rad r},  which combined with \eqref{principal bb: div tensor V x} leads to 
\begin{align}
\supp \div\left(\ol{V}_r \otimes \frac{x}{r}\right) \, \subseteq \, B_r(0). 
\label{principal bb: div Vr x inf 2}
\end{align}
We next show that this term has integral $0$, which is not an immediate from  the divergence theorem since $\ol{V}_r $ does not decay sufficiently fast. From \eqref{principal bb: partial r Wr inter 1}, since $\nabla (\wt{\Phi} \partial_r \rchi_\alpha)$ is compactly supported and integrates $0$,  and by \eqref{principal bb: main Wr integral}
\begin{align}
\int_{\mathbb{R}^2} \div\left(\ol{V}_r \otimes \frac{x}{r}\right) \, {\rm d} x = 
\int_{\mathbb R^2} (\partial_r W_r - \frac{1}{r} W_r )\, {\rm d} x =\int_{\mathbb R^2} r \partial_r \frac{W_r}{r} \, {\rm d} x = r \partial_r \left(\frac{1}{r} \int_{\mathbb R^2} W_r \, {\rm d} x\right)=
0.   
\label{principal bb: div Vr x inf 1}
\end{align}
and explicit computation


From (\ref{principal bb: estimate on Vbar}), for every $p\in [1,\infty]$ we get
\begin{equation}
    \norm{\div\left(\ol{V}_r \otimes \frac{x}{r}\right)}_{L^p} + r \norm{D\div\left(\ol{V}_r \otimes \frac{x}{r}\right)}_{L^p}
    \le C r^{\frac{2}{p} - 2}\, .
    \label{principal bb: div Vr x inf 3}
\end{equation}
We define the tensor $F_2$ as
\begin{align}
F_2 \coloneqq \mathcal{B}\left(\div\left(\ol{V}_r \otimes \frac{x}{r}\right)\right) \, .
\label{principal bb: def of F2}
\end{align}
and we observe that $\norm{\div F_2}_{L^\infty} \leq C {r^{-2}}$ by \eqref{principal bb: div Vr x inf 3}.
Now using Proposition \ref{prop:Bogovskii2}, we obtain the following estimate on the $L^p$ norm of $F_2$:
\begin{equation}
\norm{ F_2 }_{L^p} \le C(p) r \norm{\div\left(\ol{V}_r \otimes \frac{x}{r}\right)}_{L^p}
    \le C(p) r^{\frac{2}{p} - 1} \quad \text{for any } p \in (1, \infty).
\end{equation}
Finally, from (\ref{principal bb: W_r prilim eqn}), we see that $W_r$ satisfies (\ref{principal bb: eqn partial r W_r}) with $F_2$ as defined in (\ref{principal bb: def of F2}) and $P_2 = - r\wt{\Phi} \partial_r \rchi_\alpha$.

Since $\nabla^n \partial_r \rchi_r$ is supported in $\ol{B}_{2 r^\alpha} \setminus B_{r^\alpha}$ and bounded by $C(n, \alpha){r^{-1 - n \alpha}}$ for  $n \geq 0$, and by (\ref{principal bb: estimate on Vbar}) we obtain that 
$
\norm{\nabla P_2}_{L^\infty} + r\norm{ \nabla^2 P_2 }_{L^p} \leq C(\alpha){r^{-2 \alpha}}$.
\end{proof}

\subsection{Building Block with Non-Constant Speed on $\mathbb T^2$}
\label{sec:BB nonconstant speed}

Let $r: \mathbb{T}^2 \to (0,\infty)$ be a smooth positive function satisfying $\| r \|_{L_x^\infty} < \frac{1}{9}$. The function $r$ should be thought of as a space-dependent spatial scale. In the following sections, we will adjust $r(x)$ in relation to the Reynolds stress tensor in \eqref{eq:ER}.
We also consider a smooth function $\eta: \mathbb{R} \to [0,\infty)$
. It should be understood as a time-dependent cut-off function.  It will be used to keep our family of building blocks disjoint in time.

Next, we define the trajectory of the center of our building block. We want it to travel on a linear periodic trajectory in the two-dimensional torus. More precisely, given any unit vector $\xi\in \mathbb{R}^2$ with rationally dependent components, the center of mass
$x: \R \to \mathbb{T}^2$ will solve the ODE
\begin{align}\label{eq:ODE}
\begin{cases}
    x'(t) = \frac{\eta(t)}{r(x(t))} \xi
    \\
    x(t_0) = x_0 \, .
\end{cases}
\end{align}
In the sequel, the time-period of the linear trajectory $t\to t \xi$ will play the role of the frequency parameter $\lambda_{q+1}$ in classical convex integration schemes. To obtain this, the trajectory realizes a $1/\lambda_{q+1}$-dense set on the torus.

The velocity field in \eqref{eq:ODE} depends on the inverse of the space-dependent scale $r(x)$ so that the building block will spend more time in locations where the scale is large and leave quickly from locations where the scale is small. This will be key in our new mechanism of error cancellation. The function $\eta(t)$ in \eqref{eq:ODE} mainly serves as a time cut-off.

The principal building block is given by
\begin{align}\label{eq:V^p after W}
V^p(x,t) = \eta(t) \, W_{r(x(t))}(x - x(t)).
\end{align}
where $W_r\in C^\infty$ has been built in Proposition~\ref{prop:bb constant spead}  with  {$\alpha=1/5$} and has been smoothed according to Remark \ref{rmk:smoothing}, and rotated so that it solves
\begin{align}\label{eq:Wxi}
 - \frac{1}{r} (\xi \cdot \nabla) W_r + \div (W_r \otimes W_r) + \nabla P_1 = \div(F_1) \, ,  
 \\
 \label{eq:Wxi-2}
 \partial_r W_r = \frac{1}{r} W_r + \nabla P_2 + \div(F_2) \, , 
 \quad \text{on $\R^2$} \, .
\end{align}
As they are compactly supported, we consider the one-periodized version of $W_r$, $P_1$, $P_2$, and $F_1$, $F_2$, and associate them with functions on the $2$-dimensional torus $\mathbb{T}^2$. Finally, we correct the divergence of $V^p$ by adding a corrector $V^c$, obtained from $V^p$ through the anti-divergence operator $\nabla \Delta^{-1}$ on the torus, and therefore not supported in a small ball as $V^p$
\begin{align}
V(x,t) &:= V^p(x,t) + V^c(x,t)
\label{principal bb: V}
\\
V^c(x,t) &:= - \nabla \Delta^{-1} \div V^p(x,t)
\label{principal bb: Vc}
\end{align}

\begin{proposition}[Building Block]\label{prop:build-block1}
Let  $V$ be as in \eqref{principal bb: V}.
There exist $S \in C^\infty (\mathbb{T}^2 \times \R; \mathbb{R}^2)$, $P\in C^\infty ( \mathbb{T}^2 \times \R; \mathbb{R})$, and a symmetric tensor $F\in C^\infty(\mathbb{T}^2 \times \R; \mathbb{R}^{2\times 2})$ such that
\begin{equation}\label{eq:EulerV}
\begin{cases}
\partial_t V + \div(V\otimes V) + \nabla P 
=  S \frac{d}{dt} \left(\eta(t) \, r(x(t))\right)   +  \div(F) \, ,
\\
\div(V)  = 0\, ,
\end{cases}
\quad\quad \text{on $\mathbb{T}^2\times \R_+$}
\end{equation}
Moreover, the following estimates hold:
\begin{enumerate}[label = (\roman*)]
\item $\| F \|_{L^\infty_t L^{1}_x} \le C \| \eta \|_{L^\infty_t}^2 \| r \|_{L^\infty_x}^{\frac{1}{2}}(1 + \| \nabla \log(r)\|_{L^\infty_x})$.


\item For every $p\in (1,\infty)$, it holds
\begin{equation}
    \begin{split}
        &\| V \|_{L^\infty_t L^p_x} \leq C(p) \| \eta \|_{L^\infty_t } \| r^{\frac{2}{p} - 1} \|_{L^\infty_x} \, ,
        \qquad  \| D V \|_{L^\infty_t L^p_x}
        \le C(p) \| \eta \|_{L^\infty_t } \| r^{\frac{2}{p} - 2} \|_{L^\infty_x} \, ,
        \\[5pt]
        &\| \partial_t V \|_{L^\infty_t L^p_x} +
        \| r^{-1}\|_{L^\infty_x}^{-1} \| \partial_tD V \|_{L^\infty_t L^p_x}  \le C(p) \left(\norm{\eta^2}_{L^\infty_t} \norm{r^{-3}}_{L^\infty_x} \left( 1 + \norm{\nabla r}_{L^\infty_x} \right) +  \norm{\eta'}_{L^\infty_t} \norm{r^{-1}}_{L^\infty_x}\right) \, ,
        \\[2pt] 
        & \| S(\cdot, t) \|_{L_x^p}
    \le C(p)  \, r(x(t))^{\frac{2}{p}-2}\, .
    \end{split}
\end{equation}




\item For every $t\in \R_+$,  $S(\cdot, t)$ is supported in a ball  centered in $r(x(t))$ of radius $2 (r(x(t)))^{\frac{1}{5}}$, and 
\begin{align}
  \int_{\mathbb{T}^2} S(x,t)\, {\rm d} x =  2\pi \xi .
\end{align}

\end{enumerate}
\end{proposition}

\begin{proof}
 In the proof, we use the shorthand notations $ r(t) \coloneqq r(x(t))$. 
From \eqref{principal bb: V}, we have
\begin{align}
\partial_t V + \div(V \otimes V) &= \partial_t V^p + \partial_t V^c + \div (V^p \otimes V^p) 
\\&+ \div (V^p \otimes V^c + V^c \otimes V^p) + \div (V^c \otimes V^c) \, ,
\end{align}
then using \eqref{eq:Wxi} and \eqref{eq:Wxi-2}, we obtain
\begin{align}
\partial_t V^p(x,t) & = \partial_t (\eta(t) W_{r(x(t))}(x - x(t))) 
\\& 
= - \frac{\eta^2}{r} (\xi \cdot \nabla) W_{r} + \eta {r}^\prime \partial_r W_r  + \eta^\prime W_r \nonumber 
\\&
= {\eta^2} \left(- \div (W_r \otimes W_r) - \nabla P_1 +\div(F_1) \right)+ \eta {r}^\prime \left(\frac{1}{{r}} W_r + \nabla P_2 + \div (F_2) \right)  + \eta^\prime W_r \nonumber 
\\&
= -  \div (V^p \otimes V^p)  + (\eta {r})^\prime \frac{1}{{r}}W_r + \nabla (\eta {r}^\prime P_2- \eta^2 P_1) + \div (\eta {r}^\prime F_2+ \eta^2 F_1) \, .
\end{align}
Since 
$\partial_t V^c  = - \nabla \partial_t \Delta^{-1} \div V^p$,
we conclude that
\begin{align}\label{eq:EulerV-copy}
\partial_t V + \div(V \otimes V) +\nabla P 
=
S \frac{d}{dt} \left(\eta(t) {r}(t) \right) + \div (F) \, ,
\end{align}
where
\begin{align}
P &= \partial_t \Delta^{-1} \div V^p + \eta^2 P_1 - \eta {r}^\prime P_2 
\\
F & = F_3 + V^p \otimes V^c + V^c \otimes V^p + V^c \otimes V^c 
\\
S & = \frac{1}{{r}} W_r \, ;
\end{align}
the natural choice for $F_3$ to satisfy \eqref{eq:EulerV-copy} is $\eta^2 F_1 + \eta {r}' F_2$, but since this is not a symmetric tensor, we replace it with a symmetric tensor, without changing its divergence, thanks to the operator $\mathcal{R}_0 $ recalled in \eqref{eq:R0}
\begin{equation}
    F_3 \coloneqq \mathcal{R}_0 \div(\eta^2 F_1 + \eta {r}' F_2) \, .
\end{equation}
With this definition of $S$, statement (iii) and the last inequality of statement (ii) follow from Lemma~\ref{lemma: properties Pi Wr}(i), (ii) and (iv). 

As a consequence of \eqref{eq:R0est2},  item (ii) in Proposition \ref{prop:bb constant spead} and noting
\begin{align}
{r}^\prime = \nabla r \cdot x^\prime(t) = \eta \frac{(\xi \cdot \nabla r)}{{r}} = \eta (\xi \cdot \nabla \log r),
\end{align} for every $p\in (1,+\infty)$ we have
\begin{align}
\norm{F_3}_{L^\infty_t L^1_x} \leq \norm{F_3}_{L^\infty_t L^p_x}    &\le C(p) \|\eta^2 F_1 + \eta {r}^\prime F_2\|_{L^p_x}
\\&
    \le C(p) \left( \norm{\eta}_{L^\infty_t}^2 \norm{F_1}_{L^p_x} +  \, \norm{\eta}_{L^\infty_t} \; |{r}^\prime| \;  \norm{F_2}_{L^p_x} \right)
    \\&
     \leq C(p) \norm{\eta}_{L^\infty_t}^2 \left(\norm{r^{2 + \frac{2 \alpha}{p}  - 4\alpha}}_{L^\infty_x} + \norm{\nabla \log r}_{L^\infty_x} \norm{r^{\frac{2}{p}  - 1}}_{L^\infty_x} \right).
\label{principal bb: int F3}
\end{align}
With the choices of 
$p = \frac{12}{11},\, \alpha = \frac{1}{5}$
in (\ref{principal bb: int F3}), we see that
\begin{align}
\norm{F_3}_{L^\infty_t L^1_x} \leq C \norm{\eta}_{L^\infty_t}^2 \norm{r^{\frac{5}{6}}}_{L^\infty_x} \left(1 + \norm{\nabla \log r}_{L^\infty_x} \right).
\label{principal bb: L1 est F3}
\end{align}

\bigskip

{We observe that by Lemma~\ref{lemma: properties Pi Wr}(ii) and (iii), for every $p\in (1, \infty)$,} 
\begin{align}\label{est-vp}
\norm{V^p (\cdot, t)}_{L^p_x} \leq C(p) \norm{\eta}_{L^\infty_t} \, {r}^{\frac{2}{p} - 1} \, ,
\\
\norm{D V^p (\cdot, t)}_{L^p_x} \leq C(p) \norm{\eta}_{L^\infty_t} \, {r}^{\frac{2}{p} - 2}.
\end{align}
From Poincar\'e inequality $\norm{V^c}_{L^p} \leq C(p) \norm{D V^c}_{L^p}$  and the Calderon--Zygmund theory, we have
\begin{align}
\norm{V^c (\cdot, t)}_{L^p_x} + \norm{D V^c(\cdot, t)}_{L^p_x}
     & \le C(p)\norm{\eta}_{L^\infty_t} \norm{ \div W_{{r}(t)} }_{L^p_x} \leq C(p) \norm{\eta}_{L^\infty_t} \left( {r}(t) \right)^{1 + \frac{2 \alpha}{p} - 3 \alpha}\, ,
\end{align}
for every $p\in (1,\infty)$.
{
 By Calderon-Zygmund theory}, we have
$\norm{V^c (\cdot, t)}_{L^{ p}_x} \leq C(p) \norm{V^p (\cdot, t)}_{L^{p}_x}$, which together with \eqref{est-vp} establishes the first inequality in statement (ii). Moreover 
we obtain
\begin{align}
\norm{V^p \otimes V^c + V^c \otimes V^p + V^c \otimes V^c}_{L^{p}_x} 
&\leq C(p) \norm{V^p}_{L^{2p}_x} \norm{V^c}_{ L^{2p}_x}
\\
&\leq C(p)  \norm{\eta}_{L^\infty_t}^2 \| r^{\frac{1 + \alpha}{p} - 3 \alpha} \|_{L^\infty_x}\, ,
\label{principal bb: Vp Vc symm}
\end{align}
for every $p\in (1,\infty)$. Using the choice as above of $p = \frac{12}{11},\, \alpha = \frac{1}{5}$
  in (\ref{principal bb: Vp Vc symm}), we obtain
\begin{align}
\norm{V^p \otimes V^c + V^c \otimes V^p + V^c \otimes V^c}_{L^\infty_t L^{1}_x} 
&\leq C  \norm{\eta}_{L^\infty_t}^2 \| r^{\frac{1}{2}} \|_{L^\infty_x}\, .
\label{principal bb: L1 est Vp Vc symm}
\end{align}
Finally, combining (\ref{principal bb: L1 est F3}) and (\ref{principal bb: L1 est Vp Vc symm}), we conclude property (i) in Proposition \ref{prop:build-block1}.

\bigskip

Now we focus on estimating $\partial_t V$ and $D V$ and $\partial_tD V$ in $L^p$.
We begin by noticing that
\begin{equation}
    \| D_{x,t} V^c \|_{L^\infty_t L^p_x} \le C(p) \| D_{x,t} V^p\|_{L^\infty_t L^p_x} \, ,
    \quad \text{for every $p\in (1,\infty)$}
\end{equation}
by Calderon-Zygmund theory, since $V^c= - \nabla \Delta^{-1}\div(V^p)$. Hence, it will be enough to estimate $D_{x,t} V^p$. The same consideration works for $\partial_tD V^p$. By  \eqref{est-vp} we get the estimate $DV$ in statement (ii). 
To bound the time derivative, we restart from the identity
\begin{align}\label{eqn:identity}
\partial_t V^p =
- \frac{\eta^2}{{r}} (\xi \cdot \nabla) W_r  + (\eta {r})^\prime \frac{1}{{r}}W_r + \nabla (\eta {r}^\prime P_2) + \div (\eta {r}^\prime F_2) \, ,
\end{align}
which implies
\begin{equation}\label{eqn:dtvp}
    \begin{split}
        \| \partial_t V^p(\cdot, t)\|_{L^\infty_x} 
        & \le \frac{\eta^2}{{r}} \| D W_r \|_{L^\infty_x} 
        +|\eta'| \| W_r \|_{L^\infty}
        \\&
        + \eta |{r}'| \left(\frac{1}{{r}}\| W_r \|_{L^\infty_x} + \|\nabla P_2 \|_{L^\infty_x}
        + \| \div (F_2) \|_{L^\infty_x}\right)
        \\&
        \le C\eta^2 {r}^{-3} + C|\eta'|{r}^{-1} + C\eta |{r}'| {r}^{-2} \, ,
    \end{split}
\end{equation}
where we used Lemma~\ref{lemma: properties Pi Wr} 
 and Proposition \ref{lemma: properties Pi Wr} (iii).

Differentiating \eqref{eqn:identity} with respect to space, we obtain the estimate for $\| \partial_t D V^p(\cdot, t) \|_{L^\infty_x}$. This estimate is analogous to \eqref{eqn:dtvp}, but it involves, in the first and second lines, one additional derivative of $W_r$, $F_2$, and $P_2$. The estimates for these quantities can be found in Lemma~\ref{lemma: properties Pi Wr} and Proposition~\ref{lemma: properties Pi Wr}(iii).

\end{proof}
 
\section{Iteration and Proof of the Main Theorems}
\label{sec: iteration}

%
%
%
%
%
%
%
%
%
%
%
%
%
%
%
%
%
%
%
%
%

\bigskip

In this section, we will begin by presenting the choice of parameters and the Euler-Reynolds system. Subsequently, the objective of this section is to assemble all the main ingredients necessary to complete our convex integration scheme. The primary components comprise definitions of parameters, the mollification step, error decomposition, time series, adaptation of the building block from Proposition \ref{prop:build-block1}, auxiliary building block, and a time corrector.

\subsection{Choices of Parameters}
As described in the overview section, there are four parameters involved in the $q$th step of our convex integration scheme, namely, $\delta_{q}$ (the amplitude or the error size), $\lambda_{q}$ (related to the slope of the trajectory), $r_q$ (the size of the core of the building block), and $\tau_q$ (the size of time intervals). Next, we specify the following dependencies among the parameters. Let $\lambda_0,\sigma,\kappa \in \mathbb N,$ and $ \beta, \mu$ be positive parameters. 
For $q \in \mathbb{Z}_{\geq 0}$, we define
\begin{align}
\lambda_{q+1} = \lambda_q^{\sigma} \, ,
\quad 
\delta_q = \lambda_1^{2\beta}\lambda_q^{-\beta} \, ,
\quad
r_{q+1} = \lambda_{q+1}^{-\mu} \, ,
\quad
\tau_{q+1} = \lambda_{q+1}^{-\kappa}
\label{eq:parameters}
\end{align}
In the sequel, we will choose the parameters to satisfy a few simple inequalities (see Section \ref{sec:parameters}) that will be derived to close the convex integration scheme. For the reader's convenience, we prefer to specify here one admissible choice of such parameters, found a posteriori, which will satisfy all the required inequalities derived during the proof:
{
\begin{equation}\label{eq:parameters2}
\beta = \frac{1}{245}, \quad \mu = \frac{53}{10}, \quad \kappa = 3, 
\quad \sigma = 110\, .
\end{equation}
}

\subsection{The Euler--Reynolds System}
In this section, we set up the iteration of our convex integration scheme. 
At the $q$th step of the iteration, we construct solutions to the Euler--Reynolds system
\begin{equation}\label{eq:ER}
    \begin{cases}
        \partial_t u_q + \div(u_q \otimes u_q) + \nabla p_q = \div(R_q),
        \\
        \div u_q = 0 \, ,
    \end{cases}
    \tag{E--R}
\end{equation}
{on $\mathbb{T}^2 \times 
\mathbb R$,}
satisfying the estimates:
\begin{equation}\label{eq:iterate1}
    \| R_q \|_{L^\infty_t L^1_x} \le \delta_{q+1} \, ,
    \quad
     \| u_q \|_{L^\infty_t L^2_x} \leq 2 \delta_0^{1/2}- \delta_q^{1/2}\ ,
     \quad
    {\| 
     D_{x,t} u_q \|_{L^\infty_t L^4_x} }\le \lambda^n_q \, ,
\end{equation}
where the number $n$ is positive.

\begin{proposition}[Iteration Step]\label{prop:iteration}
Let $\overline p = 1 + \frac{1}{6500}$ and $\sigma,\beta,\mu, \kappa$ as in \eqref{eq:parameters2}, $n=16$, and $\alpha = 10^{-5}$. There exists $M\ge 1$ such that for every $\lambda_0\ge \lambda_0(M)$ the following statement holds.
    Let $(u_q,p_q,R_q)$ be a solution to \eqref{eq:ER} satisfying \eqref{eq:iterate1}. Then, there exists $(u_{q+1}, p_{q+1}, R_{q+1})$ smooth solution to \eqref{eq:ER} such that
    \begin{itemize}
        \item[(i)] $\| R_{q+1} \|_{L^\infty_t L^1_x} \le \delta_{q+2}$,  $\| u_{q+1} \|_{L^\infty_t L^2_x} \leq 2 \delta_0^{1/2}- \delta_{q+1}^{1/2}$,
        $\| D_{x,t} u_{q+1} \|_{L^\infty_t L^4_x} \le \lambda^n_{q+1}$.
        
        \item[(ii)] {$\|u_{q+1} - u_q\|_{L^\infty_tL^2_x}
        \le M \delta_{q+1}^{1/2}$}, 
        $\|D u_{q+1}\|_{C^\alpha_t L^{\overline{p}}_x} \leq \|  D u_q\|_{C^\alpha_t L^{\overline{p}}_x} +    \lambda_0\delta_{q+1}^{1/10}$.

        \item[(iii)] {$\|u_{q+1}(\cdot, 0) - u_q(\cdot,0)\|_{L^2} \le \lambda_q^{-1}$, $\|u_{q+1}(\cdot, 1) - u_q(\cdot ,1)\|_{L^2} \le \lambda_q^{-1}$.
        }

    \end{itemize}
\end{proposition}

{
The iterative estimate (i) guarantees that the error $R_q$ converges to zero in $L^\infty_tL^1_x$ as $q\to \infty$, while maintaining some control over the space-time derivative $D_{x,t} u_q$ of the velocity field. However, the latter control weakens as $q$ tends to infinity and is ultimately lost. Nonetheless, it serves as a crucial technical component in proving Proposition \ref{prop:iteration}.
The estimates (ii) ensure that in each iteration, the new velocity field $u_{q+1}$ is close to the previous one $u_q$ in the relevant functional space $C_t(L^2_x\cap W^{1,\overline{p}}_x)$.
Finally, (iii) tracks the velocity field at the initial and final times $t=0$ and $t=1$. This is crucial in the proof of Theorem \ref{thm:main} to prescribe the initial and final conditions up to a small error. Its validity is a consequence of the time intermittency in our construction; see Remark \ref{rmk:control ending points}. 

}

{

\begin{remark}[Locality in Time]
\label{rmk:locality in time}
In our proof of Proposition \ref{prop:iteration}, the construction of $(u_{q+1},p_{q+1},R_{q+1})$ from $(u_{q},p_{q},R_{q})$ exhibits a certain locality in time. The precise statement is as follows:
\begin{itemize}
    \item[(iv)] Assume that $(u_{q},p_{q},R_{q})$ and $(u_{q}',p_{q}',R_{q}')$ are solutions to \eqref{eq:ER} satisfying \eqref{eq:iterate1}. If they coincide on $[0,t]$, for some $t\geq 1/9$, then we can construct $(u_{q+1},p_{q+1},R_{q+1})$ and $(u_{q+1}',p_{q+1}',R_{q+1}')$ satisfying (i), (ii), and coinciding on $[0,t -\lambda_q^{-1}]$.
\end{itemize}
    
\end{remark}

}

{

\subsection{Proof of Theorem \ref{thm:main} and Theorem \ref{thm:nonuniqueness}}\label{sec:proofs1}

In this section, we rely on Proposition \ref{prop:iteration} to complete the proof of Theorem \ref{thm:main} and Theorem \ref{thm:nonuniqueness}.

\bigskip

We begin with Theorem \ref{thm:main}.
We fix $\eps > 0$ and proceed to define, for $(x,t) \in \mathbb T^2 \times [0,1]$,
\begin{align}\label{eq:u0p0R0}
    &u_0(x,t) : = \chi(t) (u_{\rm start}\ast \rho_\ell)(x) + (1-\chi(t)) (u_{\rm end}\ast \rho_\ell)(x) \, ,
    \\& 
    p_0(x,t) := 0 \, ,
    \\&  
    R_0(x,t) := \mathcal{R}_0(\partial_t u_0 + \div(u_0 \otimes u_0))(x,t) \, ,
\end{align}
where $\rho_\ell$ is a smooth convolution kernel, $\ell>0 $ is small enough to ensure that
\begin{equation}\label{eq:elleps}
    \| u_{\rm start} - u_{\rm start}\ast \rho_\ell \|_{L^2} 
    +
    \| u_{\rm end} - u_{\rm end}\ast \rho_\ell \|_{L^2}
    \le \frac{\eps}{2} \, ,
\end{equation}
and $\chi(t)$ is a smooth time cut-off such that $\chi(t) = 1$ for $t\le 1/4$, and $\chi(t)=0$ for $t\ge 1/2$.
Notice that $R_0$ is a well-defined symmetric tensor since $u_{\rm star}$ and $u_{\rm end}$ are mean-free velocity fields.

We choose $\lambda_0 \in \N$ big enough so that
\begin{equation}
    \| R_0 \|_{L^\infty_t L^1_x} \le \delta_1  \, ,
    \quad
    \| D_{x,t} u_0 \|_{L^\infty_t L^4_x} \le \lambda_{0}^{16}\, ,
    \quad
    4\lambda_0^{-1} \le \eps \, .
\end{equation}
 and it satisfies the condition specified in Proposition \ref{prop:iteration}.
We can apply Proposition \ref{prop:iteration} to produce a sequence of smooth solutions $(u_q,p_q,R_q)$ to the Euler--Reynolds system \eqref{eq:ER} satisfying properties (i), (ii), and (iii). From (ii), it follows that the sequence $(u_q)_{q\in \N}$ satisfies the gradient bound
\begin{equation}
   \sup_{q \geq n} \| D u_q - D u_n \|_{L^\infty_t L^p_x}
    \le   \sum_{q'=n}^\infty \| D u_{q'+1} - D u_{q'} \|_{L^\infty_t L^p_x}
    \le
    \lambda_0 \sum_{q'=n}^\infty \delta_{q'}^{1/10}
    \, .
\end{equation}
This shows that $(u_q)_{q\in \N}$ is a Cauchy sequence and converges in $ C_t(L^2_x\cap W^{1,\overline p}_x)$ to
\begin{equation}\label{eq:u}
    u(x,t) := u_0(x,t) + \sum_{q=0}^\infty (u_{q+1}(x,t) - u_q(x,t)) \in C_t(L^2_x\cap W^{1,\ol{p}}_x) \, .
\end{equation}
In particular, $\omega := {\rm curl} u \in  C_t L^{\overline p}_x$.

\bigskip

As a consequence of (iii) and \eqref{eq:u}, it follows that
\begin{align}
    \|u(\cdot, 0) - u_{\rm start}\ast \rho_\ell\|_{L^2}
    &=
    \|u(\cdot, 0) - u_0(\cdot,0)\|_{L^2}
    \\&
    \le 
    \sum_{q=0}^\infty \| u_{q+1}(\cdot,0) - u_q(\cdot,0)\|_{L^2}
    \\&\le \sum_{q=0}^\infty \lambda_q^{-1}
    \le 2 \lambda_0^{-1}
    \le \eps/2 \, .
\end{align}
In view of \eqref{eq:elleps}, we conclude $\| u(\cdot,0) - u_{\rm start}\|_{L^2} \le \eps$. Similarly, we also obtain $\| u(\cdot,1) - u_{\rm end}\|_{L^2} \le \eps$.

\bigskip

Finally, it is standard to check that $u$ is a weak solution of the Euler equations \eqref{EU} by taking the limit $q\to \infty$ in the distributional formulation of \eqref{eq:ER}, since $\|u_q-u\|_{L^\infty_t L^2_x} \to 0$ and $\| R_q \|_{L^\infty_t L^1_x} \to 0$, as a consequence of (i) in Proposition \ref{prop:iteration}.

\bigskip

To prove Theorem \ref{thm:nonuniqueness}, it suffices to combine the previous construction together with Remark~\ref{rmk:locality in time}. We fix a divergence-free velocity field $u_{\rm start}\in L^2$ with zero mean. For every $u_{\rm end}\in L^2$, with zero divergence and mean, we build $(u_0,p_0,R_0)$ as in \eqref{eq:u0p0R0}. All of this solutions to \eqref{eq:ER} coincide in $[0,1/4]$ by construction. Hence, by Remark \ref{rmk:locality in time}, at each stage of the iteration the new solutions will coincide in a definite neighborhood of $t=0$. Therefore, in the limit we get infinitely many solutions with the same initial condition $v$ satisfying $\| v - u_{\rm start}\|_{L^2} \le \eps$.

}

\subsection{Solutions with Time-Wise Compact Support}

In this section, we rely on Proposition \ref{prop:iteration} to complete the proof of Theorem \ref{thm:main2}.

\bigskip

Let $\beta, \sigma$ as in \eqref{eq:parameters2}. Let $\lambda_0\in \N$ be big enough. We define
\begin{align}
    &u_0(x,t) : = \lambda^{\frac{3}{4}\beta \sigma}_0 \chi(t) \sin(\lambda_0 x_2) e_1
    \\& 
    p_0(x,t) := 0 \, ,
    \\&  
    R_0(x,t) := -\lambda_0^{-1 + \frac{3}{4}\beta \sigma} \chi'(t) \cos(\lambda_0 x_2) (e_1\otimes e_2 + e_2 \otimes e_1)\, ,
\end{align}
where \(\chi \in C^\infty(\mathbb{R})\) is a cut-off function such that \(\chi = 1\) on \((1/2, 3/4)\) and \(\chi = 0\) on \((-\infty, 1/4) \cup (7/8, +\infty)\). It turns out that \((u_0, p_0, R_0)\) solves the Euler--Reynolds system \eqref{eq:ER} with the following estimates:
\begin{equation}\label{z30}
    \| R_0 \|_{L^\infty_t L^1_x} 
    \le 20 \lambda_0^{-1 + \frac{3}{4}\beta \sigma} \, ,
    \quad
    \| u_0 \|_{L^\infty_t L^2_x} = C \lambda_0^{\frac{3}{4}\beta \sigma} \, ,
    \quad
    \| D_{x,t} u_0 \|_{L^\infty_t L^4_x} \le 20 \lambda_0^{1+\frac{3}{4}\beta \sigma} \, ,
\end{equation}
which allows us to start the iteration, provided \(\lambda_0\) is sufficiently large.

We obtain a sequence $(u_q,p_q,R_q)$ of solutions to \eqref{eq:ER} satisfying the inductive estimates (i), (ii) in Proposition \ref{prop:iteration}. Taking into account Remark \ref{rmk:locality in time}, we can assume that $u_q(x,t) = 0$ when $t\le 1/8$ and $t\ge 1$, for every $x\in \mathbb{T}^2$. Arguing as in Section \ref{sec:proofs1} we deduce that $u_q \to u$ in $C_t(L^2_x \cap W_x^{1,\overline{p}})$ while $R_q \to 0$ in $L^\infty_t L^1_x$. Moreover, $u(x,t)$ solves \eqref{EU} for a suitable pressure and is compactly supported in time. Moreover,
\begin{align}
    \| u - u_0 \|_{L^\infty_t L^2_x} 
    &\le \| u_1 - u_0 \|_{L^\infty_t L^2_x}
    + \sum_{q\ge 1} \| u_{q+1} - u_q\|_{L^\infty_t L^2_x}
    \\&
    \le M\left( \delta_1^{1/2} + \sum_{q\ge 1} \delta_{q+1}^{1/2} \right)
    \\&
    \le M \lambda_1^{\beta/2} + C(M) \lambda_0^{\beta(1-\sigma/2)}\, .
\end{align}
Hence, if \(\lambda_0\ge \lambda_0(M)\) is sufficiently large, then 
\begin{equation}
    \| u - u_0 \|_{L^\infty_t L^2_x} 
    \le 2M \lambda_1^{\beta/2}
    <
    C\lambda_0^{\frac{3}{4}\beta \sigma} =\| u_0 \|_{L^\infty_t L^2_x} \, ,
\end{equation}
thus $u$ is nontrivial.

\section{The Perturbation
}

In this section, we gather all the necessary ingredients to define the new velocity field \( u_{q+1} \) as an additive perturbation of \( u_q \). As explained in the introduction, our perturbation is designed to have, up to lower order corrections, some qualitative features, which we recall here.
At any given time, the principal part of our perturbation consists of only one building block whose vorticity is compactly supported and, at first approximation, translating in a fixed direction; in different time intervals, such direction switches between four fixed directions.  
The speed of translation and the spatial scale of the building block vary and are determined by the previous error. 


\bigskip

We give a more detailed overview of the steps of the construction. Firstly, in Section~\ref{section: mollification} given a solution $u_q$ of the Euler--Reynolds system with error $R_q$, we perform a standard procedure in convex integration to avoid the ``loss of derivative'' problem. We consider a mollified version $u_{\ell}$ of $u_q$, where the convolution parameter is chosen small enough to control the error coming from the convolution of the nonlinearity of the equation. In this way, we have quantitative controls on all the derivatives of the convolved vector field  $u_{\ell}$ and on the associated Reynolds stress $R_\ell$.
 
 Next in Section~\ref{sec:err-dec} we consider a decomposition of the error $R_q$ in rank one directions as
\begin{align}\label{eqn:errr-dec}
    - \div(R_q) = \div\left(\sum_{i=1}^4 a_i(x,t) \xi_i \otimes \xi_i\right) + \nabla P^{\, d},
\end{align}
such the coefficients $a_i$, $i \in \{1, 2, 3, 4\}$, are bounded below by a positive constant and have the same size as $R_q$ in $L^\infty_t L^1_x$. The peculiarity of our family $\xi_i$ is related to the following: the lines $\mathbb R \xi_i$, which corresponds to the trajectory of our building block in direction $\xi_i$, need to reconstruct a periodic set on the torus, whose period is, up to a constant, is  $ \lambda_{q+1}^{-1}$.

In Section~\ref{sec:time}, we split $\mathbb R$ into time intervals of length $\tau_{q+1}$, namely  $\mathcal{T}^k=[k\tau_{q+1}, (k+1)\tau_{q+1}]$.
The intervals $\mathcal{T}^k$ is further divided into four subintervals $\mathcal{T}^k_i$, $i \in \{1, 2, 3, 4\}$ of length $\tau_{q+1}/4$, on each of which the principal part of the perturbation will have direction $\xi_i$ and will run around the associated trajectory in a time, called period, much smaller than $\tau_{q+1}/4$.
Different directions are then patched together with a system of cutoffs $\zeta^k_i$ whose support is in $\mathcal{T}^k_i$.


In Section~\ref{sec:spa-scale} we introduce the varying size $r^k_i(x)$ of our building block, which is proportional to the (time averaged) coefficients $a_i$ in \eqref{eqn:errr-dec}, and in Section~\ref{sub:traj cent} we specify the ODE solved by the center $x^k_i(t)$ of the main part of the perturbation, which is essentially forced by scaling once one fixes the space size and expects the building block to solve Euler up to a small error.

Next, as first mentioned in Section~\ref{sec:introbb}, we define $V_i^k(x, t)$ in Section~\ref{sec:princbb} to be the velocity field from Proposition \ref{prop:build-block1} applied to the spatial scale $r^k_i(x)$ and the time cutoff $\eta^k_i \zeta^k_i(t)$, which solves
\begin{equation}
\begin{cases}
\partial_t V_i^k + \div(V_i^k\otimes V_i^k) + \nabla P_i^k = 
 S_i^k {\frac{d}{dt} \left( \eta_i^k \zeta_i^k(t) r_i^k(x^k_i(t)) \right)} +  \div(F_i^k)\, ,
 \quad t\in \mathcal{T}_i^k \, ,
 \\
 \div V_i^k =0 \,
 \end{cases}
 \label{iteration: BB Vki eqn}
\end{equation}
where the source/sink term $S_i^k {\frac{d}{dt} \left( \eta_i^k \zeta_i^k(t) r_i^k(x^k_i(t)) \right)}$ was designed in the previous section to be supported around $x^k_i(t)$ at scale $r^k_i(x^k_i(t))$ and will be responsible for the error cancellation. Physically, this term represents the gain or shedding of momentum due to the size and speed variation of the building block.

The final two essential components of our construction are the auxiliary building blocks and the time corrector  introduced in Section~\ref{sec:auxiliary} and~\ref{sec:timecorr},
 respectively. To understand the necessity of these objects, we closely inspect the error cancellation procedure, which is broadly described in  Section~\ref{sec:ta-errcanc}.
To cancel the error $R_q$, at first, we hope to use the time average of $S_i^k {\frac{d}{dt} \left( \eta_i^k \zeta_i^k(t) r_i^k(x^k_i(t)) \right)}$ on the interval $\mathcal{T}^k$, denoted as 
$P_{\tau_{q+1}} \left(S_i^k \frac{d}{dt} \left\{ \eta_i^k \zeta_i^k(t) r_i^k(x^k_i(t)) \right\} \right)$. Hence we hope that this term approximates $\div R_q$ up to an error whose anti-divergence  is sufficiently small to be included in the smaller error $R_{q+1}$. Based on this, we define a corrector $Q_{q+1}$, first described in Section~\ref{sec:ta-errcanc}, which we add in the perturbation of $u_q$, such that $\partial_t Q_{q+1}$ cancels the difference $S_i^k \frac{d}{dt} \left\{ \eta_i^k \zeta_i^k(t) r_i^k(x^k_i(t)) \right\} - P_{\tau_{q+1}} \left(S_i^k \frac{d}{dt} \left\{ \eta_i^k \zeta_i^k(t) r_i^k(x^k_i(t)) \right\} \right)$. However, it turns out that for the corrector defined in this manner, the $L^p$ norm of the vorticity in $ Q_{q+1}$ becomes uncontrollable, since the support of $\partial_t Q_{q+1}$ lies in a thin strip of the order of the size of $r_i^k$ (which is small in our construction) around the trajectory of the building block. Therefore, we pay $(r_i^k)^{-1}$ (typically quite large) for the spatial derivative.


To remedy the situation, we introduce in Section~\ref{sec:auxiliary} the idea of an auxiliary building block $U_{q+1}$ and define the corrector $Q_{q+1}$ in Section~\ref{sec:timecorr}
 based on $U_{q+1}$ instead. We choose $U_{q+1}$ such that at any time $t$ in some $ \mathcal{T}^k_i$, 
$U_{q+1}(x, t)$ has the same space average of $S^k_i$ and the support of $U_{q+1}(\cdot, t)$ lies in a fixed ball centered around $S^k_i$ but with bigger radius, of size $\sim \lambda_{q+1}^{-1}$. Roughly, speaking the term $U_{q+1}(\cdot, t)$ runs parallel to the term $S^k_i(\cdot, t)$ as such we can absorb their difference in the new error $\div R_{q+1}$. The error cancellation happens then thanks to the term $P_{\tau_{q+1}} \left(S_i^k \frac{d}{dt} \left\{ \eta_i^k \zeta_i^k(t) r_i^k(x^k_i(t)) \right\} \right)$, which cancels $\div R_{q}$. Correspondingly we introduce the time corrector $Q_{q+1}$ based on $U_{q+1}$ instead of  $S_i^k$.

We collect the definition of our perturbation (principal part and correction), the associated Reynolds stresses and the pressure in Section~\ref{sec:velocity perturbation}. Finally, in section~\ref{sec:est-pert-rey}, we present their estimates and give choice of parameters that allow us to close the proof of Proposition~\ref{prop:iteration}.

\subsection{Mollification Step}\label{section: mollification}
Let $\ell >0$ be a scale parameter that will be chosen later. We define 
\begin{equation}\label{defn:uRell}
    u_\ell = u_q \ast \rho_\ell \, ,
    \quad
    p_\ell = p_q \ast \rho_\ell\, ,
    \quad
    R_\ell = R_q \ast \rho_\ell + u_\ell \otimes u_\ell - (u_q\otimes u_q)\ast \rho_\ell \, ,
\end{equation}
where $\rho_\ell$ is a smooth mollifier in space and time, at length scale $\ell$.
It is immediate to check that $(u_\ell, p_\ell, R_\ell)$ solves the Euler--Reynolds equation
\begin{equation}
    \partial_t u_\ell + \div(u_\ell \otimes u_\ell) + \nabla p_\ell = \div(R_\ell) \, .
\label{mollified eqn}
\end{equation}
We have
\begin{align}
\| R_\ell \|_{L^\infty_t L^1_x} 
&\le \| R_q \ast \rho_\ell \|_{L^\infty_t L^1_x} + \|u_\ell \otimes u_\ell - (u_q\otimes u_q)\ast \rho_\ell\|_{L^\infty_t L^1_x} \nonumber \\
&\le \| R_q \|_{L^\infty_t L^1_x} + \|u_\ell \otimes u_\ell - u_\ell\otimes u_q\|_{L^\infty_t L^1_x} + \|  u_\ell\otimes u_q- (u_q\otimes u_q)\ast \rho_\ell\|_{L^\infty_t L^1_x} \nonumber \\
&\le  \delta_{q+1} + C_0 \ell \|u_q\|_{L^\infty_t L^2_x} \| D_{x,t} u_q \|_{L^\infty_t L^2_x} \, .
\label{iteration: Rl estimate}
\end{align}
We choose  
\begin{align}\label{eq:ell}
\ell := 
 \delta_0^{-1/2} \lambda_0^{-\beta} \delta_{q+1}\lambda_q^{-n} \ ,
\end{align}
so that
\begin{equation}
    \| R_\ell \|_{L^\infty_t L^1_x} \le 2 \delta_{q+1}\, ,
\end{equation}
provided $\lambda_0 \ge \lambda_0(C_0)$ is big enough.
The same choice of $\ell$ also provides
\begin{align}
&\norm{R_\ell}_{C_{x, t}^0} 
\leq C \ell^{-2} \norm{R_q}_{L^\infty_t L^1_x} 
+ C \ell^{-2} \| u_q\|_{L^\infty_t L^2_x}^2
\leq C \delta_0^2 \lambda_0^{2\beta} \delta_{q+1}^{-2} \lambda^{2n}_q = C   \lambda^{2n+2\beta\sigma}_q\, , 
\\
& \norm{R_\ell}_{C_{x, t}^1} \le \ell^{-1} \norm{R_\ell}_{C_{x, t}^0}
\le C  \lambda^{3n+3 \beta\sigma}_q \, ,
\end{align}
and
\begin{equation}\label{uell-u}
\| u_\ell -u_q\|_{L^\infty_tL^2_x} \leq C \ell \| D_{x,t} u_q \|_{L^\infty_{t}L^2_x} \leq \delta_0^{-1/2} \lambda_0^{-\beta} \delta_{q+1} = \lambda_1^{\beta}\lambda_0^{-\beta/2} \lambda_{q+1}^{-\beta} =  (\lambda_0^{-\beta/2} \lambda_{q+1}^{-\beta/2}) \delta_{q+1}^{1/2} \, .
\end{equation}
Notice that the presence of $\lambda_0^{-\beta}$ in the definition of $\ell$ is useful to make $\| u_\ell -u_q\|_{L^\infty_tL^2_x}$ small for every $q$, including in particular $q=0$; this will be important in Remark~\ref{rmk:control ending points} below.

\subsection{Error Decomposition}
\label{sec:err-dec}
%
%
%
%

\begin{lemma}[{Error Decomposition}]\label{lemma:error dec}
Given $\lambda_{q+1} \geq 8, \delta_{q+1} >0 $, there exist unitary vectors $\xi_1, \xi_2, \xi_3$,$\xi_4$ in $\mathbb{R}^2$ with rationally dependent components such that the following holds. For every $R_\ell \in C^\infty(\mathbb T^2 \times [0,1];  {\rm Sym}_2)$
\begin{enumerate}[label = (\roman*)]
    \item The time period of the curve $s \to s \xi_i$ in $\mathbb{T}^2$ is $c_i \lambda_{q+1}$ for $c_i \in [1, 3]$,
    \item  The following decomposition holds:
    \begin{align}\label{eq:error dec}
     - \div(R_\ell) = \div\left(\sum_{i=1}^4 a_i(x,t) \xi_i \otimes \xi_i\right) + \nabla P^{\, d} \, ,
     \end{align}
     \item The functions $a_i(x, t)$ are smooth and satisfy
\begin{align}
a_i(x,t) \ge \delta_{q+1} \, ,
\quad 
\| a_i \|_{L^\infty_t L^1_x} \le 192
 \delta_{q+1} \, ,
\quad 
\norm{a_i}_{C^{0}_{x, t}} \leq C(\delta_{q+1}+ \norm{R_\ell}_{C^{0}_{x, t}}), \quad \norm{a_i}_{C^{1}_{x, t}} \leq C \norm{R_\ell}_{C^{1}_{x, t}}.
\label{eq:est_a}
\end{align}
\end{enumerate}
\end{lemma}

%
%

\begin{proof}
Following \cite{BrueColombo23}, we define four unitary vectors as follows:
\begin{align}
e_1 \coloneqq (1, 0)^T, \quad e_2 \coloneqq (0, 1)^T, \quad e_3 \coloneqq \left(\frac{1}{\sqrt{2}}, \frac{1}{\sqrt{2}}\right)^T, \quad   e_4 \coloneqq \left(\frac{1}{\sqrt{2}}, -\frac{1}{\sqrt{2}}\right)^T.
\end{align}
Next, let $\ol{R}$ be a symmetric $2$-by-$2$ matrix. We can decompose $\ol{R}$ as follows:
\begin{align}\label{dec-r-bar}
\ol{R} = \sum_{i=1}^4 \ol{\Gamma}_i(\ol{R}) \, e_i \otimes e_i,
\end{align}
where $\ol{\Gamma}_i$ are smooth functions given by
\begin{align}
\ol{\Gamma}_1(\ol{R}) \coloneqq R_{1, 1} - R_{1, 2} - \frac{1}{2}, \quad 
\ol{\Gamma}_2(\ol{R}) \coloneqq R_{2, 2} - R_{1, 2} - \frac{1}{2}, \quad 
\ol{\Gamma}_3(\ol{R}) \coloneqq 2 R_{1, 2} + \frac{1}{2}, \quad 
\ol{\Gamma}_4(\ol{R}) \coloneqq \frac{1}{2}.
\end{align}
We notice that when $\norm{\ol{R} - I_{2 \times 2}}_{\infty} < 1/8$ then $1/4 \leq \ol{{\Gamma}}_i  \leq 2$ and that
$\max_{i, j, k} \left|\frac{\partial \ol{\Gamma}_i}{\partial \ol{R}_{j, k}}\right| \leq 2$.

Next, we let $K_{\theta_0}$ denote the rotation matrix that rotates a vector in $\mathbb{R}^2$ by an angle $\theta_0$ in the counterclockwise direction, where
\begin{align}
\theta_0 \coloneqq - \arctan (\lambda^{-1}_{q+1}).
\end{align}
We define $\xi_1, \xi_2, \xi_3, \xi_4$ to be unitary vectors of $\R^2$ with rationally dependent components as follows:
\begin{align}
\xi_i \coloneqq K_{\theta_0} \, e_i \qquad \forall \; i \in \{1, 2, 3, 4\}.
\end{align}
After writing down the explicit expression of $\xi_i$, for instance $\xi_1=(1+\lambda^{-2}_{q+1})^{-1/2} (1, \lambda^{-1}_{q+1})^T$ and $\xi_3=2^{-1/2}(1+\lambda^{-2}_{q+1})^{1/2} (1- \lambda^{-1}_{q+1}, 1+\lambda^{-1}_{q+1} )^T$ we see that the item (i) in the lemma holds. Moreover, applying \eqref{dec-r-bar} to $\bar R = K^TRK$ and then left and right multiplying both sides by $K$ and  $K^T$ respectively, for a given  $2$-by-$2$ symmetric matrix $R$, we can write
\begin{align}
R = \sum_{i=1}^4 \Gamma_i(R) \, \xi_i \otimes \xi_i,
\qquad \mbox{where }
\Gamma_i(R) \coloneqq \ol{\Gamma}_i(K_{\theta_0}^T \, R \, K_{\theta_0}).
\end{align}
We note that if $\norm{R - I_{2 \times 2}}_{\infty} < 1/16$ and $\lambda_{q+1} \geq 8$ then $\norm{K_{\theta_0}^T \, R \, K_{\theta_0} - I_{2 \times 2}}_{\infty} < 1/8$, which then implies $1/4 \leq {\Gamma}_i  \leq 2$. Moreover, 
\begin{align}
\max_{i, j, k} \left|\frac{\partial \Gamma_i}{\partial R_{j, k}}\right| \leq 4.
\label{error: a silly est}
\end{align}

Next, we define
\begin{align}
a_i(x, t) \coloneqq  \varsigma(x, t) \, \Gamma_i \left(I_{2 \times 2} - \frac{1}{\varsigma(x, t)} R_{\ell}(x, t)\right), \qquad \text{where} \quad \varsigma(x, t) \coloneqq 16 \left(|R_\ell(x, t)|^2 + \delta_{q+1}^2\right)^{\frac{1}{2}}.
\end{align}
From, here we see that
\begin{align}
\sum a_i(x, t) \xi_i \otimes \xi_i = - R_{\ell}(x, t) + \varsigma(x, t) I_{2 \times 2}.
\end{align}
Therefore, the item (ii) holds with pressure defined as $P^{\, d}(x, t) \coloneqq \varsigma(x, t) I_{2 \times 2}$. 
From the lower bounds on the coefficient $\Gamma_i$ and definition of $\varsigma$, we derive the required lower bound on $a_i$. From the upper bound on $\Gamma_i$ and a simple integration in the definition of $a_i$ gives the required estimate on $\norm{a_i}_{L^\infty_t L^1_x}$. From the upper bound on $\Gamma_i$ and on its derivatives in  (\ref{error: a silly est}), we also have
\begin{align}
\norm{a_i}_{C^{0}_{x, t}} \leq C(\delta_{q+1}+ \norm{R_\ell}_{C^{0}_{x, t}}), \qquad \norm{a_i}_{C^{1}_{x, t}} \leq C \norm{R_\ell}_{C^{1}_{x, t}}.
\end{align}


\end{proof}



\subsection{Time Series and Time Cutoffs}
\label{sec:time}

We partition $[0,+\infty)$ into time intervals of length $\tau_{q+1}$. We define,
\begin{align}
\mathcal{T}^k \coloneqq [k\tau_{q+1}, (k+1)\tau_{q+1}),
\end{align}
Each of the $\mathcal{T}^k$ intervals are further divided into four intervals of equal length as
\begin{equation}
    \mathcal{T}^k_i \coloneqq \left[\tau_{q+1} \left(k +\frac{i-1}{4}\right),\; \tau_{q+1} \left(k+ \frac{i}{4}\right)\right) \, ,
    \quad i=1, 2, 3, 4\, , \quad  k\in \N \, .
\end{equation}
It is clear that $\mathcal{T}^k= \bigcup_{i=1}^4 \mathcal{T}^k_i$. 
For future convenience, we also define a slightly shorter version of the time interval $\mathcal{T}^k$ as
\begin{align}
   \ol{\mathcal{T}}^k_i \coloneqq \left[\tau_{q+1} \left(k +\frac{i-1}{4} + \frac{1}{\lambda_{q+1}}\right),\; \tau_{q+1} \left(k+ \frac{i}{4} - \frac{1}{\lambda_{q+1}}\right)\right) \, ,
    \quad i=1, 2, 3, 4\, , \quad  k\in \N \, .
\label{iteration: short tauik}
\end{align}

\begin{definition}[Time-Average Operator]
    Given a time-dependent function $g: \R_+ \to \R$, we introduce the time-average operator 
\begin{equation}\label{eq:Ptau}
    P_{\tau}g(t) := \dashint_{\mathcal{T}^k} g(s)\, ds \, , \quad \quad t\in \mathcal{T}^k \, .
\end{equation}
\end{definition}


Given $R_\ell$ as in \eqref{defn:uRell}, we define $a_i(x,t)$ from Lemma \ref{lemma:error dec}. Next using the definition of $P_\tau$ above, we introduce a shorthand notation
\begin{equation}
    a_i^k(x):= P_{\tau_{q+1}} a_i(x,t)
    =
    \dashint_{\mathcal{T}^k} a_i(x,t)\, dt \, ,
    \quad \text{for some $t\in \mathcal{T}^k$}\, .
\end{equation}
From (\ref{eq:est_a}), we see that
\begin{align}
a_i^k(x) \ge \delta_{q+1} \, ,
\quad 
\| a_i^k \|_{L^1_x} \le 192 \delta_{q+1} \, ,
\quad 
\| a_i^k \|_{C^0_{x}} \leq C  \lambda^{2n+2 \beta \sigma}_q
\, ,
\quad 
\| a_i^k \|_{C^1_{x}}
\le C  \lambda^{3n+3 \beta \sigma}_q
\, .
\label{eq:est_aik}
\end{align}
Also, note that
\begin{equation}\label{eq:a-aik}
    \| P_{\tau_{q+1}} a_i - a_i\|_{L^\infty_t L^1_x} 
    \le \tau_{q+1} 
    \| a_i \|_{C^1_{x,t}}
    \le \tau_{q+1}  \lambda^{3n+3 \beta \sigma}_q
    \, ,
    \quad \text{ for every $t\in \mathcal{T}_i^k$} \, ,
\end{equation}
is going to be small provided $\tau_{q+1}$ is sufficiently small.

\bigskip

Given $k\in \N$ and $i\in \{1,2,3,4\}$, we define a smooth, sharp time cut-off function $\zeta^k_i : \mathbb{R} \to [0, 1]$ associated to the intervals $\mathcal{T}^k_i$ satisfying
 $\supp \zeta^k_i \subset \mathcal{T}^k_i,$
$\zeta_i^k \equiv 1$ in $\ol{\mathcal{T}}_i^k,$
and
    \begin{equation}\label{eqn:zeta-est}
        \norm{\frac{d}{dt}\zeta_i^k}_{L^\infty_t} \le 10 \frac{\lambda_{q+1}}{\tau_{q+1}}.
    \end{equation}

\subsection{Space Dependent Spatial Scale}
\label{sec:spa-scale}
To adapt the building block from Proposition \ref{prop:build-block1}, we need a spatial scale that varies with space. Subsequently, we make the following choice that will eventually allow us to cancel out the error:
\begin{align}
    r^k_i(x) \coloneqq r_{q+1} a^k_i(x) \, ,
    \quad i=1, 2, 3, 4\, , \quad k\in \N \, .
\label{def: rki}
\end{align}


We choose $r_{q+1}$ small enough such that $2 (r^k_i)^{\frac{1}{5}} \leq \lambda_{q+1}^{-1}$. From choices (\ref{eq:parameters}), \eqref{eq:est_aik}
, we impose this by requiring:
{\begin{align}
C  
 \lambda^{2n+2 \beta \sigma}_q r_{q+1} \leq \lambda^{-5}_{q+1}\, . 
\label{iteration: rki small lambda}
\end{align}
}


\subsection{The Trajectory of the Center of the Core}\label{sub:traj cent}
The cutoff function we will use in Proposition \ref{prop:build-block1} is a constant multiplication of $\zeta^k_i$, i.e.,
\begin{align}
\eta(t) = \eta^k_i \zeta^k_i(t),
\end{align}
where $\eta^k_i$ is a constant 
chosen to cancel out the error exactly:
\begin{align}\label{eqn:choiceetamagic}
    (\eta_i^k )^2 =  4 \int_{\mathbb{T}^2}a_i^k(x) \, dx
    \in [4 \delta_{q+1}, 768 \delta_{q+1}]\,
\end{align}
The estimate on the size of $\eta_i^k$ above follow from (\ref{eq:est_aik}).
We selected this value of $\eta^k_i$ at the outset of the proof. However, alternatively, we could have kept the value of $\eta^k_i$ as a free parameter and determine its value in the error cancellation in (\ref{eqn:choiceetamagic}) later on.

Next, we define the trajectory of the center of the core as
\begin{equation}\label{eq:traj1}
        \frac{d}{dt} x^k_i(t) = \frac{ \eta_i^k \zeta_i^k(t)}{r_i^k(x^k_i(t))} \xi_i \, , \quad \quad
        \quad t \in \mathcal{T}_i^k \, .
\end{equation}
We denote $t_0 = \tau_{q+1}(k + \frac{i-1}{4} + \lambda_{q+1}^{-1})$. We fix $x_i^k(t_0)$  so that
\begin{equation}
    \dashint_0^{c_i\lambda_{q+1}} a_i^k(x_i^k(t_0) + s \xi_i) \, ds = \int_{\mathbb{T}^2}a_i^k(x)\, dx =   \frac{  (\eta_i^k )^2}{4}  \, .
\end{equation}
We solve the ODE (\ref{eq:traj1}) both forward and backward in time starting at $t_0$ with position $x^k_i(t_0)$ to obtain the trajectory on the entire interval $\mathcal{T}^k_i$.
We assume
\begin{align}
\lambda_{q+1}^2 r_{q+1} \delta_{q+1}^{1/2} \le \frac{\tau_{q+1}}{200}, \qquad \text{which from (\ref{eq:parameters}) requires} \quad \mu - \beta - 2 - \kappa > 0.
\label{traj: ineq for cutoff}
\end{align}


and we observe that by (\ref{eqn:choiceetamagic}), the trajectory $x^k_i(t)$, $ t \in \mathcal{T}_i^k$, is periodic with period
\begin{equation}\label{eq:imp}
    T_i^k 
    = \int_0^{c_i\lambda_{q+1}} \frac{r_{q+1}}{\eta_i^k} a_i^k(x_i^k(t_0) + s\xi_i)\, ds
    = \frac{c_i\lambda_{q+1} r_{q+1}\eta_i^k}{4} 
\le    8 {c_i\lambda_{q+1} r_{q+1}}\delta_{q+1}^{1/2} \le \frac{\tau_{q+1}}{8 \lambda_{q+1}} \, .
\end{equation}


In particular, we see from \eqref{eq:imp} that the meaning of the assumption \eqref{traj: ineq for cutoff} is to guarantee sufficiently many periods of  $x^k_i(t)$ lie inside $ \overline {\mathcal{T}}_i^k$.
Next we call $M_i^k\in \N$, $M_i^k \ge \lambda_{q+1}$ the number of such periods, namely a natural number such that 
\begin{itemize}
    \item[(i)] $x_i^k(t_0 + m T_i^k) = x_i^k(t_0)$ for every $m\in \N$, $0 \le m \le M_i^k$,

    \item[(ii)] $t_0 + M_i^k T_i^k \le \tau_{q+1}(k + \frac{i}{4} - \lambda_{q+1}^{-1})$ and $t_0 + (M_i^k + 1) T_i^k > \tau_{q+1}(k + \frac{i}{4} - \lambda_{q+1}^{-1})$,
\end{itemize}
 Notice that (ii) can be equivalently rewritten as
 \begin{equation}
 \label{eqn:MT-tau}
 0 \leq  \tau_{q+1}\left(\frac{1}{4} - 2\lambda_{q+1}^{-1}\right) -M_i^k T_i^k < T_i^k=\frac{c_i\lambda_{q+1} r_{q+1}\eta_i^k}{4} 
 \end{equation}
 and that 
\begin{equation}\label{eq:M}
     M_i^k = \left\lfloor \frac{|\overline {\mathcal{T}}_i^k| }{T_i^k} \right\rfloor 
     = \left\lfloor \frac{\tau_{q+1}(\frac{1}{4} - \frac{2}{\lambda_{q+1}})\eta_i^k}{c_i \lambda_{q+1}r_{q+1}\int_{\mathbb{T}^2}a_i^k(x)\, dx} \right\rfloor \ge \lambda_{q+1}\, .
\end{equation}

\subsection{Principal Building Blocks of our perturbation
}\label{sec:princbb}
For every $k \in \mathbb{N}$ and $i = 1,2,3,4$, we define $V_i^k(x, t)$ to be the velocity field $V$ from Proposition \ref{prop:build-block1} applied to the spatial scale $r^k_i$ and the time cutoff $\eta^k_i \zeta^k_i(t)$, introduced in the preceding sections. By Proposition \ref{prop:build-block1}, the velocity field $V^k_i$ solves the following equations
\begin{equation}
\begin{cases}
\partial_t V_i^k + \div(V_i^k\otimes V_i^k) + \nabla P_i^k = 
 S_i^k {\frac{d}{dt} \left( \eta_i^k \zeta_i^k(t) r_i^k(t) \right)} +  \div(F_i^k)\, ,
 \quad t\in \mathcal{T}_i^k \, ,
 \\
 \div V_i^k =0 \,
 \end{cases}
 \label{iteration: BB Vki eqn}
\end{equation}
and the associated pressure $P^k_i$, source/sink term $S^k_i$, and error $F^k_i$ satisfy the following properties. All such properties follow  from  Proposition \ref{prop:build-block1}, the estimates on $r^k_i$ and $\zeta^k_i$ (\ref{iteration: rki small lambda}) and \eqref{eqn:zeta-est}, and the definition of $\eta_i^k$ in \eqref{eqn:choiceetamagic}. 
\begin{enumerate}[label = (\roman*)]
\item {}  
Since 
$
\norm{\nabla \log r^k_i}_{L^\infty_x} = \norm{({a^k_i})^{-1}{\nabla a^k_i}}_{L^\infty_x} \leq  C \delta_{q+1}^{-1} \lambda_q^{3 n+3  \beta \sigma}\leq  C \lambda_q^{3 n+4 \beta \sigma}
$
by \eqref{eq:est_aik}, we get the estimate on the error $F^k_i$:
\begin{align}
\norm{F_i^k}_{L_t^\infty L_x^1} 
&\leq C ( \eta^k_i )^2 \| r^k_i \|_{L^\infty_x}^{\frac{1}{2}}(1 + \| \nabla \log(r^k_i)\|_{L^\infty_x}) \nonumber \\
& \leq C \delta_{q+1} \, r_{q+1}^{\frac{1}{2}} \, (\lambda_q^{2 n+2 \beta \sigma})^{\frac{1}{2}} \left(\lambda_q^{3 n+4 \beta \sigma} \right) \nonumber
\\& 
= C \delta_{q+1} \, r_{q+1}^{\frac{1}{2}} \, \lambda_q^{4 n+5 \beta \sigma} \, .
\label{error: Fki}
\end{align}

\item {The $L^p$ estimate on $V^k_i$ for $p = 3/2$ and $p = 2$ are} 
\begin{align}
& \norm{V^k_i}_{L^\infty_t L^{3/2}_x} \leq C \eta^k_i \norm{r^k_i}^{1/3}_{L^\infty_x} \leq  C \delta_{q+1}^{\frac{1}{2}} r_{q+1}^{\frac{1}{3}}\left(\lambda_q^{2n+2 \beta \sigma} \right)
^{\frac{1}{3}}
\label{iteration: L32 Vki}
\\
& \norm{V^k_i}_{L^\infty_t L^{2}_x} 
\leq C \eta^k_i \leq C \delta_{q+1}^{\frac{1}{2}} .
\label{iteration: L2 Vki}
\end{align}
For $p\ge 1$, the the $L^p$ norm of $DV$ is controlled as
\begin{align}
&\norm{D V^k_i}_{L^\infty_t L^{p}_x} \leq C \eta^k_i \norm{(r^k_i)^{\frac{2}{p} - 2}}_{L^\infty_x} \leq C \delta_{q+1}^{\frac{1}{2}} r_{q+1}^{\frac{2}{p} - 2} \delta_{q+1}^{\frac{2}{p} - 2} \, .
\label{iteration: Lp DVki}
\end{align}
Finally, the $L^{p}$ estimate on $\partial_t V$ and $\partial_t DV$ reads as
\begin{align}
\norm{\partial_t V^k_i}_{L^1_t L^p_x}  + &( \delta_{q+1} r_{q+1})\| \partial_tD V^k_i \|_{L^\infty_t L^p_x}  
\\& \leq C (\eta^k_i)^2 \norm{(r^k_i)^{-3}}_{L^\infty_x} \left( 1 + \norm{\nabla r^k_i}_{L^\infty_x} \right) + C  \eta^k_i \norm{(\zeta^k_i)^\prime}_{L^\infty_t} \norm{(r^k_i)^{-1}}_{L^\infty_x} \nonumber \\
& \leq C \delta_{q+1}^{-2} r_{q+1}^{-3} (1 + r_{q+1} \lambda_q^{3n+3 \beta \sigma}  
 ) + C \delta_{q+1}^{-\frac{1}{2}} (\lambda_{q+1} \tau_{q+1}^{-1}) r_{q+1}^{-1} \nonumber \\
& \leq C \delta_{q+1}^{-2} r_{q+1}^{-3}. 
\label{iteration: Lp part Vki}
\end{align}
The reasoning behind the last inequality is as follows. From \eqref{iteration: rki small lambda}, we see that $ C r_{q+1} \lambda_q^{3n+3 \beta \sigma} \leq \lambda_{q+1}^{-5}\lambda_q^{n + \beta \sigma}$. In rest of the paper, we impose 
\begin{align}
\lambda_{q+1}^{-5}\lambda_q^{n + \beta \sigma} \leq 1.
\label{extra cond. 5}
\end{align}
Finally, from \eqref{traj: ineq for cutoff}, we see that $ C \delta_{q+1}^{-\frac{1}{2}} (\lambda_{q+1} \tau_{q+1}^{-1}) r_{q+1}^{-1} \leq  \delta_{q+1}^{-1}  r_{q+1}^{-2} \lambda_{q+1}^{-1}$ which is then controlled by $\delta_{q+1}^{-2} r_{q+1}^{-3}$.

\item {The source term $S^k_i$} satisfies for every $p\in (1,\infty)$
\begin{align}\label{est:ki}
&\| S_i^k \|_{L^\infty_t L^p_x} \le C(p) \norm{(r^k_i)^{\frac{2}{p} - 2}}_{L^\infty_x} \leq C(p) \delta_{q+1}^{\frac{2}{p} - 2} r_{q+1}^{{\frac{2}{p} - 2}} \,  
,
\end{align}
\begin{align}\label{eq:S_i^k mean}
\supp S_i^k(\cdot, t) \subseteq B_{\lambda_{q+1}^{-1}}(x_i^k(t)), \qquad   \int_{\mathbb{T}^2} S_i^k(x,t)\, {\rm d} x =   \xi_i \, .
\end{align}
   In fact, in Proposition \ref{prop:build-block1}, the average of $S_i^k$ should be $2\pi \xi_i$. However, for the sake of simplicity we normalize $V_i^k$, rescaling accordingly the time variable, namely replacing it by  $(2\pi)^{-1}V_i^k(x,(2\pi)^{-1} t)$ with a small abuse of notation, so that \eqref{eq:S_i^k mean} holds and all other properties stated above continue to hold. 
\end{enumerate}

\subsection{Auxiliary Building Block}\label{sec:auxiliary}

As discussed in Section \ref{sec:ta-errcanc}, the time-average of the source term in (\ref{iteration: BB Vki eqn}) is responsible for cancelling the error in time average. However, to produce smaller errors in the error cancellation process, we first replace the source term
\begin{align}\label{err-canc-from-here}
S_i^k \frac{d}{dt} \left( \eta_i^k \zeta_i^k(t) r_i^k(t) \right),
\end{align}
with a different term of the form 
\begin{equation}\label{eq:U}
    U_{q+1}(x,t) := \frac{d}{dt} \left( \eta_i^k \zeta_i^k(t) r_i^k(t) \right) \wt{U}_i^k(x - x_i^k(t)) \xi_i,
    \quad t \in \mathcal{T}_i^k.
\end{equation}
This new term is designed such that the anti-divergence of the difference $ S^k_i(x,t) - \wt{U}_i^k(x - x_i^k(t)) \xi_i $ is small, which ensures that the error introduced by replacing \eqref{err-canc-from-here} with \( U_{q+1} \) is small. In addition, the term $U_{q+1}$ satisfies two more properties: 

\begin{enumerate}
    \item Constant average along trajectories (see (i) below),
    
    \item Mild concentration of support: The support is concentrated on a set of size approximately \( \lambda_{q+1}^{-1} \), which is significantly larger than the support of \( S^k_i \), which is of size \( r_{q+1} \).
\end{enumerate}
These properties contribute to reducing errors in the error cancellation process.

\bigskip

 We define $\wt{U}_i^k(x) \ge 0$ a scalar function supported on a ball of radius $10 \lambda_{q+1}^{-1}$ such that
\begin{itemize}
    \item[(i)]  Space-average is one:
    \begin{equation}\label{eqn:tildeu1}
        \dashint_{\mathbb{T}^2} \wt{U}_i^k(x)\, dx = 1 \, ,
    \end{equation}

    \item[(ii)] For every $x\in \mathbb{T}^2$, it holds
     \begin{equation}\label{eqn:tildeu2}
         \dashint_0^{c_i \lambda_{q+1}} \wt{U}_i^k(x - (x_i^k(t_0) + s \xi_i))\, ds
         = 1 \, .
     \end{equation}

     \item[(iii)] For every $p\in [1,\infty]$, it holds
     \begin{equation}\label{eq:est tildeU}
         \| \wt{U}_i^k \|_{L^p} \le C(p) \lambda_{q+1}^{2 - \frac{2}{p}}\,, \qquad \text{and} \qquad \| D \wt{U}_i^k \|_{L^p} \le C(p) \lambda_{q+1}^{3 - \frac{2}{p}}\, .
     \end{equation}
\end{itemize}
To build such a scalar function $\wt{U}_i^k$ 
 we argue as follows. First, we fix any nonnegative $\Omega_0 \in C^\infty_c(B_{10}(0))$, which is bigger than $1$ in $B_5(0)$. We rescale it by $\lambda_{q+1}$ as $\Omega = \lambda_{q+1}^2 \Omega_0(\lambda_{q+1} \cdot)$, so that it is supported on a ball of radius $10 \lambda_{q+1}^{-1}$ and bigger than $ \lambda_{q+1}^{-2}$ on half of such a ball. We periodize it as a function on the torus and we define
\begin{equation}
    \overline \Omega  (x) := \dashint_0^{c_i \lambda_{q+1}} \Omega (x - (x_i^k(t_0) + s \xi_i))\, ds,
\end{equation}
which is a function invariant on the set $\{ x\in \mathbb{T}^2: x= x_i^k(t_0) + s\xi_i\}$ and bounded below by a constant $c>0$ independent of $\lambda_{q+1}$. We then set $\wt{U}^k_i = \Omega / \overline \Omega$ and observe that it satisfies \eqref{eqn:tildeu2}, which in turn implies \eqref{eqn:tildeu1} by further integrating with respect to the variable $x$.

\begin{proposition}\label{prop:auxiliary}
    Assume \eqref{traj: ineq for cutoff} given by  $\lambda_{q+1}^2 r_{q+1} \delta_{q+1}^{1/2} \le \frac{\tau_{q+1}}{200}$, let 
    $U_{q+1}$ be as in 
    \eqref{eq:U}. Then, 
   
    \begin{equation}
    \label{eqn:estU}
    \norm{U_{q+1}}_{L^\infty_t L^p_x} + \lambda_{q+1}^{-1} \norm{D U_{q+1}}_{L^\infty_t L^p_x}
     \leq C(p)  \lambda_{q+1}^{2 - \frac{2}{p}}  \norm{a^k_i}_{ C^1} 
    \end{equation}
 {and there exists a smooth symmetric tensor $G_i^k$ such that}

\begin{equation}\label{eqn:ptauu}
    P_{\tau_{q+1}} U_{q+1}(x,t) = \dashint_{\mathcal{T}^k}U_{q+1}(x,s) \, ds
   = \sum_{i=1}^4 \div\left(a_i^k(x) \xi_i \otimes \xi_i +  G_i^k(x,t) \right) \, ,
    \quad t\in \mathcal{T}^k\, .
\end{equation}
    \begin{align}\label{eq:Gik}
        \| G_i^k \|_{L^\infty_t L^1_x} 
        \le
       C \frac{\| a_i^k \|_{C^1}}{\lambda_{q+1}} \, .
    \end{align}
\end{proposition}

{
In other words, \eqref{eqn:ptauu} shows that with only a small error term $G_i^k$, $P_{\tau_{q+1}} U_{q+1}(x,t)$ exactly matches $a_i^k(x) \xi_i \otimes \xi_i$.
In light of \eqref{eq:Gik}, we make the following extra assumption on the parameters, which in particular implies that the errors $G^k_i$ is suitably  small
\begin{equation}
    \label{ineq-lambda}
    \lambda_{q+1}^{-\frac 9 {10}}\lambda_q^{3n+3\beta\sigma} \leq \delta_{q+2} \, .
\end{equation}
}

\begin{proof}[Proof of Proposition \ref{prop:auxiliary}]

From the definition \eqref{eq:U} of $U_{q+1}$ and from \eqref{eq:est tildeU}, we deduce that for every $t\in \mathcal T^k_i$
\begin{align}\label{est-u-proof}
     \norm{U_{q+1}(t)}_{L^p_x} + \lambda_{q+1}^{-1} \norm{D U_{q+1}(t)}_{ L^p_x}
 & \le \sup_{s\in \mathcal T^k_i} \Big| \frac{d}{ds}(\eta_i^k \zeta_i^k(s) r_i^k(s))\Big|
    \Big(\norm{\wt{U}_i^k}_{L^p_x} + \lambda_{q+1}^{-1} \norm{D \wt{U}_i^k}_{ L^p_x} \Big) \nonumber
\\&
 \le C(p) \lambda_{q+1}^{2 - \frac{2}{p}} \sup_{s\in \mathcal T^k_i} \Big| \frac{d}{ds}(\eta_i^k \zeta_i^k(s) r_i^k(s))\Big|
.
\end{align}
Using the definitions of $r^k_i$ and $\eta^k_i$ from (\ref{def: rki}) and (\ref{eqn:choiceetamagic}) respectively and from the estimate on the derivatives of the cutoff $\zeta_i^k$ in \eqref{eqn:zeta-est}, we upper bound the quantity inside the supremum in \eqref{est-u-proof} for every $s\in \mathcal T^k$ as 
\begin{align}
\Big| \frac{d}{ds}(\eta_i^k \zeta_i^k(s) r_i^k(s))\Big| 
&
\leq \Big| \eta_i^k r_i^k(s) \frac{d \zeta_i^k}{ds}\Big| + \Big| (\eta^k_i \zeta^k_i)^2 \frac{\nabla r^k_i}{r^k_i} \Big| 
\\&\leq 
C \delta_{q+1}^{\frac{1}{2}} r_{q+1} \norm{a^k_i}_{L^\infty_x} \, \frac{\lambda_{q+1}}{\tau_{q+1}}
+
C \delta_{q+1} \norm{\frac{\nabla a_i^k}{a_i^k}}_{L^\infty_x}
\\& \leq 
 C \norm{a^k_i}_{ C^1_x} \, .
\label{corrector: an int est}
\end{align}
{
To obtain the last inequality we used the fact that $a_i^k \ge \delta_{q+1}$ on the second term and we used the assumption \eqref{traj: ineq for cutoff} to control the first term.
}
This estimate together with \eqref{est-u-proof} yields \eqref{eqn:estU}.

To show \eqref{eqn:ptauu}, we actually prove that 
   \begin{align}
        \frac{1}{\tau_{q+1}}\int_{\mathcal{T}_i^k}U_{q+1}(x,t)\, dt
        = \div\left( a_i^k(x) \xi_i \otimes \xi_i
        + G_i^k(x)\right) \, ,
        \quad x\in \mathbb{T}^2 \, ,
    \end{align}
    and then sum over $i=1,2,3,4$.
We compute the time average of $U_{q+1}(x,t)$ using integration by parts as follows.
\begin{align}
\int_{\mathcal{T}_i^k}U_{q+1}(x,t)\, dt & = \int_{\mathcal{T}_i^k}\frac{d}{dt} \left( \eta_i^k \zeta_i^k(t) r_i^k(t) \right) \wt{U}_i^k(x - x_i^k(t)) \xi_i\, dt  \nonumber \\
& = - \int_{\mathcal{T}_i^k}  \eta_i^k \zeta_i^k(t) r_i^k(t)  \frac{d}{dt} \left( \wt{U}_i^k(x - x_i^k(t)) \right) \, \xi_i\, dt \nonumber \\
& = \int_{\mathcal{T}_i^k}  \left( \eta_i^k \zeta_i^k(t) \right)^2  \div \left( \wt{U}_i^k(x - x_i^k(t)) \xi_i \otimes \xi_i \right)\, dt \nonumber \\
& = \div \left( \xi_i \otimes \xi_i\int_{\mathcal{T}_i^k}  \left( \eta_i^k \zeta_i^k(t) \right)^2   \wt{U}_i^k(x - x_i^k(t)) \, dt \right)\, .
\label{aux bb: int U}
\end{align}

Recall the definitions of $t_0$, $T_i^k$ and $M_i^k$ from Section \ref{sub:traj cent}. We set $\wt{\mathcal{T}}_i^k = [t_0, t_0 + M_i^k T_i^k]$. From definition (\ref{iteration: short tauik}), we note that
$\wt{\mathcal{T}}_i^k \subseteq \ol{\mathcal{T}}_i^k$ and therefore 
$\zeta_i^k(t)=1$ for $t\in \wt{\mathcal{T}}_i^k$. Next, we write
\begin{align}
\int_{\mathcal{T}_i^k}&  \left( \eta_i^k \zeta_i^k(t) \right)^2   \wt{U}_i^k(x - x_i^k(t)) \, dt \nonumber \\
& =\big(\eta_i^k \big)^2\int_{\wt{\mathcal{T}}_i^k} \wt{U}_i^k(x - x_i^k(t))  \, dt
+ \big(\eta_i^k \big)^2
\int_{\mathcal{T}_i^k\setminus \wt{\mathcal{T}}_i^k}  \big( \eta_i^k \zeta_i^k(t) \big)^2  
\wt{U}_i^k(x - x_i^k(t)) \, dt \eqqcolon I + II \, .
\label{aux bb: I + II}
\end{align}
The term $II$, multiplied by $\xi_i \otimes \xi_i$, will be part of the error $G_i^k$. By \eqref{eq:est tildeU}, the $L^1$ estimate on $II$ is given by
\begin{equation}
    \| II \|_{L^1} 
    \le 
    (\eta_i^k)^2 \| \wt{U}_i^k \|_{L^1} \int_{\mathcal{T}_i^k\setminus \wt{\mathcal{T}}_i^k}
    (\zeta_i^k(t))^2\, dt
    \le 
    C \delta_{q+1} \tau_{q+1} \lambda_{q+1}^{-1} \, .
\label{aux bb: L1 est II}
\end{equation}
Next, we investigate the main term $I$.
We perform a change of variables $t \to t(s) \in \wt{\mathcal{T}}_i^k$ such that $x_i^k(t(s)) = x_i^k(t_0) + s \xi_i$, where $t(0)=t_0$. We note that
\begin{equation}
    \frac{d}{ds} t(s) = \frac{r_{q+1}}{\eta_i^k \zeta_i^k(t(s))} a_i^k(x_i^k(t_0) + s \xi_i)\, .
\end{equation}
Employing the change of variables, we compute the term $I$ as follows:
\begin{align}
I & = \eta_i^k r_{q+1} \int_0^{c_i M_i^k \lambda_{q+1}} (a_i^k \wt{U}_i^k)(x - (x_i^k(t_0) + s\xi_i)) \,  ds \nonumber \\
& =\eta_i^k r_{q+1} a_i^k(x) \int_0^{c_i M_i^k \lambda_{q+1}} \wt{U}_i^k(x - (x_i^k(t_0) + s\xi_i)) \,  ds \nonumber \\
& + \eta_i^k r_{q+1} \int_0^{c_i M_i^k \lambda_{q+1}} (a_i^k(x - (x_i^k(t_0) + s\xi_i) - a_i^k(x)) \wt{U}_i^k(x - (x_i^k(t_0) + s\xi_i)) \,  ds \nonumber \\
& \eqqcolon I' + II'
\label{aux bb: I + II prime}
\end{align}
where $I'$ is the main term and $II''$ will be part of the error $G^k_i$. By \eqref{eqn:tildeu2}, We rewrite the main term  
\begin{align}
I' & = \eta_i^k r_{q+1} a_i^k(x) \int_0^{c_i M_i^k \lambda_{q+1}} \wt{U}_i^k(x - (x_i^k(t_0) + s\xi_i)) \, ds, = \eta_i^k r_{q+1} c_i M_i^k \lambda_{q+1} a_i^k(x) = (\tau_{q+1} + E) a_i^k(x),
\label{aux bb: def I prime}
\end{align}
where using \eqref{eqn:MT-tau} and the formula and the estimate for the period in \eqref{eq:imp}, 
the error $E$ is estimated by
\begin{align}
&  |E| \leq 2\frac{\tau_{q+1}}{\lambda_{q+1}} + \frac{c_i\lambda_{q+1} r_{q+1}\eta_i^k}{4}  \leq C \frac{\tau_{q+1}}{\lambda_{q+1}}  + C \delta_{q+1}^{1/2} r_{q+1} \lambda_{q+1} \leq C \frac{\tau_{q+1}}{\lambda_{q+1}}.
\label{aux bb: est E}
\end{align}
Now, we estimate the term $II^\prime$ from (\ref{aux bb: I + II prime}) as follows
\begin{align}
    \| II' \|_{L^1} \le C \eta_i^k r_{q+1} c_i M_i^k \lambda_{q+1}  \frac{\| a_i^k \|_{C^1}}{\lambda_{q+1}} \leq C \left(\tau_{q+1} + E\right)\frac{\| a_i^k \|_{C^1}}{\lambda_{q+1}} \leq C \tau_{q+1}\frac{\| a_i^k \|_{C^1}}{\lambda_{q+1}}.
\label{aux bb: L1 est II prime}
\end{align}
Finally, combining (\ref{aux bb: int U}), (\ref{aux bb: I + II}), (\ref{aux bb: I + II prime}) and (\ref{aux bb: def I prime}), we obtain
\begin{align}
\frac{1}{\tau_{q+1}}\int_{\mathcal{T}_i^k}U_{q+1}(x,t)\, dt = \div\left( a_i^k(x) \xi_i \otimes \xi_i
        + G_i^k(x)\right), \quad \text{where} \quad G_i^k = \frac{1}{\tau_{q+1}} \left(II + II^\prime + E \, a_i^k(x)\right)  \xi_i \otimes \xi_i .
\end{align}
Combining the estimates (\ref{aux bb: L1 est II}), (\ref{aux bb: est E}), and (\ref{aux bb: L1 est II prime}), we obtain
\begin{align}
\norm{G^k_i}_{L^1} \leq C \frac{\delta_{q+1}}{\lambda_{q+1}} + C \frac{\| a_i^k \|_{C^1}}{\lambda_{q+1}} + C \frac{\| a_i^k \|_{L^1}}{\lambda_{q+1}} \leq C \frac{\| a_i^k \|_{C^1}}{\lambda_{q+1}}. 
\end{align}
\end{proof}

\subsection{The Time Corrector}\label{sec:timecorr}
As stated earlier, in our construction, it is the time average of $U_{q+1}$, which we called $P_{\tau_{q+1}} U_{q+1}$, that cancels the error. Therefore, to compensate for the difference $U_{q+1} - P_{\tau_{q+1}} U_{q+1}$, we define a time corrector, a method used in other contexts (see, for example, \cite{BV19, cheskidov2022sharp}). We define our time corrector $Q_{q+1}: \mathbb{T}^2 \times [0, 1] \to \mathbb{R}^2$ as 
\begin{align}\label{eq:Q}
    - Q_{q+1}(x,t) \coloneqq \mathbb{P}\left( \int_0^t (U_{q+1}(x,s) - P_{\tau_{q+1}} U_{q+1}(x,s))\, ds 
    \right) \, ,
\end{align}
where $P_{\tau_{q+1}}$ is as defined in (\ref{eq:Ptau}). Let $t \in \mathcal{T}^k$ for some $k \in \mathbb{N}$. The crucial observation is that the integral of $U_{q+1}(x,s) - P_{\tau_{q+1}} U_{q+1}(x,s)$ vanishes when computed on any time interval of the form $[k\tau_{q+1}, (k+1)\tau_{q+1}]$, $k\in \N$, by definition of $\tau_{q+1}$-average. Hence, length of interval contributing towards the integral in \eqref{eq:Q} is of length less than $\tau_{q+1}$. 
\begin{equation}\label{eq:Qid2}
    -Q_{q+1}(x,t) = \mathbb P \int_{k \tau_{q+1}}^t \left(U_{q+1}(x,s) - \dashint_{\mathcal{T}^k} U_{q+1}(x, s')\, ds' \right)\, ds. \,
\end{equation}

We easily estimate the norms of $Q_{q+1}$ in terms of the norms of $U_{q+1}$, which were computed in \eqref{eqn:estU}. By \eqref{eq:Qid2} and the Calderon--Zygmund estimates applied to $\mathbb P$, we get
\begin{align}\label{eqn:Qest}
    \norm{Q_{q+1}}_{L^\infty_t L^p_x} + \lambda_{q+1}^{-1} \norm{D Q_{q+1}}_{L^\infty_t L^p_x} &\le C\int_{k \tau_{q+1}}^t ( \norm{U_{q+1}}_{L^\infty_t L^p_x} +  \lambda_{q+1}^{-1} \norm{DU_{q+1}}_{L^\infty_t L^p_x}) \nonumber
\\&    \le C\tau_{q+1} (
    \norm{U_{q+1}}_{L^\infty_t L^p_x} +  \lambda_{q+1}^{-1} \norm{DU_{q+1}}_{L^\infty_t L^p_x})
    \end{align}
 for any $p \in (1, \infty)$ and any $t\in \mathcal{T}^k$. From (\ref{eq:Q}), we also see that
\begin{align}
\norm{\partial_t Q_{q+1}}_{L^\infty_t L^p_x}  =\norm{ - \mathbb{P}U_{q+1}(x,t) +  \dashint_{\mathcal{T}^k} \mathbb{P} U_{q+1}(x,s)\, ds \, }_{L^\infty_t L^p_x}     \leq 2 \norm{\mathbb{P}U_{q+1}}_{L^\infty_t L^p_x}  \leq 2 \norm{U_{q+1}}_{L^\infty_t L^p_x} 
,
\label{corrector: part Q Lp no space deriv}
\end{align}
and analogously
\begin{align}
\norm{\partial_t DQ_{q+1}}_{L^\infty_t L^p_x}   \leq 2 \norm{D\mathbb{P}U_{q+1}}_{L^\infty_t L^p_x}  \leq 2 \norm{DU_{q+1}}_{L^\infty_t L^p_x} 
.
\label{corrector: part Q Lp}
\end{align}

\subsection{The Perturbation and the New Reynolds Stress
}
\label{sec:velocity perturbation}
We define the velocity field $u_{q+1}$ as
\begin{equation}
    u_{q+1}(x,t) \coloneqq u_\ell(x,t) + v_{q+1}(x,t) + Q_{q+1}(x,t) \, ,
    \quad\quad 
    x\in \mathbb{T}^2\, , \, \, t\ge 0\, .
\label{def: uq+1}
\end{equation}
where the velocity field $v_{q+1}$ is the main perturbation and is given by
\begin{equation}
    v_{q+1}(x,t) = \sum_{k\in \N} \sum_{i=1}^4 V_i^k(x,t) \,.
\end{equation}
The velocity field $V^k_i$ and $Q_{q+1}$ are the adapted building block and time corrector from the previous section. 

{
\begin{remark}\label{rmk:control ending points}
    For every $k\in \mathbb{Z}$, it follows that $v_{q+1}(x, k \tau_{q+1}) = Q_{q+1}(x,k \tau_{q+1}) = 0$ due to the time cut-off in the definition of $V_i^k$ and \eqref{eq:Qid2}. This implies that, 
    \begin{equation}
        u_{q+1}(x,k\tau_{q+1}) = u_\ell(x,k\tau_{q+1}) \, ,
        \quad \text{for every $k\in \mathbb{Z}$ and $x\in \mathbb{T}^2$}\, .
    \end{equation}
    Since $\tau_{q+1}^{-1}$ is integer, by \eqref{uell-u} for every $k\in \mathbb{Z}$ we conclude that
    \begin{equation}
        \norm{u_{q+1}(\cdot, k) - u_q(\cdot, k)}_{L^2} 
        \le \| u_{q} - u_\ell \|_{L^\infty_t L^2_x}
        \le C  \lambda_1^{\beta}\lambda_0^{-\beta/2} \lambda_{q+1}^{-\beta}
    \end{equation}
\end{remark}
}

We define the new pressure field as
\begin{align}
p_{q+1} \coloneqq p_{\ell} + P^{\, d} + \sum_{k\in \N} \sum_{i=1}^4 P_i^k,
\end{align}
where $p_\ell$, $P^{\, d}$ and $P_i^k$ are from (\ref{mollified eqn}), (\ref{eq:error dec}) and (\ref{iteration: BB Vki eqn}), respectively. The velocity field $u_{q+1}$ and the pressure $p_{q+1}$ satisfies the Euler--Reynolds system with the error term given by 
\begin{equation}\label{eq:Rq+1 dec}
    R_{q+1} \coloneqq 
    R_{q+1}^{(l)} 
    +
    R_{q+1}^{(c)}
    +
    R_{q+1}^{(t)}
    +
    R_{q+1}^{(s)}
    +
    \sum_{k \in \mathbb{N}} \; \sum_{i= 1}^{4} (F_{i}^k + G_i^k)
\end{equation}
where the error $F_i^k$ is from (\ref{iteration: BB Vki eqn}) and $G_i^k$ is given by Proposition \ref{prop:auxiliary}. The error $R^{(l)}_{q+1}$ represents the terms that are linear in the perturbation and $R_{q+1}^{(c)}$ contains the terms involving the time corrector $Q_{q+1}$:
\begin{align}
R_{q+1}^{(l)} 
& \coloneqq v_{q+1} \otimes u_\ell + u_\ell \otimes v_{q+1} \, , \\
R_{q+1}^{(c)} & \coloneqq  Q_{q+1} \otimes (u_\ell + v_{q+1}) + (u_\ell + v_{q+1}) \otimes Q_{q+1}  
+ Q_{q+1} \otimes Q_{q+1} \, ,       
\end{align}
Finally, for every $t\in \mathcal{T}_i^k$ we define the error $R_{q+1}^{(t)}$ coming from freezing the coefficients in time and the error $R_{q+1}^{(s)} $ coming from replacing the source term with the auxiliary building block
\begin{align}
R_{q+1}^{(t)}(x,t) & \coloneqq a_i(x,t) -  a_i^k(x)
 \\
R_{q+1}^{(s)}(x,t) & \coloneqq \mathcal{R}_0\left( (S_i^k- \wt U_i^k) \frac{d}{dt}(\eta_i^k \zeta_i^k r_i^k ) \right)(x,t) .
\end{align}

 \section{Estimates on the Perturbation and on the Reynolds Stress}\label{sec:est-pert-rey}
The goal of this section is to provide a proof of Proposition \ref{prop:iteration}. Essentially, given a solution $(u_{q}, p_{q}, R_{q})$ of the Euler-Reynolds system at the $q$th stage, our aim is to construct a solution $(u_{q+1}, p_{q+1}, R_{q+1})$ at the $(q+1)$th stage that satisfies the conditions stated in the proposition. To that end, we express the estimates on the velocity and the error and express these estimates as powers of $\lambda_{q+1}$, where the exponents are determined by the constants $\beta$, $\mu$, $\kappa$, $\sigma$, and $n$ from Section \ref{sec: iteration}. Finally, we ascertain the values of these constants to establish a proof of the Proposition \ref{prop:iteration}.

\subsection{Estimate on the Velocity Field}
We begin with estimating the $L^2$ norm of $u_{q+1} - u_q$. From the definition of $u_{q+1}$ given in (\ref{def: uq+1}) and estimates \eqref{eq:est_aik},(\ref{iteration: L2 Vki}), \eqref{eqn:estU}, (\ref{eqn:Qest}),  we obtain
\begin{align}
\| u_{q+1} -u_q\|_{L^\infty_t L^2_x}
& \le \| u_\ell -u_q\|_{L^\infty_tL^2_x} + \sup_{i,k}\| V_i^k \|_{L^\infty_t L^2_x} + \| Q_{q+1}\|_{L^\infty_t L^2_x} \nonumber \\
& \le 
 C\delta_{q+1}^{1/2} + C \tau_{q+1} \lambda_{q+1}\lambda_q^{3n+3\beta\sigma}
\nonumber \\
& \le  C\delta_{q+1}^{1/2}, 
\label{eq:uq+1 estL2}
\end{align}
where we impose the more restrictive condition
\begin{equation}
    \label{eqn:condQ}
   \lambda_{q+1}^{-\kappa+1+\frac{4n}{\sigma}+3\beta+ \beta \sigma}\leq  1
    \quad\Rightarrow \quad \delta_0^{1/2}\tau_{q+1} \lambda_{q+1}\lambda_q^{4n+3\beta\sigma} \leq \delta_{q+2}.
\end{equation}


The second inductive assumption follows from the previous computation
\begin{equation*}
\| u_{q+1} \|_{L^\infty_t L^2_x} \leq \| u_q \|_{L^\infty_t L^2_x}+ \|u_{q+1} - u_q\|_{L^\infty_tL^2_x}  \leq  2 \delta_0^{1/2}- \delta_{q}^{1/2}+  M \delta_{q+1}^{1/2} \leq  2 \delta_0^{1/2}- \delta_{q+1}^{1/2}
\end{equation*}
provided $\lambda_0$ is sufficiently large in terms of $M$.

Now we obtain the $L^p$ estimate on $D u_{q+1}$. From \eqref{eq:est_aik}, (\ref{iteration: Lp DVki}), \eqref{eqn:estU} and (\ref{eqn:Qest}) for any $p\in (1,\infty)$, we get
\begin{align}
& \| D u_{q+1}\|_{L^\infty_t L^p_x}  
 \le \| D u_\ell \|_{L^\infty_t L^p_x} + \sup_{i,k} \| D V_i^k \|_{L^\infty_t L^p_x} + \| D Q_{q+1} \|_{L^\infty_t L^p_x}
  \nonumber \\
 & \le \| D u_q \|_{L^\infty_t L^p_x} + C(p) \delta_{q+1}^{\frac{1}{2}} \left(r_{q+1} \delta_{q+1}\right) ^{\frac{2}{p} - 2} 
 + C(p)  \tau_{q+1} \lambda_{q+1}^{3-\frac{2}{p} } 
\lambda_q^{3n+3\beta \sigma}
 \nonumber \\
 & \le \| D u_q \|_{L^\infty_t L^p_x} + C(p)\delta_{q+1}^{1/10} \lambda_1^{ \beta} \left(  \lambda_{q+1}^{-\beta \left(-\frac{8}{5} + \frac{2}{p}\right) + \mu  \left(2 - \frac{2}{p}\right) }
 +
 \lambda_{q+1}^{\frac{3n}{\sigma} + 5 \beta -\kappa + 3 - \frac{2}{p} }
  \right)\, .
\label{error: Duq+1 estLp}
\end{align}
We will impose \eqref{z10}
so that $ \| D u_{q+1}\|_{L^\infty_t L^{\overline{p}}_x}
 \le  \| D u_{q}\|_{L^\infty_t L^{\overline{p}}_x} + \lambda_1^{\beta} \delta_{q+1}^{1/10}$, which then satisfies item (ii) in Proposition \ref{prop:iteration} as we will ensure $\beta \sigma < 1$. 

By using a simple interpolation inequality, estimates from \eqref{error: Duq+1 estLp} together with 
\eqref{eq:est_aik}, 
\eqref{iteration: Lp part Vki}, \eqref{eqn:estU} and (\ref{corrector: part Q Lp}),  we obtain
\begin{align}
\|&D u_{q+1}\|_{C^\alpha_t L^{\overline{p}}_x} \nonumber 
\\&\leq \|  D u_q\|_{C^\alpha_t L^{\overline{p}}_x} +  \sup_{i,k} \|  D V_i^k\|_{C^\alpha_t L^{\overline{p}}_x} +\| DQ_{q+1}\|_{C^\alpha_t L^{\overline{p}}_x} \nonumber 
\\&
\leq  \|  D u_q\|_{C^\alpha_t L^{\overline{p}}_x} +  \sup_{i,k} \|  D V_i^k\|_{L^\infty_t L^{\overline{p}}_x}^{1-\alpha} \|  \partial_t D V_i^k\|_{L^\infty_t L^{\overline{p}}_x}^{\alpha} +\| DQ_{q+1}\|_{L^\infty_t L^{\overline{p}}_x}^{1-\alpha}  \| \partial_t DQ_{q+1}\|_{L^\infty_t L^{\overline{p}}_x}^{\alpha} \nonumber
\\&
\leq \|  D u_q\|_{C^\alpha_t L^{\overline{p}}_x} + 
(C\delta_{q+1}^{\frac{1}{2}} \left(r_{q+1} \delta_{q+1}\right) ^{\frac{2}{{\overline{p}}} - 2} )^{1-\alpha} (C\delta_{q+1}^{-3}r_{q+1}^{-4})^{\alpha} 
+
C
\tau_{q+1}^{1-\alpha}  (\lambda_{q+1}^{3-\frac{2}{{\overline{p}}} } 
\lambda_q^{3n+3\beta}) \nonumber
\\&
\leq \|  D u_q\|_{C^\alpha_t L^{\overline{p}}_x} + C \delta_{q+1}^{1/100}\delta_{q+1}^{9/100}\Big( (\delta_{q+1}^{\frac{1}{2}} \left(r_{q+1} \delta_{q+1}\right) ^{\frac{2}{{\overline{p}}} - 2} )^{-\alpha} (\delta_{q+1}^{-3} r_{q+1}^{-4})^\alpha +\tau_{q+1}^{-\alpha}
\Big) \nonumber \\
&
\leq \|  D u_q\|_{C^\alpha_t L^{\overline{p}}_x} + C \delta_{q+1}^{1/100} \left[\delta_{q+1}^{9/100}\Big(
 \delta_{q+1}^{-\frac{2}{\overline{p}} \alpha - \frac{3}{2} \alpha} r_{q+1}^{-5 \alpha}
+ \tau_{q+1}^{-\alpha} \Big) \right]
\label{holder est in time}
\end{align}
The factor multiplying $\delta_{q+1}^{1/100}$ in (\ref{holder est in time}) is bounded by $1$  provided $\alpha$ is chosen sufficiently close to $0$ in terms of the parameters.

Finally, we obtain $L^\infty_t L^p_x$ estimate on $\partial_t u_{q+1}$. We note from \eqref{eq:est_aik}, (\ref{iteration: Lp part Vki}),  \eqref{eqn:estU}, \eqref{corrector: part Q Lp no space deriv} that
\begin{align}
 \| \partial_t u_{q+1}\|_{L^\infty_t L^p_x}  
& \le \| \partial_t u_\ell \|_{L^\infty_t L^p_x} + \sup_{i,k} \| \partial_t V_i^k \|_{L^\infty_t L^p_x} + \| \partial_t Q_{q+1} \|_{L^\infty_t L^p_x} \nonumber \\
& \leq \lambda_{q+1}^{\frac{n}{\sigma}} + C \delta_{q+1}^{-2} r_{q+1}^{-3} +C \lambda_{q+1}^{2-\frac{2}{p} } 
\lambda_q^{3n+3\beta \sigma} \nonumber \\
& \leq \lambda_{q+1}^{\frac{n}{\sigma}} + C \lambda_{q+1}^{2\beta + 3 \mu}  + C \lambda_{q+1}^{\frac{3n}{\sigma} + 3 \beta + 2-\frac{2}{p} }. 
\label{error: part t uq+1 estLp}
\end{align}

\subsection{Estimate on the Error}
\label{subsec:estimate error}
From  \eqref{eq:iterate1} and \eqref{iteration: L32 Vki}, we deduce
\begin{align}
\| R_{q+1}^{(l)}(\cdot, t)\|_{L^\infty_t L_x^1} 
& \le C \| v_{q+1}\|_{L^\infty_t L_x^{3/2}} \| u_\ell\|_{L^\infty_t L_x^3}
\leq C (\sup_{i, k} \| V^k_i \|_{L^\infty_t L^{3/2}_x} +\| Q_{q+1}\|_{L^\infty_t L^2_x} ) \, \| u_\ell\|_{L^\infty_t L_x^4}
\\&
\leq C \delta_{q+1}^{\frac{1}{2}} r_{q+1}^{\frac{1}{3}}\left(\lambda_q^{2n+2 \beta \sigma}\right)
^{\frac{1}{3}}\lambda_q^n+C \tau_{q+1} \lambda_{q+1}\lambda_q^{4n+3\beta\sigma} \delta_0^{1/2}\leq \frac{1}{10} \delta_{q+2}
,
\label{error: linear error}
\end{align}
where in the last inequality we used the condition on the parameters \eqref{iteration: rki small lambda}, \eqref{eqn:condQ} and $\lambda_{q+1} \leq \lambda_q^n$ which follows from \eqref{ineq-lambda}. Here as well, one can consume the constants by choosing $\lambda_0$ large. From \eqref{eq:est_aik}, \eqref{eqn:estU}, (\ref{eqn:Qest}), \eqref{eqn:condQ}, it follows:
\begin{align}
\| R_{q+1}^{(c)} \|_{L^\infty_t L^1_x}
&\le C \| Q_{q+1} \|_{L^\infty_t L_x^2} (  \| u_\ell \|_{L^\infty_t L^2_x} +  \| v_{q+1} \|_{L^\infty_t L^2_x} + \| Q_{q+1} \|_{L^\infty_t L^2_x}) \nonumber  \\
& \le C \tau_{q+1} \lambda_{q+1}\lambda_q^{3n+3\beta\sigma} \delta_0^{1/2}\leq \frac{1}{10}  \delta_{q+2}.
\label{error: corrector error}
\end{align}

From \eqref{eq:a-aik}, we deduce that
\begin{equation}\label{eq:Rt}
    \| R_{q+1}^{(t)}\|_{L^\infty_t L^1_x} 
    \le \tau_{q+1} 
    \| a_i \|_{C^1_{x,t}}
    \le C \tau_{q+1}  \lambda^{3n+3 \beta \sigma}_q
    \le  \frac{1}{10}  \delta_{q+2}, 
\end{equation}
where in the last inequality we used the condition on the parameters \eqref{eqn:condQ}. 
Moreover, from Proposition~\ref{prop:antidiv compact}, \eqref{est:ki}, \eqref{eq:est tildeU} (notice that the estimate for the $L^\infty_t L^p_x$ of $S_i^k$ is worse than that of $\tilde U_i^k$, as expected since they are both normalized in $L^1$ but the former is more concentrated) and \eqref{corrector: an int est}, for every $p\in (1,\infty)$, we see that
\begin{align}
&\norm{R_{q+1}^{(s)}}_{L^\infty_t L^p_x} \leq C(p) \lambda_{q+1}^{-1}
     \left( \| S_i^k\|_{L^\infty_t L^p_x} + \|\tilde U_i^k\|_{L^\infty_tL^p_x}\right) 
     \sup_{i,k}\Big| \frac{d}{dt}(\eta_i^k \zeta_i^k r_i^k(x_i^k))\Big|
     \\&
     \leq C(p) \lambda_{q+1}^{-1}\delta_{q+1}^{\frac{2}{p} - 2} r_{q+1}^{{\frac{2}{p} - 2}} \lambda^{3n+3 \beta \sigma}_q \,   \leq  \frac{\delta_{q+2}}{10} \left(C(p)\lambda_1^{- 4 \beta + \frac{4 \beta}{p}}\lambda_{q+1}^{-\frac 1 {10}+ \big(2-\frac{2}{p}\big)(\beta+\mu)} \right)
\label{error: source error}
\end{align}
and the factor multiplying $\delta_{q+2}$ is bounded by $\frac{1}{10}$ provided $p$ is chosen sufficiently close to $1$, and $\lambda_0$ is big enough.

Finally by (\ref{error: Fki}), (\ref{eq:Gik}), (\ref{iteration: rki small lambda}) and \eqref{ineq-lambda} we get
\begin{align}
&\norm{F_{i}^k}_{L^\infty_t L^1_x} \leq
C \delta_{q+1} \, r_{q+1}^{\frac{1}{2}} \, \lambda_q^{4 n+5 \beta \sigma} 
\leq \frac{1}{10}  \delta_{q+2},
\\
& \norm{G_i^k}_{L^\infty_t L^1_x} \leq \lambda_{q+1}^{-1} \lambda^{3n+3 \beta \sigma}_q
\leq \frac{1}{10}  \delta_{q+2} \, .
\end{align}


\subsection{Constraints on the Parameters}
\label{sec:parameters}
Need to take care of extra conditions now.
The following constraints have been imposed on the parameters, and are in turn implied by the inequalities below:
\begin{enumerate}[label = (\roman*)]
\item We satisfy the extra conditions (\ref{iteration: rki small lambda}) and (\ref{traj: ineq for cutoff}), 
\eqref{ineq-lambda}, \eqref{eqn:condQ}, (\ref{extra cond. 5}):
\begin{align}
5+ \frac{2n}{\sigma}+2 \beta-\mu < 0\, ,
\qquad 
\mu -\beta - 2 - \kappa > 0 \, , \qquad \qquad \qquad \qquad \\ 
 -\frac{9}{10} + \frac{3n}{\sigma} + 3 \beta + \beta \sigma < 0\, , 
\qquad  {-\kappa+1+\frac{4n}{\sigma}+3\beta+ \beta \sigma} < 0\, , \qquad -5 + \frac{n}{\sigma} + \beta < 0.
\end{align}

\item The terms other than $\norm{D u_q}_{L^\infty_t L^{\bar p}_x}$ in (\ref{error: Duq+1 estLp}) are smaller than $\delta_{q+1}^{\frac{1}{10}}$:
\begin{align}\label{z10}
-\beta \left(-\frac{8}{5} + \frac{2}{p}\right) + \mu  \left(2 - \frac{2}{\bar p}\right) < 0,
\qquad
{\frac{3n}{\sigma} + \frac{9 \beta}{10} -\kappa + 3 - \frac{2}{\bar p} } < 0.
\end{align}

\item Finally, $\norm{D_{x, t} u_{q+1}}_{L^\infty_t L^4_x}$ is smaller than $\lambda_{q+1}^n$ using \eqref{error: Duq+1 estLp} and \eqref{error: part t uq+1 estLp} for $p=4$:
\begin{align}
\frac{21 \beta}{10} + \frac{3}{2} \mu  < n,
\qquad
{\frac{3n}{\sigma} + 6 \beta -\kappa + \frac{5}{2} } < n,
\qquad
2\beta + 3 \mu < n,
\qquad
\frac{3n}{\sigma} + 3 \beta +  \frac{3}{2} < n \, .
\end{align}
\end{enumerate}
These conditions are satisfied, for example, when
\begin{align}
\beta = \frac{1}{245}, \quad \mu = \frac{53}{10}, \quad \kappa = 3, \quad n = 16, \quad \sigma = 110, \quad \bar p = 1 + \frac{1}{6500}.
\end{align}
Once this choice is made, we see that $\alpha = \frac{1}{10^5}$ satisfies \eqref{error: source error}.
\begin{remark}[Improved exponent]
It is likely that the exponent in our scheme can be increased, for instance, to $\ol{p} = 1 + \frac{1}{2000}$. Two changes required to get this exponent are as follows. Firstly, one needs to improve the estimate on the support of the source/sink term in Proposition \ref{prop:build-block1} by choosing $\alpha = 1/3$ in the proof. Secondly, in equation (\ref{traj: ineq for cutoff}), one should instead impose $\lambda_{q+1}^{3/2} r_{q+1} \delta_{q+1}^{1/2} \le \frac{\tau_{q+1}}{200}$.
\end{remark}

\appendix

\begin{center}
{\bf APPENDIX}
\end{center}

\section{Anti-Divergence Operator}
\label{appendix:antidiv}

\subsection{Bogovskii operator}
Let us fix a cut-off function $\gamma\in C^\infty_c(\R^2)$ satisfying
\begin{itemize}
    \item[(i)] $\supp \gamma \subset B_1(0)$,

    \item[(ii)]  $\int_{\R^2} \gamma(x) \, dx =1$.
\end{itemize}
For every $f\in L^1(\R^2)$ with $\int_{\R^2} f(x)\, dx =0$, we define the Bogovskii operator
\begin{equation}
    \mathcal{B}(f)(x) := \int_{\R^2} f(y) (x-y) \left(\int_1^\infty \gamma(y + \rho(x-y)) \rho \, d \rho \right) \,dy
\end{equation}
The following Lemma is well-known (see \cite{galdi2011introduction}).

\begin{lemma}\label{lemma:BogovskiiI}
    Assume that $f\in C^\infty_c(\R^2)$ is supported on $B_1(0)$ and $\int_{\R^2} f(x)\, dx =0$. Then, $\mathcal{B}(f)\in C^\infty_c(\R^2; \R^2)$ is supported on $B_1(0)$ and satisfies
    \begin{itemize}
        \item[(a)] $\div(\mathcal{B}(f))=f$.

        \item[(b)] For every $p\in (1,\infty)$, it holds
        \begin{equation}
        \|\mathcal{B}(f)\|_{L^p}
        \le C(p) 
        \| f \|_{L^p} \, . 
        \end{equation}
    \end{itemize}
\end{lemma}
Given a function $f\in C^\infty_c(\R^2)$ with zero mean supported on $B_r(x_0)$, we define
\begin{equation}
    v(x) = r \mathcal{B}(f( x_0+r \cdot))\left(\frac{x-x_0}{r}\right) \, .
\end{equation}
By Lemma \ref{lemma:BogovskiiI} applied to $f( x_0+r \cdot) \in C^\infty_c(B_1)$, we know that $\supp v \subset B_r(x_0)$, $\div (v) = f$, and
\begin{equation}
    \| v \|_{L^p} \le C(p) r \| f \|_{L^p} \, ,
    \quad \text{for every $p\in (1,\infty)$}\, .
\end{equation}
By applying the previous argument to velocity fields we deduce the following.

\begin{proposition}[Compactly Supported Anti-divergence]\label{prop:Bogovskii2}
    Assume that $v\in C^\infty(\R^2;\R^2)$ is supported in $B_r(x_0)$ and $\int_{\R^2} v(x) \, dx =0$. Then there exists $A\in C^\infty(\R^2; \R^{2\times 2})$ such that
    \begin{itemize}
        \item[(a)] $\supp A \subset B_r(x_0)$,

        \item[(b)] 
        $\div (A) = v$,
        
        \item[(c)] for every $p\in (1,\infty)$ it holds
        \begin{equation}
            \| A \|_{L^p} \le C(p) r \| v \|_{L^p} \, .
        \end{equation}
    \end{itemize}
\end{proposition}

\begin{remark}[A necessary condition for the symmetry of $A$]
    The tensor $A$ built in Proposition \ref{prop:Bogovskii2} is not necessarily symmetric. A necessary condition for symmetry is 
    \begin{align}
    \label{Appendix A: extra cond. for symmetry}
        \int_{\R^2} (x_1 v_2(x) - x_2 v_1(x)) \, dx = 0 \, .
    \end{align}
\end{remark}

\subsection{Symmetric Anti-divergence}

On the torus $\mathbb{T}^2$, we consider the operator
\begin{equation}\label{eq:R0}
    \mathcal{R}_0(v) = (D \Delta^{-1} + (D \Delta^{-1})^t - I \cdot \div \Delta^{-1})(v)
\end{equation}
for every $v\in C^\infty(\mathbb{T}^2;\R^2)$ such that $\int_{\mathbb{T}^2} v(x)\, dx = 0$.
It turns out that 
\begin{equation}
   \mathcal{R}_0 : C^\infty(\mathbb{T}^2; \R^2) \to C^\infty (\mathbb{T}^2; {\rm Sym}_2) \, ,  
\end{equation}
where ${\rm Sym}_2$ is the space of symmetric tensors in $\R^2$. It is immediate to check that $\div(\mathcal{R}_0(v)) = v$ and that $D \mathcal{R}_0$ and $\mathcal{R}_0 \div$ are Calderon-Zygmund operators. In particular, the following estimates hold for every $p\in (1,\infty)$, $v\in C^\infty(\mathbb{T}^2; \R^{2})$, $A\in C^\infty(\mathbb{T}^2; \R^{2\times 2})$
\begin{align}
    \| \mathcal{R}_0(v) \|_{L^p} \le C(p) \| D \mathcal{R}_0(v) \|_{L^p} \le C(p) \| v \|_{L^p} \, , \quad 
    \| \mathcal{R}_0 \div (A) \|_{L^p} \le C(p) \| A \|_{L^p} .\label{eq:R0est2}
\end{align}
As a consequence of \eqref{eq:R0est2}, we have the following.

\begin{proposition}[Symmetric Anti-divergence of compactly supported vector fields on the torus]
\label{prop:antidiv compact}
Let $0<r<1/4$. Assume that $v\in C^\infty_c(\mathbb{T}^2;\R^2)$ is supported on an ball or radius $r$ and $\int_{\mathbb{T}^2}v(x)\, dx =0$. Then,
\begin{equation}
    \|\mathcal{R}_0(v)\|_{L^p} \le C(p) r \| v \|_{L^p} \, ,
    \quad \text{for every $p\in (1,\infty)$}\, .
\end{equation}  
\end{proposition}

\begin{proof}
    We first identify $v$ with a velocity field on $\R^2$ supported on a ball of radius $r$ contained in $[0,1]^2$. We invert the divergence by means on the Bogovskii operator as in Proposition \ref{prop:Bogovskii2}, obtaining a compactly supported tensor $A$. We then periodize $A$ and apply \eqref{eq:R0est2}.
\end{proof}







\bibliographystyle{halpha-abbrv}
\bibliography{references.bib}

\newcommand{\etalchar}[1]{$^{#1}$}
\begin{thebibliography}{LFMNL06}
\expandafter\ifx\csname url\endcsname\relax
  \def\url#1{\texttt{#1}}\fi
\expandafter\ifx\csname doi\endcsname\relax
  \def\doi#1{\burlalt{doi:#1}{http://dx.doi.org/#1}}\fi
\expandafter\ifx\csname urlprefix\endcsname\relax\def\urlprefix{URL }\fi
\expandafter\ifx\csname href\endcsname\relax
  \def\href#1#2{#2}\fi
\expandafter\ifx\csname burlalt\endcsname\relax
  \def\burlalt#1#2{\href{#2}{#1}}\fi

\bibitem[ABC{\etalchar{+}}24]{OurLectureNotes}
D.~Albritton, E.~Bru\'e, M.~Colombo, C.~De~Lellis, V.~Giri, M.~Janisch, and H.~Kwon.
\newblock {\em Instability and non-uniqueness for the 2{D} {E}uler equations, after {M}. {V}ishik}, volume 219 of {\em Annals of Mathematics Studies}.
\newblock Princeton University Press, Princeton, NJ, 2024.

\bibitem[Amb04]{Ambrosio04}
L.~Ambrosio.
\newblock Transport equation and {C}auchy problem for {$BV$} vector fields.
\newblock {\em Invent. Math.}, 158(2):227--260, 2004.
\newblock \doi{10.1007/s00222-004-0367-2}.

\bibitem[BC23]{BrueColombo23}
E.~Bru{\'e} and M.~Colombo.
\newblock Nonuniqueness of solutions to the {Euler} equations with vorticity in a {Lorentz} space.
\newblock {\em Commun. Math. Phys.}, 403(2):1171--1192, 2023.
\newblock \doi{10.1007/s00220-023-04816-4}.

\bibitem[BCDL21]{BrueColomboDeLellis21}
E.~Bru\'{e}, M.~Colombo, and C.~De~Lellis.
\newblock Positive solutions of transport equations and classical nonuniqueness of characteristic curves.
\newblock {\em Arch. Ration. Mech. Anal.}, 240(2):1055--1090, 2021.
\newblock \doi{10.1007/s00205-021-01628-5}.

\bibitem[BCK24]{BrueColomboKumar}
E.~Bru{\`e}, M.~Colombo, and A.~Kumar.
\newblock {Sharp Nonuniqueness in the Transport Equation with Sobolev Velocity Field}.
\newblock {\em arXiv preprint arXiv:2405.01670}, 2024.

\bibitem[BCV21]{BuckColomboVicol}
T.~Buckmaster, M.~Colombo, and V.~Vicol.
\newblock {Wild solutions of the Navier--Stokes equations whose singular sets in time have Hausdorff dimension strictly less than 1}.
\newblock {\em Journal of the European Mathematical Society}, 24(9):3333--3378, 2021.

\bibitem[BDLIS15]{BDLISZ15}
T.~Buckmaster, C.~De~Lellis, P.~Isett, and L.~Sz\'ekelyhidi, Jr.
\newblock Anomalous dissipation for {$1/5$}-{H}\"older {E}uler flows.
\newblock {\em Ann. of Math. (2)}, 182(1):127--172, 2015.
\newblock \doi{10.4007/annals.2015.182.1.3}.

\bibitem[BDLS16]{BDLSZ16}
T.~Buckmaster, C.~De~Lellis, and L.~Sz{{\'e}}kelyhidi, Jr.
\newblock Dissipative {E}uler flows with {O}nsager-critical spatial regularity.
\newblock {\em Comm. Pure Appl. Math.}, 69(9):1613--1670, 2016.
\newblock \doi{10.1002/cpa.21586}.

\bibitem[BDLSV19]{BDSV}
T.~Buckmaster, C.~De~Lellis, L.~j. Sz{\'e}kelyhidi, and V.~Vicol.
\newblock Onsager's conjecture for admissible weak solutions.
\newblock {\em Commun. Pure Appl. Math.}, 72(2):229--274, 2019.
\newblock \doi{10.1002/cpa.21781}.

\bibitem[BL15]{BL15}
J.~Bourgain and D.~Li.
\newblock Strong ill-posedness of the incompressible {Euler} equation in borderline {Sobolev} spaces.
\newblock {\em Invent. Math.}, 201(1):97--157, 2015.
\newblock \doi{10.1007/s00222-014-0548-6}.

\bibitem[BM24a]{BuckModena24II}
M.~Buck and S.~Modena.
\newblock Compactly supported anomalous weak solutions for 2d euler equations with vorticity in hardy spaces.
\newblock 2024, \burlalt{2405.19214}{http://arxiv.org/abs/2405.19214}.

\bibitem[BM24b]{BuckModena24}
M.~Buck and S.~Modena.
\newblock Non-uniqueness and energy dissipation for 2d euler equations with vorticity in hardy spaces.
\newblock {\em Journal of Mathematical Fluid Mechanics}, 26(26), 2024.
\newblock \doi{10.1007/s00021-024-00860-9}.

\bibitem[BMNV23]{BMNV23}
T.~Buckmaster, N.~Masmoudi, M.~Novack, and V.~Vicol.
\newblock {\em Intermittent convex integration for the 3{D} {E}uler equations}, volume 217 of {\em Annals of Mathematics Studies}.
\newblock Princeton University Press, Princeton, NJ, [2023] \copyright 2023.

\bibitem[BS21]{BrSh21}
A.~Bressan and W.~Shen.
\newblock A posteriori error estimates for self-similar solutions to the {E}uler equations.
\newblock {\em Discrete Contin. Dyn. Syst.}, 41(1):113--130, 2021.
\newblock \doi{10.3934/dcds.2020168}.

\bibitem[Buc15]{Buckmaster15}
T.~Buckmaster.
\newblock Onsager's conjecture almost everywhere in time.
\newblock {\em Comm. Math. Phys.}, 333(3):1175--1198, 2015.
\newblock \doi{10.1007/s00220-014-2262-z}.

\bibitem[BV19a]{BVsurvey19}
T.~Buckmaster and V.~Vicol.
\newblock Convex integration and phenomenologies in turbulence.
\newblock {\em EMS Surv. Math. Sci.}, 6(1-2):173--263, 2019.
\newblock \doi{10.4171/emss/34}.

\bibitem[BV19b]{BV19}
T.~Buckmaster and V.~Vicol.
\newblock Nonuniqueness of weak solutions to the {Navier}-{Stokes} equation.
\newblock {\em Ann. Math. (2)}, 189(1):101--144, 2019.
\newblock \doi{10.4007/annals.2019.189.1.3}.

\bibitem[BV21]{BVsurvey21}
T.~Buckmaster and V.~Vicol.
\newblock Convex integration constructions in hydrodynamics.
\newblock {\em Bull. Amer. Math. Soc. (N.S.)}, 58(1):1--44, 2021.
\newblock \doi{10.1090/bull/1713}.

\bibitem[CCFS08]{CCFS08}
A.~Cheskidov, P.~Constantin, S.~Friedlander, and R.~Shvydkoy.
\newblock Energy conservation and {Onsager}'s conjecture for the {Euler} equations.
\newblock {\em Nonlinearity}, 21(6):1233--1252, 2008.
\newblock \doi{10.1088/0951-7715/21/6/005}.

\bibitem[CCS21]{CiampaCrippaSpirito21}
G.~Ciampa, G.~Crippa, and S.~Spirito.
\newblock Strong convergence of the vorticity for the 2{D} {E}uler equations in the inviscid limit.
\newblock {\em Arch. Ration. Mech. Anal.}, 240(1):295--326, 2021.
\newblock \doi{10.1007/s00205-021-01612-z}.

\bibitem[CDE22]{ConstantinDrivasElgindi22}
P.~Constantin, T.~D. Drivas, and T.~M. Elgindi.
\newblock Inviscid limit of vorticity distributions in the {Y}udovich class.
\newblock {\em Comm. Pure Appl. Math.}, 75(1):60--82, 2022.
\newblock \doi{10.1002/cpa.21940}.

\bibitem[CFLS16]{CLL16}
A.~Cheskidov, M.~C.~L. Filho, H.~J.~N. Lopes, and R.~Shvydkoy.
\newblock Energy conservation in two-dimensional incompressible ideal fluids.
\newblock {\em Comm. Math. Phys.}, 348(1):129--143, 2016.
\newblock \doi{10.1007/s00220-016-2730-8}.

\bibitem[Cho13]{Choffrut13}
A.~Choffrut.
\newblock {{\(h\)}}-principles for the incompressible {Euler} equations.
\newblock {\em Arch. Ration. Mech. Anal.}, 210(1):133--163, 2013.
\newblock \doi{10.1007/s00205-013-0639-3}.

\bibitem[CL21]{CheskidovLuo3}
A.~Cheskidov and X.~Luo.
\newblock Nonuniqueness of weak solutions for the transport equation at critical space regularity.
\newblock {\em Ann. PDE}, 7(1):Paper No. 2, 45, 2021.
\newblock \doi{10.1007/s40818-020-00091-x}.

\bibitem[CL22]{cheskidov2022sharp}
A.~Cheskidov and X.~Luo.
\newblock Sharp nonuniqueness for the {N}avier-{S}tokes equations.
\newblock {\em Invent. Math.}, 229(3):987--1054, 2022.
\newblock \doi{10.1007/s00222-022-01116-x}.

\bibitem[CL23]{CheskidovLuo2}
A.~Cheskidov and X.~Luo.
\newblock {$L^2$}-critical nonuniqueness for the 2{D} {N}avier-{S}tokes equations.
\newblock {\em Ann. PDE}, 9(2):Paper No. 13, 56, 2023.
\newblock \doi{10.1007/s40818-023-00154-9}.

\bibitem[CLJ12]{CDS12}
A.~Choffrut, C.~D. Lellis, and L.~S. Jr.
\newblock Dissipative continuous euler flows in two and three dimensions, 2012, \burlalt{1205.1226}{http://arxiv.org/abs/1205.1226}.

\bibitem[CMZO24]{Cordoba2024}
D.~C{\'o}rdoba, L.~Mart{\'i}nez-Zoroa, and W.~O{\.z}a{\'n}ski.
\newblock Instantaneous gap loss of sobolev regularity for the 2d incompressible euler equations.
\newblock {\em Duke Math. J.}, 2024, \burlalt{arXiv:2210.17458}{http://arxiv.org/abs/arXiv:2210.17458}.
\newblock to appear.

\bibitem[CS]{CS21}
G.~Crippa and G.~Stefani.
\newblock An elementary proof of existence and uniqueness for the euler flow in localized yudovich space.
\newblock {\em To appear on Calc. Var. Partial Differential Equations}.

\bibitem[CS15]{CrippaSpirito15}
G.~Crippa and S.~Spirito.
\newblock Renormalized solutions of the 2{D} {E}uler equations.
\newblock {\em Comm. Math. Phys.}, 339(1):191--198, 2015.
\newblock \doi{10.1007/s00220-015-2411-z}.

\bibitem[Del91]{Delort91}
J.-M. Delort.
\newblock Existence de nappes de tourbillon en dimension deux.
\newblock {\em J. Amer. Math. Soc.}, 4(3):553--586, 1991.
\newblock \doi{10.2307/2939269}.

\bibitem[DL89]{DiPernaLions}
R.~J. DiPerna and P.-L. Lions.
\newblock Ordinary differential equations, transport theory and {S}obolev spaces.
\newblock {\em Invent. Math.}, 98(3):511--547, 1989.
\newblock \doi{10.1007/BF01393835}.

\bibitem[DLS09]{DeLellisSzekelyhidi09}
C.~De~Lellis and L.~Sz{\'e}kelyhidi, Jr.
\newblock The {E}uler equations as a differential inclusion.
\newblock {\em Ann. of Math. (2)}, 170(3):1417--1436, 2009.
\newblock \doi{10.4007/annals.2009.170.1417}.

\bibitem[DLS13]{DeLellisSzekelyhidi13}
C.~De~Lellis and L.~Sz{{\'e}}kelyhidi, Jr.
\newblock Dissipative continuous {E}uler flows.
\newblock {\em Invent. Math.}, 193(2):377--407, 2013.
\newblock \doi{10.1007/s00222-012-0429-9}.

\bibitem[DLS17]{DSsurvey17}
C.~De~Lellis and L.~Sz\'{e}kelyhidi, Jr.
\newblock High dimensionality and h-principle in {PDE}.
\newblock {\em Bull. Amer. Math. Soc. (N.S.)}, 54(2):247--282, 2017.
\newblock \doi{10.1090/bull/1549}.

\bibitem[DLS19]{DSsurvey19}
C.~De~Lellis and L.~Sz\'ekelyhidi, Jr.
\newblock On turbulence and geometry: from {N}ash to {O}nsager.
\newblock {\em Notices Amer. Math. Soc.}, 66(5):677--685, 2019.

\bibitem[DLS22]{DSsurvey22}
C.~De~Lellis and L.~Sz\'{e}kelyhidi, Jr.
\newblock Weak stability and closure in turbulence.
\newblock {\em Philos. Trans. Roy. Soc. A}, 380(2218):Paper No. 20210091, 16, 2022.
\newblock \doi{10.1098/rsta.2021.0091}.

\bibitem[DM87]{DPM87}
R.~J. DiPerna and A.~J. Majda.
\newblock Concentrations in regularizations for {$2$}-{D} incompressible flow.
\newblock {\em Comm. Pure Appl. Math.}, 40(3):301--345, 1987.
\newblock \doi{10.1002/cpa.3160400304}.

\bibitem[DS17]{DSZ17}
S.~Daneri and L.~Sz{{\'e}}kelyhidi, Jr.
\newblock Non-uniqueness and h-principle for {H}\"older-continuous weak solutions of the {E}uler equations.
\newblock {\em Arch. Rational Mech. Anal.}, 224(2):471--514, 2017.

\bibitem[Elg21]{ElgindiAnnals}
T.~Elgindi.
\newblock Finite-time singularity formation for {{\(C^{1,\alpha}\)}} solutions to the incompressible euler equations on {{\(\mathbb{R}^3\)}}.
\newblock {\em Ann. Math. (2)}, 194(3):647--727, 2021.
\newblock \doi{10.4007/annals.2021.194.3.2}.

\bibitem[EM94]{EvansMuller94}
L.~C. Evans and S.~M\"{u}ller.
\newblock Hardy spaces and the two-dimensional {E}uler equations with nonnegative vorticity.
\newblock {\em J. Amer. Math. Soc.}, 7(1):199--219, 1994.
\newblock \doi{10.2307/2152727}.

\bibitem[EM20]{EM20}
T.~M. Elgindi and N.~Masmoudi.
\newblock {{\(L^\infty\)}} ill-posedness for a class of equations arising in hydrodynamics.
\newblock {\em Arch. Ration. Mech. Anal.}, 235(3):1979--2025, 2020.
\newblock \doi{10.1007/s00205-019-01457-7}.

\bibitem[Eyi01]{Eyink01}
G.~L. Eyink.
\newblock Dissipation in turbulent solutions of 2d {Euler} equations.
\newblock {\em Nonlinearity}, 14(4):787--802, 2001.
\newblock \doi{10.1088/0951-7715/14/4/307}.

\bibitem[Gal11]{galdi2011introduction}
G.~Galdi.
\newblock {\em {An introduction to the mathematical theory of the Navier-Stokes equations: Steady-state problems}}.
\newblock Springer Science \& Business Media, 2011.

\bibitem[GP22]{GP22}
F.~Grotto and U.~Pappalettera.
\newblock Correction to: {B}urst of point vortices and non-uniqueness of 2{D} {E}uler equations.
\newblock {\em Arch. Ration. Mech. Anal.}, 246(1):139--140, 2022.
\newblock \doi{10.1007/s00205-022-01814-z}.

\bibitem[GR24]{GiriRadu23}
V.~Giri and R.-O. Radu.
\newblock The 2d {O}nsager conjecture: a {N}ewton--{N}ash iteration.
\newblock 2024, \burlalt{2405.19214}{http://arxiv.org/abs/2405.19214}.

\bibitem[H\"33]{Holder33}
E.~H\"{o}lder.
\newblock \"{U}ber die unbeschr\"{a}nkte {F}ortsetzbarkeit einer stetigen ebenen {B}ewegung in einer unbegrenzten inkompressiblen {F}l\"{u}ssigkeit.
\newblock {\em Math. Z.}, 37(1):727--738, 1933.
\newblock \doi{10.1007/BF01474611}.

\bibitem[Ise18]{Isett2018Annals}
P.~Isett.
\newblock A proof of {O}nsager's conjecture.
\newblock {\em Ann. of Math. (2)}, 188(3):871--963, 2018.
\newblock \doi{10.4007/annals.2018.188.3.4}.

\bibitem[Ise22]{Isett22}
P.~Isett.
\newblock Nonuniqueness and existence of continuous, globally dissipative {Euler} flows.
\newblock {\em Arch. Ration. Mech. Anal.}, 244(3):1223--1309, 2022.
\newblock \doi{10.1007/s00205-022-01780-6}.

\bibitem[KGN23]{Kwon2023}
H.~Kwon, V.~Giri, and M.~Novack.
\newblock A wavelet-inspired \(l^3\)-based convex integration framework for the euler equations.
\newblock arXiv:2305.18142, 2023.
\newblock \urlprefix\url{https://arxiv.org/abs/2305.18142}.

\bibitem[Lam24]{lamb1924hydrodynamics}
H.~Lamb.
\newblock {\em Hydrodynamics}.
\newblock University Press, 1924.

\bibitem[LFMNL06]{MMN06}
M.~C. Lopes~Filho, A.~L. Mazzucato, and H.~J. Nussenzveig~Lopes.
\newblock Weak solutions, renormalized solutions and enstrophy defects in 2{D} turbulence.
\newblock {\em Arch. Ration. Mech. Anal.}, 179(3):353--387, 2006.
\newblock \doi{10.1007/s00205-005-0390-5}.

\bibitem[LMPP21]{LMP21}
S.~Lanthaler, S.~Mishra, and C.~Par{\'e}s-Pulido.
\newblock On the conservation of energy in two-dimensional incompressible flows.
\newblock {\em Nonlinearity}, 34(2):1084--1135, 2021.
\newblock \doi{10.1088/1361-6544/abb452}.

\bibitem[Loe06]{Loeper06}
G.~Loeper.
\newblock Uniqueness of the solution to the {V}lasov-{P}oisson system with bounded density.
\newblock {\em J. Math. Pures Appl. (9)}, 86(1):68--79, 2006.
\newblock \doi{10.1016/j.matpur.2006.01.005}.

\bibitem[Maj93]{Majda93}
A.~J. Majda.
\newblock Remarks on weak solutions for vortex sheets with a distinguished sign.
\newblock {\em Indiana Univ. Math. J.}, 42(3):921--939, 1993.
\newblock \doi{10.1512/iumj.1993.42.42043}.

\bibitem[Mas07]{Masmudi07}
N.~Masmoudi.
\newblock Remarks about the inviscid limit of the {N}avier-{S}tokes system.
\newblock {\em Comm. Math. Phys.}, 270(3):777--788, 2007.
\newblock \doi{10.1007/s00220-006-0171-5}.

\bibitem[MB02]{MajdaBertozzi02}
A.~J. Majda and A.~L. Bertozzi.
\newblock {\em Vorticity and incompressible flow}, volume~27 of {\em Cambridge Texts in Applied Mathematics}.
\newblock Cambridge University Press, Cambridge, 2002.

\bibitem[Men23]{Mengual23}
F.~Mengual.
\newblock Non-uniqueness of admissible solutions for the 2{D} {E}uler equation with \(l^p\) vortex data.
\newblock {\em arXiv preprint}, 2023, \burlalt{2304.09578}{http://arxiv.org/abs/2304.09578}.
\newblock \urlprefix\url{https://doi.org/10.48550/arXiv.2304.09578}.

\bibitem[MP94]{MarchioroPulvirenti94}
C.~Marchioro and M.~Pulvirenti.
\newblock {\em Mathematical theory of incompressible nonviscous fluids}, volume~96 of {\em Applied Mathematical Sciences}.
\newblock Springer-Verlag, New York, 1994.
\newblock \doi{10.1007/978-1-4612-4284-0}.

\bibitem[MS18]{MS18}
S.~Modena and L.~Sz\'{e}kelyhidi, Jr.
\newblock Non-uniqueness for the transport equation with {S}obolev vector fields.
\newblock {\em Ann. PDE}, 4(2):Paper No. 18, 38, 2018.
\newblock \doi{10.1007/s40818-018-0056-x}.

\bibitem[Mv03]{MullerSverak03}
S.~M\"{u}ller and V.~\v{S}ver\'{a}k.
\newblock Convex integration for {L}ipschitz mappings and counterexamples to regularity.
\newblock {\em Ann. of Math. (2)}, 157(3):715--742, 2003.
\newblock \doi{10.4007/annals.2003.157.715}.

\bibitem[MVH94]{meleshko1994chaplygin}
V.~Meleshko and G.~Van~Heijst.
\newblock {On Chaplygin's investigations of two-dimensional vortex structures in an inviscid fluid}.
\newblock {\em Journal of Fluid Mechanics}, 272:157--182, 1994.

\bibitem[Nas54]{Nash}
J.~Nash.
\newblock {$C^1$} isometric imbeddings.
\newblock {\em Ann. of Math. (2)}, 60:383--396, 1954.
\newblock \doi{10.2307/1969840}.

\bibitem[NV23]{NovackVicol23}
M.~Novack and V.~Vicol.
\newblock An intermittent onsager theorem.
\newblock {\em Inventiones Mathematicae}, 233:223--323, 2023.
\newblock \doi{10.1007/s00222-023-01185-6}.

\bibitem[RP24]{derosaPark24}
L.~D. Rosa and J.~Park.
\newblock No anomalous dissipation in two-dimensional incompressible fluids, 2024, \burlalt{2403.04668}{http://arxiv.org/abs/2403.04668}.

\bibitem[Sch93]{Scheffer93}
V.~Scheffer.
\newblock An inviscid flow with compact support in space-time.
\newblock {\em J. Geom. Anal.}, 3(4):343--401, 1993.
\newblock \doi{10.1007/BF02921318}.

\bibitem[Shn00]{Shnirelman00}
A.~Shnirelman.
\newblock Weak solutions with decreasing energy of incompressible {E}uler equations.
\newblock {\em Comm. Math. Phys.}, 210(3):541--603, 2000.
\newblock \doi{10.1007/s002200050791}.

\bibitem[Visa]{Vis18a}
M.~Vishik.
\newblock Instability and non-uniqueness in the cauchy problem for the euler equations of an ideal incompressible fluid. part i.
\newblock {\em arxiv:1805.09426}, \burlalt{1805.09426}{http://arxiv.org/abs/1805.09426}.
\newblock \urlprefix\url{https://arxiv.org/pdf/1805.09426.pdf}.

\bibitem[Visb]{Vis18}
M.~Vishik.
\newblock Instability and non-uniqueness in the cauchy problem for the euler equations of an ideal incompressible fluid. part ii.
\newblock {\em arxiv:1805.09440}, \burlalt{1805.09440}{http://arxiv.org/abs/1805.09440}.
\newblock \urlprefix\url{https://arxiv.org/pdf/1805.09440.pdf}.

\bibitem[Vis99]{Vishik99}
M.~Vishik.
\newblock Incompressible flows of an ideal fluid with vorticity in borderline spaces of {B}esov type.
\newblock {\em Ann. Sci. \'{E}cole Norm. Sup. (4)}, 32(6):769--812, 1999.
\newblock \doi{10.1016/S0012-9593(00)87718-6}.

\bibitem[VW93]{VecchiWu}
I.~Vecchi and S.~J. Wu.
\newblock On {$L^1$}-vorticity for {$2$}-{D} incompressible flow.
\newblock {\em Manuscripta Math.}, 78(4):403--412, 1993.
\newblock \doi{10.1007/BF02599322}.

\bibitem[Wol33]{Wolibner33}
W.~Wolibner.
\newblock Un theor\`eme sur l'existence du mouvement plan d'un fluide parfait, homog\`ene, incompressible, pendant un temps infiniment long.
\newblock {\em Math. Z.}, 37(1):698--726, 1933.
\newblock \doi{10.1007/BF01474610}.

\bibitem[Yud62]{Yud62}
V.~I. Yudovich.
\newblock Some bounds for solutions of elliptic equations.
\newblock {\em Mat. Sb. (N.S.)}, 59 (101)(suppl.):229--244, 1962.
\newblock \urlprefix\url{https://mathscinet.ams.org/mathscinet-getitem?mr=0149062}.

\bibitem[Yud63]{Yud63}
V.~I. Yudovich.
\newblock Non-stationary flows of an ideal incompressible fluid.
\newblock {\em \v Z. Vy\v cisl. Mat i Mat. Fiz.}, 3:1032--1066, 1963.

\end{thebibliography}

\end{document}